\documentclass[11pt]{article}
\usepackage{float, graphicx}
\usepackage{amsmath, amsthm, amssymb}
\usepackage{verbatim}
\usepackage{multicol}
\usepackage{url}
\usepackage{graphicx}
\usepackage{amsmath}
\usepackage{enumerate}
\usepackage{bm}
\usepackage[noadjust]{cite}
\usepackage{color}
\usepackage{mathtools}
\definecolor{vdarkred}{rgb}{0.6,0,0.2}
\definecolor{vdarkblue}{rgb}{0,0.2,0.6}
\usepackage[colorlinks, linkcolor=vdarkblue,citecolor=vdarkred]{hyperref}

\usepackage[letterpaper, hmargin=1in, top=1in, bottom=1.2in, footskip=0.6in]{geometry}

\usepackage{titlesec}
\titleformat{\section}[block]{\filcenter\normalfont\bfseries\large}{\thesection.}{.5em}{}\titlespacing*{\section}{0pt}{2\baselineskip}{1\baselineskip}
\titleformat{\subsection}[runin]{\normalfont\bfseries}{\thesubsection.}{.4em}{}[.]\titlespacing{\subsection}{0pt}{2ex plus .1ex minus .2ex}{.8em}
\titleformat{\subsubsection}[runin]{\normalfont\itshape}{\thesubsubsection.}{.3em}{}[.]\titlespacing{\subsubsection}{0pt}{1ex plus .1ex minus .2ex}{.5em}
\titleformat{\paragraph}[runin]{\normalfont\itshape}{\theparagraph.}{.3em}{}[.]\titlespacing{\paragraph}{0pt}{1ex plus .1ex minus .2ex}{.5em}

\usepackage[T1]{fontenc}
\usepackage[utf8]{inputenc}
\usepackage{lmodern}

\let\originalleft\left
\let\originalright\right
\renewcommand{\left}{\mathopen{}\mathclose\bgroup\originalleft}
\renewcommand{\right}{\aftergroup\egroup\originalright}

\numberwithin{equation}{section}

\renewcommand{\cal}{\mathcal}

\newcommand{\cC}{{\cal C}}

\newcommand{\cG}{{\cal G}}

\newcommand{\cN}{{\cal N}}

\newcommand{\cR}{{\mathcal R}}

\newcommand{\fa}{{\frak a}}
\newcommand{\fb}{{\frak b}}
\newcommand{\fc}{{\frak c}}
\newcommand{\fd}{{\frak d}}
\newcommand{\fe}{{\frak e}}

\newcommand{\fo}{{\frak o}}
\newcommand{\fK}{{\frak K}}

\newcommand{\rd}{{\rm d}}

\newcommand{\ri}{\mathrm{i}}

\newcommand{\bC}{{\mathbb C}}
\newcommand{\bE}{\mathbb{E}}
\newcommand{\bN}{\mathbb{N}}
\newcommand{\bP}{\mathbb{P}}

\newcommand{\bR}{{\mathbb R}}

\newcommand{\la}{\lambda}

\DeclareMathOperator{\Tr}{Tr}

\DeclareMathOperator{\dist}{dist}

\DeclareMathOperator{\OO}{O}
\DeclareMathOperator{\oo}{o}

\renewcommand{\Re}{\mathop{\mathrm{Re}}}
\renewcommand{\Im}{\mathop{\mathrm{Im}}}

\renewcommand{\leq}{\leqslant}
\renewcommand{\geq}{\geqslant}

\definecolor{darkred}{rgb}{0.9,0,0.3}
\definecolor{darkblue}{rgb}{0,0.3,0.9}
\definecolor{purple}{rgb}{0.7,0,0.6}
\definecolor{darkyellow}{rgb}{0.8,0.8,0}

\newcommand{\nc}{\normalcolor}

\newcommand{\del}{\partial}

\newcommand{\wh}{\widehat}

\newcommand{\beq}{\begin{equation}}
\newcommand{\bEq}{\end{equation}}

\newcommand{\qq}[1]{[\![#1]\!]}

\newcommand{\pb}[1]{\bigl(#1\bigr)}
\newcommand{\pB}[1]{\Bigl(#1\Bigr)}
\newcommand{\pbb}[1]{\biggl(#1\biggr)}
\newcommand{\pBB}[1]{\Biggl(#1\Biggr)}
\newcommand{\pa}[1]{\left(#1\right)}

\newcommand{\qb}[1]{\bigl[#1\bigr]}

\newcommand{\qbb}[1]{\biggl[#1\biggr]}
\newcommand{\qBB}[1]{\Biggl[#1\Biggr]}
\newcommand{\qa}[1]{\left[#1\right]}

\newcommand{\h}[1]{\{#1\}}

\newcommand{\abs}[1]{\lvert #1 \rvert}
\newcommand{\absb}[1]{\bigl\lvert #1 \bigr\rvert}

\newcommand{\absbb}[1]{\biggl\lvert #1 \biggr\rvert}

\newcommand{\scalar}[2]{\langle#1 \mspace{2mu}, #2\rangle}

\theoremstyle{plain} 
\newtheorem{theorem}{Theorem}[section]
\newtheorem*{theorem*}{Theorem}
\newtheorem{lemma}[theorem]{Lemma}
\newtheorem*{lemma*}{Lemma}
\newtheorem{corollary}[theorem]{Corollary}
\newtheorem*{corollary*}{Corollary}
\newtheorem{proposition}[theorem]{Proposition}
\newtheorem*{proposition*}{Proposition}

\newtheorem*{conjecture*}{Conjecture}
\newtheorem{claim}[theorem]{Claim}

\theoremstyle{definition} 
\newtheorem{definition}[theorem]{Definition}
\newtheorem*{definition*}{Definition}

\newtheorem*{example*}{Example}
\newtheorem{remark}[theorem]{Remark}
\newtheorem*{remark*}{Remark}
\newtheorem{assumption}[theorem]{Assumption}
\newtheorem*{assumption*}{Assumption}

\newcommand{\bld}[1]{\boldsymbol{\mathrm{#1}}} 
\renewcommand{\cal}{\mathcal}

\newcommand{\txt}[1]{\text{\rm{#1}}}

\renewcommand{\P}{\mathbb{P}}
\newcommand{\E}{\mathbb{E}}
\newcommand{\R}{\mathbb{R}}
\newcommand{\C}{\mathbb{C}}
\newcommand{\N}{\mathbb{N}}
\newcommand{\Z}{\mathbb{Z}}

\newcommand{\ee}{\mathrm{e}}

\newcommand{\dd}{\mathrm{d}}
\newcommand{\col}{\vcentcolon}
\newcommand*{\deq}{\mathrel{\vcenter{\baselineskip0.65ex \lineskiplimit0pt \hbox{.}\hbox{.}}}=}

\renewcommand{\leq}{\leqslant}
\renewcommand{\geq}{\geqslant}
\renewcommand{\epsilon}{\varepsilon}

\def\author#1{\par
    {\centering{\authorfont#1}\par\vspace*{0.05in}}
}

\def\titlefont{\fontsize{15}{17}\bfseries\boldmath\selectfont\centering{}}
\def\authorfont{\fontsize{13}{15}}

\let\affiliationfont\rhfont

\def\address#1{\par
    {\centering{\affiliationfont#1\par}}\par\vspace*{11pt}
}

\def\body{
\setcounter{footnote}{0}
\def\thefootnote{\alph{footnote}}
\def\@makefnmark{{$^{\rm \@thefnmark}$}}
}

\def\title#1{
    \thispagestyle{plain}
    \vspace*{-14pt}
    \vskip 79pt
    {\centering{\titlefont #1\par}}%
    \vskip 1em
}

\newcommand{\rhosc}{\rho_{\mathrm{sc}}}

\newcommand{\xid}{\xi_{d}}

\renewcommand{\i}{{\rm{i}}}
\newcommand{\msc}{m_{\rm sc}}
\newcommand{\md}{m_d}

\newcommand{\Lambdao}{\Lambda_{\rm o}}
\newcommand{\Lambdad}{\Lambda_{\rm d}}

\newcommand{\cDe}{\mathbf D_{\rm e}}

\begin{document}

\title{Edge rigidity and universality of  random regular graphs \\ of  intermediate degree}

\vspace{1.2cm}

\noindent \begin{minipage}[c]{0.5\textwidth}
 \author{Roland Bauerschmidt}
\address{University of Cambridge\\
   E-mail:  rb812@cam.ac.uk}
 \end{minipage}
 \begin{minipage}[c]{0.5\textwidth}
 \author{Jiaoyang Huang}
\address{IAS\\
   E-mail: jiaoyang@ias.edu}
 \end{minipage}
 
 \noindent \begin{minipage}[c]{0.5\textwidth}
 \author{Antti Knowles}
\address{University of Geneva \\
E-mail: antti.knowles@unige.ch }
 \end{minipage}
 \begin{minipage}[c]{0.5\textwidth}
 \author{Horng-Tzer Yau}
\address{Harvard University \\
E-mail: htyau@math.harvard.edu}
 \end{minipage}

 \begin{center}
{\large June 11, 2020}
\end{center}

\vspace{1em}
 
\begin{abstract}
  For random $d$-regular graphs on $N$ vertices with $1 \ll d \ll N^{2/3}$,
  we develop a $d^{-1/2}$ expansion of the local eigenvalue distribution about the Kesten--McKay law up to order $d^{-3}$.
  This result is valid up to the edge of the spectrum. 
  It implies that the eigenvalues of such random regular graphs are more rigid than those of Erd\H{o}s--R\'enyi graphs of the same average degree.
  As a first application, for $1 \ll d \ll N^{2/3}$, we show that all nontrivial eigenvalues of the adjacency matrix are with very high probability bounded in absolute value by $(2 + \oo(1)) \sqrt{d - 1}$. As a second application, for $N^{2/9} \ll d \ll N^{1/3}$,
  we prove that the extremal eigenvalues are concentrated at scale $N^{-2/3}$ and their fluctuations are governed by Tracy--Widom statistics. Thus, in the same regime of $d$, $52\%$ of all $d$-regular graphs have second-largest eigenvalue strictly less than $2 \sqrt{d - 1}$.
\end{abstract}

\section{Introduction}

\subsection{Main results}

Let $\bP$ be the uniform probability measure on the set of $d$-regular graphs on $N$ vertices. We identify a graph with its adjacency matrix $A = (A_{ij}) \in \{0,1\}^{N \times N}$, defined as $A_{ij} = 1$ if and only if $i$ and $j$ are adjacent. Thus, $\bP$ is the uniform probability measure on the set of Hermitian matrices $A \in \{0,1\}^{N \times N}$ satisfying $\sum_{j = 1}^N A_{ij} = d$  and $A_{ii} = 0$  for all $i = 1,\dots, N$.

Since $A$ is $d$-regular, it is immediate that $A$ has a trivial eigenvalue $d$ with associated eigenvector $(1,1,\dots,1)^*$. Moreover, by the Perron-Frobenius theorem, all other eigenvalues are bounded in absolute value by $d$. For convenience, we shall consider the normalized adjacency matrix
\begin{equation} \label{def_H}
H\deq(d-1)^{-1/2}A.
\end{equation}
We denote its eigenvalues by $\lambda_1 = d/\sqrt{d-1} \geq \lambda_2 \geq \cdots \geq \lambda_N \geq -d/\sqrt{d-1}$.

Unless stated otherwise, all quantities depend on the fundamental parameter $N$, and we omit this dependence from our notation. For the following statements, for deterministic $N$-dependent quantities  $X$ and $Y$ \nc we write
\begin{equation}
 X \ll Y \quad \txt{if} \quad X =  \OO_\fc(N^{-\fc} Y)  \nc \txt{ for some fixed $\fc>0$.}
\end{equation}
(See the conventions in Section \ref{sec:structure} below.) \nc

Our first main result is about the locations of the nontrivial extremal eigenvalues, $\lambda_2$ and $\lambda_N$.

\begin{theorem}\label{t:eigloc}
Fix $\fc > 0$. For $1\ll d\ll N^{2/3}$ and large enough $N$,
with probability $1-N^{-1/\fc}$ we have
\begin{align}\label{e:egbound}
|\la_2-2|,\; |\la_N+2|\leq N^\fc\left( 
\frac{1}{d^3}+\frac{1}{N^{2/3}}+\frac{d^2}{ N^{4/3}}\right).
\end{align}
\end{theorem}

An immediate consequence is the following optimal upper bound on the nontrivial eigenvalues.
It was conjectured for instance in \cite[Conjecture~5.3]{MR2432537} and \cite[Conjecture~7.3]{MR3727622}.

\begin{corollary} \label{cor:bounds}
For $1 \ll d \ll N^{2/3}$, all nontrivial eigenvalues of the random $d$-regular graph are with very high probability bounded in absolute value by $(2 + \oo(1)) \sqrt{d - 1}$.
\end{corollary}

Our second main result is about the limiting distribution of the extremal eigenvalues.
\begin{theorem} \label{thm:univ}
For $N^{2/9} \ll d \ll N^{1/3}$ the distribution of $N^{2/3}(\lambda_2-2)$ converges to the Tracy--Widom$_1$ distribution,
the limiting distribution of the largest eigenvalue of a GOE matrix. The analogous statement holds for $-N^{2/3} (\lambda_N + 2)$.
\end{theorem}

Universality for the edge statistics of Wigner matrices (the statement that the distribution of the extremal eigenvalues converge to the Tracy--Widom law)   was first established by the moment method \cite{MR1727234} under certain symmetry assumptions on the distribution of the matrix elements. The moment method was further developed in \cite{MR2475670,MR2647136} and \cite{MR2726110}. A different approach to edge universality for Wigner matrices  based on the direct comparison with corresponding Gaussian ensembles was developed in \cite{MR2669449,MR2871147}.  Edge universality for sparse Erd{\H{o}}s--R{\'e}nyi graphs was proven first in the regime $pN \gg N^{2/3}$ in the works \cite{MR3098073,MR2964770} and then extended to the regime $pN \gg N^{1/3}$ in \cite{MR3800840}. For smaller values of the average degree $pN$, edge universality no longer holds: it was proved in  \cite{HLY,HK20} that, in the regime $1 \ll pN \ll N^{1/3}$, \nc the second-largest eigenvalue has Gaussian fluctuations instead of Tracy--Widom fluctuations. These Gaussian fluctuations result from degree fluctuations which are absent in regular graphs. Our Theorem \ref{thm:univ} implies that  the eigenvalues of random regular graphs are indeed more rigid than those of Erd\H{o}s--R\'enyi graphs of the same average degree.
For random regular graphs,
  it is expected that \eqref{e:egbound} is not optimal for small $d$, and in 
  fact it is conjectured that
  the extremal eigenvalues continue to have Tracy--Widom fluctuations
  down to degree $d\geq 3$.

As emphasized in \cite{MR2072849},
the Tracy--Widom$_1$ distribution has positive measure on the set $\{x: x<0\}$;
in fact it has about $52\%$ of its mass on negative values.
Therefore Theorem~\ref{thm:univ} implies the existence of many
$d$-regular graphs whose second eigenvalue is \emph{less than} $2\sqrt{d-1}$,
provided that $N$ and $d$ obey the conditions of Theorem~\ref{thm:univ}.

\begin{corollary}\label{c:ramanujan_graph}
  For $d$ large enough and $d^3\ll N\ll d^{9/2}$, $52\%$ of $d$-regular graphs on $N$ vertices have second-largest eigenvalue bounded by $2\sqrt{d-1}$.
  An analogous statement holds for the smallest eigenvalue.
\end{corollary}

In \cite{MR963118}, a $d$-regular graph whose largest nontrivial eigenvalue is bounded in absolute value by $2\sqrt{d-1}$ is
called a Ramanujan graph.
Corollary~\ref{c:ramanujan_graph} states that for $d$ large enough and $d^3\ll N\ll d^{9/2}$,
precisely $52\%$ of $d$-regular graphs have largest (respectively smallest) nontrivial eigenvalue bounded from above by $2\sqrt{d-1}$
(respectively from below by $-2\sqrt{d-1}$).
Hence, at least $4\%$ have all nontrivial eigenvalues bounded by $2\sqrt{d-1}$ in absolute value.
In fact, Theorem~\ref{thm:univ} and its proof can be extended to show 
that in the regime $N^{2/9}\ll d\ll N^{1/3}$ the largest and smallest nontrivial eigenvalues converge in distribution to \emph{independent} Tracy--Widom$_1$ distributions; see Remark~\ref{rem:independence} and Theorem~\ref{thm:univ_gen} below.
As a consequence, we have the following result.

\begin{corollary}
For $d$ large enough and $d^3\ll N\ll d^{9/2}$, $27\%$ of $d$-regular graphs on $N$ vertices have all nontrivial eigenvalues bounded in absolute value by $2\sqrt{d-1}$.
\end{corollary}

The conjecture \cite{MR2072849,MR2433888}
that a positive fraction of regular graphs of fixed $d$ is Ramanujan remains open.
Explicit constructions of Ramanujan graphs with $d=p+1$ for some prime and prime powers $p$ were introduced in \cite{MR963118,MR939574} (see also \cite{MR2072849});
a construction that applies in the bipartite case for all degrees is given in \cite{MR3374962,MR3892446}
and see \cite{MR3630988} for polynomial time algorithm of this construction.

The following results on the extremal eigenvalues of random regular graphs are known.
For fixed degree $d \geq 3$, Friedman \cite{MR2437174} proved that
$|\la_2-2| + |\la_N+2| = \oo(1)$ with high probability.
This result was recently reproved using an alternative method in \cite{Bord15}; see also \cite{MR3385636}.
For $d$ tending to infinity with the number of vertices,
it was proved in \cite{FriedmanKahnSzemeredi,MR3078290}
that the nontrivial extremal eigenvalues are $\OO(\sqrt{d})$ for the permutation model of random regular graphs.
Recently, in \cite{MR3758727,MR3909972} the bound $\OO(\sqrt{d})$ was established with high probability for the uniform model of random regular graphs for all $1 \ll d \leq N/2$. Previous results in this direction include \cite{Krivelevich2001}.

\subsection{Related results}

In random matrix theory, the bulk spectral statistics of Wigner matrices are  well
understood; see in particular
\cite{MR1810949,MR2810797,MR2639734,MR2662426,MR2981427,MR2784665,MR3372074,MR3541852,MR2964770,MR3699468}.
For Erd\H{o}s--R\'enyi random graphs and random regular graphs with growing average degrees, the bulk spectral statistics were analysed in \cite{MR3098073,MR2964770,MR3429490,1510.06390, MR3729611}, and complete eigenvector delocalization for logarithmically growing average degree was proved in \cite{MR2964770,HKM19,MR3688032}. In the same regime, edge rigidity of Erd\H{o}s--R\'enyi graphs was proved in \cite{benaych2017spectral,alt2019extremal}.
Similar results have also been proved for more general degree distributions \cite{1509.03368}.
These types of results are false for Erd\H{o}s--R\'enyi graphs with bounded average degree,
whereas random regular graphs are expected to have random matrix statistics even for bounded degree graphs;
see \cite{MR3962004} for the proof of complete eigenvector delocalization in this regime.
For a review of other results for discrete random matrices, see also \cite{MR2432537}.

Macroscopic eigenvalue statistics for random regular graphs of fixed degree have been studied
using the techniques of Poisson approximation of short cycles \cite{MR3078290,MR3315475} and (non-rigorously) using the replica method \cite{PhysRevE.90.052109}.
These results show that the macroscopic eigenvalue statistics for random regular graphs of fixed degree
are different from those of a Gaussian matrix.
However, this is not predicted to be the case for the \emph{local} eigenvalue statistics.
Spectral properties of regular directed graphs have also been studied recently \cite{1508.00208,Cook2015}.

For the eigenvectors of random regular graphs with $d \in [N^\fc, N^{2/3-\fc}]$,
the asymptotic normality was proved in \cite{MR3690289};
see also the prior results for Wigner matrices \cite{MR3034787,MR2930379,BY2016}.
For random regular graphs of fixed degree,
a Gaussian wave correlation structure for the eigenvectors was predicted in \cite{0907.5065}
and partially confirmed in \cite{MR3945757}.

In the non-Hermitian setting, the limit of the empirical eigenvalue distribution of random matrices with i.i.d.\ entries is governed by the circular law. The circular law for non-Hermitian random matrices with i.i.d.\ entries with certain moment conditions was verified in \cite{MR1437734, MR2575411, MR2663633},  and  the paper \cite{MR2722794} established the circular law under the weakest moment condition.  In the directed $d$-regular graph setting, it was conjectured \cite{MR2908617} that for any fixed degree $d$, the empirical eigenvalue distribution converges to the oriented Kesten--McKay distribution,
\begin{align*}
\frac{1}{\pi}\frac{d^2(d-1)}{(d^2-(x^2+y^2))^2}{\bm 1}_{|z|\leq \sqrt{d}} \, \rd x \, \rd y.
\end{align*}
Up to rescaling by $\sqrt{d}$, this measure tends to the circular law as $d$ tends to infinity. In the regime that $d$ grows with the size of the graph, the circular law for the directed $d$-regular was established in  \cite{Nik1,MR3798243,LLTTY}.

\subsection{Notation and structure of paper}
\label{sec:structure}
We use the convention $\N = \{0,1,2,\dots\}$.
We usually omit the argument $N$ from our notation, with the convention that most quantities are allowed to depend on $N$ and all of our estimates are uniform in them. Estimates are not uniform in quantities that are explicitly constant or fixed.
 By definition, a random variable $F$ is a function $F(A)$ of the adjacency matrix $A$. We shall often omit the explicit argument $A$ from our notation, and simply write $F$ for a random variable evaluated at $A$.
For $n \in \bN$ we use the notation $\qq{n} \deq \{1,2,\dots, n\}$.
We use the usual big O notation $\OO(\cdot)$, and if the implicit constant depends on a parameter $\alpha$ we indicate it by writing $\OO_\alpha(\cdot)$. We use the letter $\fc$ to denote a generic small positive constant.
For $N$-dependent random variables $X$ and $Y \geq 0$ we write
\begin{equation}
  X \prec Y \quad \txt{if} \quad \bP[\abs{X} > N^\fc Y] = \OO_{\fc}(N^{-1/\fc}) \txt{ for all $\fc>0$}.
\end{equation}
If $X \prec Y$ then we also write $X = \OO_\prec(Y)$. If the the implicit constant in $\OO_\fc$ is the same for a family of random variables, we say that the $\prec$ is uniform in that family. All of our uses of $\prec$ will be uniform in the matrix indices and the spectral parameter of the Green's function.

The rest of the paper is devoted to the proofs of Theorems \ref{t:eigloc} and \ref{thm:univ}. In Section \ref{sec:notation} we introduce the Green's function, which is the main tool in the proof of Theorem \ref{t:eigloc}. In Section \ref{sec:switchings} we discuss switchings of graphs, which are the key operations on graphs that we use to generate self-consistent equations. In Section \ref{sec:polynomials}, we introduce a family of polynomials of the Green's functions entries that underlies our proof of Theorem \ref{t:eigloc}, and derive some basic estimates on them. In Section \ref{sec:P-construct} we derive the self-consistent equation for the Green's function in expectation. In Section \ref{sec:P-identification} we relate the self-consistent equation derived in the previous section with the Kesten--McKay law. In Section \ref{sec:P-moments}, we upgrade the self-consistent equation from Section \ref{sec:P-construct} to an equation in high probability. In Section \ref{sec:P-prop} we use the self-consistent equation from Section \ref{sec:P-moments} to derive a local law around the spectral edges, and as a consequence deduce Theorem \ref{t:eigloc}. Finally, in Section \ref{sec:universality} we use the rigidity estimates of Theorem \ref{sec:universality} to conclude edge universality and Theorem \ref{thm:univ}.

\section{Green's function}\label{sec:notation}

We consider the adjacency matrix $A$ restricted to the subspace orthogonal to the vector ${\bf 1} \deq (1,\dots, 1)^*$.
More precisely,
let $P_\perp \col \bR^N \to \bR^N$ be the orthogonal projection onto ${\bf 1}^{\perp}$,
explicitly given by $P_{\perp}=I-{\bf 1}{\bf 1^*}/N$ where $I$ is the $N\times N$ identity matrix.
Since $H$ is the normalized adjacency matrix of a regular graph, the matrices $H$ and $P_\perp$ commute: $HP_\perp=P_\perp H$.

For a spectral parameter $z \in \bC_+ \deq \{z\in \bC \col \Im[z]>0\}$ we define the \emph{Green's function} by
\begin{equation}
  G(z) \deq P_\perp (H-z)^{-1}P_{\perp}.
\end{equation}
Thus, $G$ and $(H-z)^{-1}$ agree on the image of $P_\perp$, which is the subspace of $\bR^N$ perpendicular to $\bm 1$.
The Green's function satisfies the relation
\begin{align}\label{e:GHexp}
G(z)H=HG(z)=zG(z)+P_{\perp}=zG(z)+(I-{\bf 1}{\bf 1^*}/N).
\end{align}
Moreover,
\begin{equation} \label{sumG0}
\sum_{i}G_{ij}(z)=\sum_{j}G_{ij}(z)=0.
\end{equation}
{Here, and throughout the following, sums over indices run over $i \in \qq{N}$.}
We denote the normalized trace of $G$, which is also the Stieltjes transform of the empirical spectral measure of $H|_{\bm 1^{\perp}}$, by
\begin{equation} \label{e:m}
  m(z) \deq \frac{1}{N} \sum_i G_{ii}(z).
\end{equation}
 We refer to $m(z)$ simply as the \emph{Stieltjes transform}.
Our goal is to approximate $m(z)$ by $\md(z)$,
the Stieltjes transform of the Kesten--McKay law \cite{MR0109367},
\begin{align*}
\md(z) \deq \int_\bR \frac{\rho_d(x)}{x-z}\rd x , \qquad \rho_d(x)\deq \left(1+\frac{1}{d-1}-\frac{x^2}{d}\right)^{-1}\frac{\sqrt{[4-x^2]_+}}{2\pi}.
\end{align*}
The Kesten--McKay law $\rho_d$ is the spectral measure at any vertex of the infinite $d$-regular tree (see, for example, \cite{MR3364516}
or \cite[Section~5]{MR3962004}). It has support $[-2,2]$ in our normalization. The Stieltjes transform $\md(z)$ is explicitly given by
\begin{equation} \label{e:md}
  \md(z) = -\pa{z+\frac{d}{d-1} \msc(z)}^{-1},
\end{equation}
where $\msc(z)$ is the Stieltjes transform of the Wigner semicircle law,
\begin{equation} \label{def_sc}
\msc(z)\deq\int_\bR\frac{\rhosc(x)}{x-z}\rd x , \qquad \rhosc(x)\deq\frac{\sqrt{[4-x^2]_+}}{2\pi},
\end{equation}
satisfying the self-consistent equation
\begin{equation} \label{e:sc_sce}
1+z\msc(z)+\msc^2(z)=0.
\end{equation}
Later we shall use that, alternatively, $\md(z)$ can be characterized by the self-consistent equation
\begin{align} \label{e:md-sce}
P_\infty(z,\md(z)) = 0, \quad
P_\infty(z,w) = 1 + z w + \frac{d}{d - 1} w^2 + \sum_{k \geq 2} \frac{(-2)^{k - 1} (2k - 3)!!}{k!} \, \frac{d}{(d - 1)^k} \, w^{2k}.
\end{align}
Indeed, from \eqref{e:md} and \eqref{e:sc_sce} we get
\begin{equation} \label{e:mdmscquad}
\frac{1}{\md(z)} = \frac{1}{\msc(z)} - \frac{\msc(z)}{d - 1},
\end{equation}
from which we obtain, for large enough $d$,
\begin{equation*}
\msc(z) = \frac{d - 1}{2 \md(z)} \biggl( \sqrt{1 + 4 \frac{\md(z)^2}{d - 1}} - 1\biggr) = \md(z) + \frac{d - 1}{2 \md(z)} \sum_{k \geq 2} (-1)^{k - 1} \frac{2^k (2k - 3)!!}{k!} \, \frac{\md(z)^{2k}}{(d - 1)^k}.
\end{equation*}
Plugging this into $0 = 1 + z \md(z) + \frac{d}{d - 1} \md(z) \msc(z)$, as follows from \eqref{e:md}, yields \eqref{e:md-sce}.

We fix a large $\fK > 0$ and define the spectral domain
\begin{equation} \label{e:D}
  \mathbf D \deq \{z=E+\ri \eta\col -\fK \leq E \leq \fK, N^{-1 + 1/\fK}\leq \eta \leq \fK\}.
\end{equation}
Here, and throughout the following, we use the notation
\begin{equation*}
z = E + \ri \eta
\end{equation*}
for the real and imaginary parts of $z$.
The local semicircle law for random regular graphs \cite{MR3688032}
shows that $m(z)$ is approximated by $\md(z)$ to order $1/\sqrt{d}$ (up to logarithmic corrections),
at least away from the edges $\pm 2$ of the spectral measure.  The next result follows from \cite[Theorem 1.1]{MR3688032}. In fact, \cite[Theorem 1.1]{MR3688032} gives a much better estimate for $\Lambdad$ for $z$ away from the edge $\pm 2$.
\begin{proposition}[{\cite[Theorem 1.1]{MR3688032}}]  \label{thm:rigidity}
With the deterministic control parameters
\begin{equation} \label{Lambda_weak}
\Lambdao(z) \deq \frac{1}{\sqrt{N\eta}} + \frac{1}{\sqrt{d}} + \frac{d^{3/2}}{N},\quad \Lambdad(z) \deq \left(\frac{1}{\sqrt{N\eta}} + \frac{1}{\sqrt{d}} + \frac{d^{3/2}}{N}\right)^{1/2},
\end{equation}
we have, for $1 \ll d\ll N^{2/3}$ and all $z\in \mathbf D$,
\begin{align} \label{G_estimates}
\max_{i\neq j}|G_{ij}(z)|\prec \Lambdao(z), \quad \max_{i}|G_{ii}(z)-\md(z)|\prec \Lambdad(z).
\end{align}
\end{proposition}

One by-product of our proof is an improved estimate for the Green's function entries close to the edges $\pm 2$.
The bound $1/\sqrt{d}$ is the best one can expect, since if $A_{ij}=1$, the off-diagonal Green's function entry $G_{ij}$ is of order $1/\sqrt{d}$.
In Proposition \ref{p:newboundG} below, we show that near the spectral edges the estimate \eqref{G_estimates} in fact holds with the smaller control parameters
\begin{align} \label{Lambda_strong}
  \Lambdao(z) =   \Lambdad(z) = 
  \frac{1}{\sqrt{N\eta}} + \frac{1}{\sqrt{d}} + \frac{d^{3/2}}{N}.
\end{align}

{Averaging over the index $i$, the estimate \eqref{G_estimates} implies an
  estimate on the Stieltjes transform $m(z)$. Using additional cancelations from this average,
  in this paper we shall derive a more precise estimate
  (see Theorem~\ref{thm:edgerigidity})
  from which we obtain our results about the extremal eigenvalues.}

Throughout this paper, we consistently omit the spectral parameter $z$ from our notation in quantities such as $G$ and $m$, unless it is needed to avoid confusion.

\section{Switchings and exchangeability} \label{sec:switchings}

Our analysis makes use of switchings for regular graphs and also makes some use of the invariance under the permutation of vertices.
We use ideas related to those introduced in \cite{MR3729611}, to which we also refer for references to other uses of switchings.

\subsection{Switchings}

As in \cite{MR3729611}, we define the signed adjacency matrices
\begin{equation}\label{e:defxi}
(\Delta_{ij})_{ab} \deq \delta_{ia} \delta_{jb} + \delta_{ib} \delta_{ja},\qquad
\xi_{ij}^{kl} \deq \Delta_{ij} + \Delta_{kl} - \Delta_{ik} - \Delta_{jl},
\end{equation}
corresponding to an edge at $ij$ respectively to a switching of the edges $ij$ and $kl$; see Figure~\ref{fig:switch1}.
Clearly we have $\xi_{ij}^{kl} \bm 1=0$ and $\xi_{ik}^{jl} = - \xi_{ij}^{kl}$.
 For any indices $i,j,k,l$, we denote the indicator function that the edges $ij$ and $kl$ are \emph{switchable}
(i.e.\ the edges $ij$ and $kl$ are present and the switching again results in a simple regular graph) by
\begin{align}\label{e:defchi}
\chi_{ij}^{kl}(A)=A_{ij}A_{kl}(1-A_{ik})(1-A_{jl}).
\end{align} 
\begin{figure}[t]
\begin{center}
\input{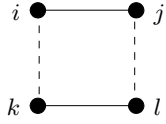}
\end{center}
\caption{A simple switching is given by replacing the solid edges by the dashed edges.
\label{fig:switch1}}
\end{figure}
In this section, we use switchings to estimate terms of the form
\begin{align}
\bE\left[\prod_{a=1}^{b}A_{i_aj_a}F(A)\right],
\end{align}
where $F$ is any function which depends on the random graph $A$, and possibly on the indices $i_1,j_1,\cdots,i_b,j_b$.
(Later, we shall take $F$ to be a polynomial of the Green's function entries $\{G_{ij}\}_{i,j\in\qq{N}}$ and the Stieltjes transform $m$.)

\begin{proposition}\label{p:bijection}
If the indices $i_1,j_1,k_1,l_1\cdots, i_b,j_b,k_b,l_b$ are distinct, we have the identity
\begin{align} \label{e:bijection}
\bE\left[F(A)\prod_{a=1}^{b}\chi_{i_aj_a}^{k_al_a}(A)\right]
=\bE\left[F\left(A+\sum_{a=1}^b\xi_{i_aj_a}^{k_al_a}\right)\prod_{a=1}^{b}\chi_{i_ak_a}^{j_al_a}(A)\right],
\end{align}
where the indicator function $\chi$ is as defined in \eqref{e:defchi}.\end{proposition}

\begin{proof}
Define the sets of graphs
\begin{align}\begin{split}
\cG_1&=\left\{A \col \prod_{a=1}^{b}\chi_{i_aj_a}^{k_al_a}(A)=1\right\},\\
\cG_2&=\left\{A\col\prod_{a=1}^{b}\chi_{i_ak_a}^{j_al_a}(A)=1\right\}.
\end{split}\end{align}
By our assumption that the indices $i_1,j_1,k_1,l_1\cdots, i_b,j_b,k_b,l_b$ are distinct, there is a simple bijection between $\cG_1$ and $\cG_2$, namely
\begin{align}\label{e:G1G2bij}
A\in \cG_2\mapsto A+\sum_{a=1}^b\xi_{i_aj_a}^{k_al_a}\in \cG_1.
\end{align}
Since $\bP$ is the uniform probability measure on $d$-regular graphs, the claim follows from \eqref{e:G1G2bij}.
\end{proof}

In the switching in \eqref{e:bijection}, the indicator function $\chi_{ij}^{kl}(A)$
enforces that the matrix $A+\xi^{kl}_{ij}$ is again the adjacency matrix of a simple graph.
Note that without this indicator function, such as in the following corollary and elsewhere throughtout our proof,
$A+\xi_{ij}^{kl}$ is not necessarily the adjacency matrix of a simple graph, just a real symmetric matrix. This does however not affect our arguments, which should be viewed as operating with general symmetric matrices instead of adjacency matrices of simple graphs.

\begin{corollary} \label{c:intbp1}
Fix indices $i,j\in \qq{N}$ and $i \neq j$.  Let $F  \equiv F_{ij}$ be a random variable possibly depending on $ij$. With the random control parameter
\begin{equation*}
\cal C_{ij}(F,A)\deq |F(A)|+\max_{kl}|F(A+\xi_{ij}^{kl})|
\end{equation*}
we have { the integration by parts formula}
\begin{align} \label{e:intbp1}
\bE[A_{ij}F(A)]
&=\frac{d}{N}\bE[F(A)]+
\frac{1}{Nd}\sum_{kl}\bE\qa{A_{ik}A_{jl}(F(A+\xi_{ij}^{kl})-F(A))}+\OO\left(\frac{d\bE[A_{ij}\cal C_{ij}(F,A)]}{N}\right).
\end{align}
\end{corollary}

\begin{proof}
  Since $\sum_{kl} A_{kl}= dN$, we have
\begin{align}\begin{split}\label{e:ffterm}
    \bE[A_{ij}F(A)]
    &= \frac{1}{Nd} \sum_{kl} \bE[A_{ij}A_{kl} F(A)] 
    = \frac{1}{Nd} \sum_{\substack{kl\col ijkl \\ \txt{distinct}}} \bE[\chi_{ij}^{kl}(A) F(A)] \\
    &+ \OO\pa{\frac{1}{Nd} \sum_{\substack{kl\col ijkl \\  \txt{not distinct}}} \bE[A_{ij}A_{kl} |F(A)|]+\frac{1}{Nd}\sum_{kl} \bE[A_{ij}A_{kl}(A_{ik}+A_{jl}) |F(A)|]}\\
    &= \frac{1}{Nd} \sum_{\substack{kl\col ijkl  \\ \txt{distinct}}} \bE[\chi_{ij}^{kl}(A) F(A)] +\OO\left(\frac{d\bE[A_{ij}\cal C_{ij}(F,A)]}{N}\right),
    \end{split}
  \end{align}  
  where we used that the row sums and column sums of $A$ are $d$.
  By \eqref{e:bijection}, the first term on the right-hand side of \eqref{e:ffterm} equals
  \begin{align}\begin{split}\label{e:ssterm}
    &\phantom{{}={}}\frac{1}{Nd} \sum_{\substack{kl\col ijkl \\  \txt{distinct}}} \bE[\chi_{ik}^{jl}(A) F(A+\xi_{ij}^{kl})]
    =
      \frac{1}{Nd} \sum_{kl} \bE[A_{ik}A_{jl} F(A+\xi_{ij}^{kl})] \\
    & +
      \OO\pa{\frac{1}{Nd} \sum_{\substack{kl\col ijkl \\ \txt{not distinct}}} \bE[A_{ik}A_{jl} |F(A+\xi_{ij}^{kl})|]+\frac{1}{Nd} \sum_{kl} \bE[A_{ik}A_{jl}(A_{ij}+A_{kl}) |F(A+\xi_{ij}^{kl})|]} \\
    &=
      \frac{1}{Nd} \sum_{kl} \bE[A_{ik}A_{jl} F(A+\xi_{ij}^{kl})]+\OO\left(\frac{d\bE[A_{ij}\cal C_{ij}(F,A)]}{N}\right),
     \end{split}
  \end{align}
  where we used that $\cal C_{ij}(F,A)$ is independent of indices $k,l$ and we can sum over them. The claim \eqref{e:intbp1} follows from combining \eqref{e:ffterm}, \eqref{e:ssterm} and the fact $\sum_{kl}A_{ik}A_{jl}=d^2$.
\end{proof}

For $b,c \geq 0$ and multi-indices ${\bf i} \in \qq{N}^b$ and ${\bf j} \in \qq{N}^c$, we denote by ${\bf i} {\bf j}\in \qq{N}^{b+c}$ their concatenation. We shall often need the following random control parameter.

\begin{definition} \label{def:C}
For fixed $b,c\geq 0$ we denote the $b$-tuples  ${\bf i} = (i_1,\dots,i_b), {\bf j} = (j_1,\dots,j_b)$,  ${\bf k} = (k_1,\dots,k_b), {\bf l} = (l_1,\dots,l_b)$ and the $c$-tuple ${\bf m} = (m_1,\dots,m_c)$.
Let $F = \{F_{\bf ijm}\}$ be a family of random variables indexed by $\bf ijm$. Define the random control parameter
\begin{equation} \label{def_CA}
\cal C(F,A)\deq \max_{\bf m}\max_{\bf ijkl}
\pa{
\left|F_{\bf ijm}(A)\right|
+\left|F_{\bf ijm}\left(A+\sum_{a=1}^b\xi_{i_aj_a}^{k_al_a}\right)\right|}.
\end{equation}
\end{definition}

\begin{corollary}\label{c:intbp}
Let $F_{\bf ijm}$ be as in Definition \ref{def:C}. We have the integration by parts formula
\begin{align}\begin{split} \label{e:intbp}
&\phantom{{}={}}\frac{1}{N^{b+c}d^b}\sum_{\bf m}\sum_{\bf ij}\bE\left[\prod_{a=1}^{b}A_{i_aj_a}F_{\bf ijm}(A)\right]\\
&=
\frac{1}{N^{2b+c}d^{2b}}\sum_{\bf m}\sum_{\bf ijkl}\bE\left[\prod_{a=1}^{b}A_{i_ak_a}A_{j_al_a}\left(F_{\bf ijm}\left(A+\sum_{a=1}^b\xi_{i_aj_a}^{k_al_a}\right)-F_{\bf ijm}(A)\right)\right]\\
&+\frac{1}{N^{2b+c}}\sum_{\bf m}\sum_{\bf ij}\bE\left[F_{\bf ijm}(A)\right]+\OO\left(
\frac{d\bE[\cal C(F, A)]}{N}\right).
\end{split}\end{align}
\end{corollary}

\begin{proof}
Since the row and column sums of $A$ equal $d$, by introducing new indices $k_1,l_1,\cdots, k_b,l_b$,
we rewrite the left-hand side of \eqref{e:intbp} as
\begin{align}\begin{split}\label{e:newindices}
&\phantom{{}={}}\frac{1}{N^{b+c}d^b}\sum_{\bf m}\sum_{\bf ij}\bE\left[\prod_{a=1}^{b}A_{i_aj_a}F_{\bf ijm}(A)\right]=\frac{1}{N^{2b+c}d^{2b}}\sum_{\bf m}\sum_{\bf ijkl}\bE\left[\prod_{a=1}^{b}A_{i_aj_a}A_{k_al_a}F_{\bf ijm}(A)\right]\\
&=\frac{1}{N^{2b+c}d^{2b}}\sum_{\bf m}\sum_{\substack{\bf ijkl \\ \txt{distinct}}}\bE\left[\prod_{a=1}^{b}\chi_{i_aj_a}^{k_al_a}(A)F_{\bf ijm}(A)\right]+\frac{1}{N^{2b+c}d^{2b}}\sum_{\bf m}
\sum_{\substack{\bf ijkl \\ \txt{not distinct}}}\bE\left[\prod_{a=1}^{b}A_{i_aj_a}A_{k_al_a}F_{\bf ijm}(A)\right]\\
&
+\frac{1}{N^{2b+c}d^{2b}}\sum_{\bf m}\sum_{\substack{\bf ijkl\\ \txt{distinct}}}\bE\left[\left(\prod_{a=1}^{b}A_{i_aj_a}A_{k_al_a}-\prod_{a=1}^{b}\chi_{i_aj_a}^{k_al_a}(A)\right)F_{\bf ijm}\right].
\end{split}\end{align} 
The second last term of the last right-hand side in \eqref{e:newindices} can be estimated by
\begin{align}\begin{split} \label{e:avt1}
&\phantom{{}={}}\left|\frac{1}{N^{2b+c}d^{2b}}\sum_{\bf m}\sum_{\substack{\bf ijkl \\ \txt{not distinct}}}\prod_{a=1}^{b}A_{i_aj_a}A_{k_al_a}F_{\bf ijm}(A)\right|\\
&\leq \frac{1}{N^{2b+c}d^{2b}}\sum_{\bf m}\sum_{\substack{\bf ijkl \\ \txt{not distinct}}}\prod_{a=1}^{b}A_{i_aj_a}A_{k_al_a}\cal C(F,A)\\
&\leq \frac{1}{N^{2b}d^{2b}}\sum_{\substack{\bf ijkl \\ \txt{not distinct}}}\prod_{a=1}^{b}A_{i_aj_a}A_{k_al_a}\cal C(F,A)\leq  \frac{1}{N}  \cal C(F,A),
\end{split}\end{align}
where in the last line we used that $\sum_{i}A_{ij}=\sum_j A_{ij}=d$, that $A_{ii}=0$,
and hence that at least one factor $1/N$ remains because of the constraint in the sum that $\bf ijkl$ be not distinct.
Similarly, using $ |A_{ij}A_{kl}- \chi_{ij}^{kl}(A)|=A_{ij}A_{kl}[A_{ik}+A_{jl}-A_{ik}A_{jl}]\leq A_{ij}A_{kl}[A_{ik}+A_{jl}]$, we have
\begin{align}\label{e:avt2}
\left|\frac{1}{N^{2b+c}d^{2b}}\sum_{\bf m}\sum_{\bf ijkl}\left[\left(\prod_{a=1}^{b}A_{i_aj_a}A_{k_al_a}-\prod_{a=1}^{b}\chi_{i_aj_a}^{k_al_a}(A)\right)F_{\bf ij m}(A)\right]
\right|\leq \frac{d}{N}\cal C(F,A).
\end{align}
By plugging the estimates \eqref{e:avt1} and \eqref{e:avt2} into \eqref{e:newindices}, we get
\begin{align}\begin{split}\label{e:newindices2}
&\phantom{{}={}}\frac{1}{N^{b+c}d^b}\sum_{\bf m}\sum_{\bf ij}\bE\left[\prod_{a=1}^{b}A_{i_aj_a}F_{\bf ijm}(A)\right]\\
&=\frac{1}{N^{2b+c}d^{2b}}\sum_{\bf m}\sum_{\substack{\bf ijkl \\ \txt{distinct}}}\bE\left[\prod_{a=1}^{b}\chi_{i_aj_a}^{k_al_a}(A)F_{\bf ijm}(A)\right]+\OO\left(\frac{d\bE[\cal C(F,A)]}{N}\right).
\end{split}\end{align}
By Proposition \ref{p:bijection} and an estimate analogous to the one above, we have
\begin{align}\begin{split}\label{e:switchindices}
&\phantom{{}={}}\frac{1}{N^{2b+c}d^{2b}}\sum_{\bf m}\sum_{\substack{\bf ijkl \\ \txt{distinct}}}\bE\left[\prod_{a=1}^{b}\chi_{i_aj_a}^{k_al_a}(A)F_{\bf ijm}(A)\right]\\
&=\frac{1}{N^{2b+c}d^{2b}}\sum_{\bf m}\sum_{\substack{\bf ijkl \\ \txt{distinct}}}\bE\left[\prod_{a=1}^{b}\chi_{i_ak_a}^{j_al_a}(A)F_{\bf ijm}\left(A+\sum_{a=1}^b \xi_{i_aj_a}^{k_al_a}\right)\right]\\
&=\frac{1}{N^{2b+c}d^{2b}}\sum_{\bf m}\sum_{\bf ijkl}\bE\left[\prod_{a=1}^{b}A_{i_ak_a}A_{j_al_a}F_{\bf ijm}\left(A+\sum_{a=1}^b \xi_{i_aj_a}^{k_al_a}\right)\right]+\OO\left(\frac{d\bE[\cal C(F,A)]}{N}\right).
\end{split}\end{align}
The claim now follows from combining \eqref{e:newindices2} and \eqref{e:switchindices}.
\end{proof}

\section{Polynomials in Green's function entries} \label{sec:polynomials}

In this section we collect some estimates on the Green's function $G$ and the Stieltjes transform of the spectral measure $m$. These will be used repeatedly in the rest of the paper. We also introduce polynomials in the Green's function entries, and record some of their basic properties.
We work under the following assumption throughout this section.
{ Recall the Stieltjes transform $\md$ of the Kesten--McKay law from \eqref{e:md}.}

\begin{assumption} \label{ass:main}
We assume that $1 \ll d\ll N^{2/3}$
and that there are deterministic $z$-dependent control parameters $\Lambdao,\Lambdad \in [d^{-1/2}, 1]$ such that 
\begin{equation} \label{Lambda_assumptions}
\max_{i} \abs{G_{ii}-\md} \prec \Lambdad,\quad \max_{i \neq j} \abs{G_{ij}} \prec \Lambdao, 
\end{equation}
for all $z \in \bld D$ defined in \eqref{e:D}.
\end{assumption}
Note that, by Proposition \ref{thm:rigidity}, we know that Assumption \ref{ass:main} holds at least when $\Lambdao$ and $\Lambdad$ are given by \eqref{Lambda_weak}.

By our definition, the Green's function $G=P_\perp (H-z)^{-1}P_\perp $ is symmetric and satisfies \eqref{sumG0}.
The \emph{Ward identity} states that the Green's function $G$ satisfies
\begin{equation} \label{e:Ward}
  \frac{1}{N} \sum_{j} |G_{ij}|^2 = \frac{\Im[G_{ii}]}{N\eta},
  \qquad
  \frac{1}{N} \sum_{ij} |G_{ij}|^2 = \frac{\Im[m]}{\eta};
\end{equation}
it can be proved using the resolvent identity on $G - G^*$.
{Here recall that $m$ is the Stieltjes transform  \eqref{e:m} of the empirical spectral
  measure.}
We record the following basic result, which we shall use tacitly throughout the rest of the paper.

\begin{lemma} \label{l:basicestimates}
Suppose that $1 \ll d\ll N^{2/3}$.
\begin{enumerate}
\item
For any $\fc > 0$, with probability at least $1 -  \OO_{\fc}(N^{-1/\fc})$ we have
for all $z \in \mathbf D$
\begin{align}\label{e:Lambound}
\max_{xy}|G_{xy}|\leq 2.
\end{align}

\item
Denoting by $u_\alpha(i)$ the $i$-th component of the $\alpha$-th normalized eigenvector of $P_\perp H P_\perp$, we have the delocalization estimate
\begin{equation}\label{e:evbound}
  \max_\alpha\max_i |u_\alpha(i)| \prec 1/\sqrt{N}.
\end{equation}
\end{enumerate}
\end{lemma}
\begin{proof}
The claim (i) follows from Proposition \ref{thm:rigidity}, the estimate $\abs{\msc} \leq 1$, and a simple $N^{-3}$-net argument in $\mathbf D$ combined with a union bound to obtain a simultenous estimate for all $z \in \mathbf D$. The claim (ii) follows from \cite[Corollary 1.2]{MR3688032}.
\end{proof}

\begin{remark} \label{r:basicestimates}
More explicitly, Lemma \ref{l:basicestimates}(i) says that $\max_{xy} \abs{G_{xy}(A)} \leq 2$ with probability at least $1 -  \OO_{\fc}(N^{-1/\fc})$. As a consequence, for any fixed $b \in \N$ we find using a simple resolvent expansion that
\begin{equation*}
\max_{xy} \absbb{G_{xy} \biggl(A + \sum_{a=1}^b \xi_{i_aj_a}^{k_al_a}\biggr)} \leq 2 + \OO(d^{-1/2})
\end{equation*}
with probability at least $1 -  \OO_{\fc}(N^{-1/\fc})$. Moreover, by a similar argument, using \eqref{Lambda_assumptions}, for the off-diagonal entries we have the estimate
\begin{equation*}
\max_{x \neq y} \absbb{G_{xy} \biggl(A + \sum_{a=1}^b \xi_{i_aj_a}^{k_al_a}\biggr)}  \prec \Lambdao.
\end{equation*}
\end{remark}

We define the discrete and continuous derivatives for any indices $i,j,k,l$, 
\begin{equation} \label{e:D-defn}
D_{ij}^{kl}F(A) \deq F(A+\xi_{ij}^{kl})-F(A), \quad
\partial_{ij}^{kl} F(A) \deq \sqrt{d-1}\left.\partial_t  F(A + t \xi_{ij}^{kl})\right|_{t = 0},
\end{equation}
where the matrix $\xi_{ij}^{kl}$ was defined in \eqref{e:defxi}.
Note that $\partial_{ij}^{kl}$ is the directional derivative in the direction $\xi_{ij}^{kl}$ of the rescaled variable $H=A/\sqrt{d-1}$.
For the discrete derivative operator $D_{ij}^{kl}$, we have the discrete product rule
\begin{equation} \label{e:D-product}
  D_{ij}^{kl}(FG) = (D_{ij}^{kl}F)G + F(D_{ij}^{kl}G) +  (D_{ij}^{kl}F)(D_{ij}^{kl}G),
\end{equation}
and the Taylor expansion with remainder gives
\begin{align} \label{e:D-expand}
 D_{ij}^{kl} F(A) = \sum_{n=1}^{\fb-1}  \frac{1}{n!} \left(\frac{\del_{ij}^{kl}}{\sqrt{d-1}}\right)^nF(A) +  \frac{1}{\fb!}\left(\frac{\del_{ij}^{kl}}{\sqrt{d-1}}\right)^\fb F(A+\theta\xi_{ij}^{kl}),
  \end{align}
for some $0\leq \theta\leq 1$.

For any indices $i,j,k,l$ (which might be not distinct),
the derivatives of the Green's function entries $G_{ij}$ and the Stieltjes transform $m$ are given by
\begin{align}
\label{e:dG}
&\del_{ij}^{kl}G_{ij} =-G_{ii}G_{jj}-G_{ij}G_{ij}-G_{ik}G_{lj}-G_{il}G_{kj}+G_{il}G_{jj}+G_{ii}G_{kj}+G_{ik}G_{ij}+G_{ij}G_{lj},
\\
\label{e:dm}
\begin{split}
&\del_{ij}^{kl}m
=\frac{2}{N}\sum_{a=1}^N(-G_{ia}G_{ja}-G_{ka}G_{la}+G_{ia}G_{ka}+G_{ja}G_{la})\\
&\phantom{{}\del_{ij}^{kl}m{}}=\frac{2}{N}(-(G^2)_{ij}-(G^2)_{kl}+(G^2)_{ik}+(G^2)_{jl}).
\end{split}
\end{align}

A central object in our proof is the following notion of a polynomial in the entries of the Green's function. 
\begin{definition} \label{d:evaluation}
\begin{enumerate}
\item
  Let $F = F(\{x_{st}\}_{s,t = 1}^r)$ be a polynomial in the $r^2$ abstract variables $\{x_{st}\}_{s,t = 1}^r$. We denote its degree by $\deg(F)$. For ${\bf i} \in \qq{N}^r$, we define its \emph{evaluation} on the Green's function by
  \begin{equation}
    F_{\bf i} = F(\{G_{i_s i_t}\}_{s,t = 1}^r),
  \end{equation}
  and say that $F_{\bf i}$ is a \emph{polynomial in the Green's function entries} $\{G_{i_s i_t}\}_{s,t = 1}^r$. By a slight abuse of notation, we sometimes abbreviate $F$ instead of $F_{\bld i}$ for the polynomials in the Green's function entries when there is no risk of confusion.
  \item
  Let $F = F(\{x_{st}\}_{s,t = 1}^r)$ be a monomial in $r^2$ variables.
  Then the number of off-diagonal entries of $F$ is the total degree of variables $x_{st}$ with $s \neq t$. If the number of off-diagonal entries of $F$ is zero then we define $\chi_F = 1$, otherwise we define $\chi_F = 0$.
\item
  For $U$ a polynomial in two variables, by a slight abuse of notation we often abbreviate $U = U(m, \bar m)$ for the polynomial in the Stieltjes transform $m$ and its complex conjugate $\bar m$. For $r, \bar r \in \N$ we abbreviate $U^{(r, \bar r)} \deq \partial_m^r \partial_{\bar m}^{\bar r} U$.
\end{enumerate}
\end{definition}

\begin{claim}
Using \eqref{e:Lambound} and \eqref{e:evbound}, we have for any indices $a,b\in \qq{N}$
\begin{align} \label{e:Gkbd}
  \left|\frac{1}{N^{k-1}} (G^{k})_{ab}\right|
  \prec \frac{\Im[m]}{(N\eta)^{k-1}},
  \quad
  \frac{1}{N^{2k-1}} (|G|^{2k})_{ab}
  \prec \frac{\Im[m]}{(N\eta)^{2k-1}}.
\end{align}
In particular, 
\begin{equation} \label{e:GGbd}
\frac{1}{N}\sum_{j=1}^N |G_{aj}G_{jb}| \prec \frac{\Im[m]}{N\eta},
\end{equation}
and for distinct indices $i,j,k,l$, 
  $(\del_{ij}^{kl})^{2}G_{ij}$ is a cubic polynomial in the Green's function with at least one off-diagonal factor,
  \begin{equation} \label{e:d2G}
    |(\del_{ij}^{kl})^{2}G_{ij}|\prec\Lambdao.
  \end{equation}

We also have the following estimates for the derivatives of the Stieltjes transform $m$:
for any integer $s\geq 1$,
\begin{equation} \label{e:d2m}
  |(\del_{ij}^{kl})^s m| \prec \frac{\Im [m]}{N\eta},
\end{equation}
and for any fixed polynomial $U$,
\begin{equation} \label{e:dPbd}
  \left|(\del_{ij}^{kl})^sU(m, \bar m)\right|\prec  \max_{r + \bar r  \geq 1}|U^{(r, \bar r)}(m, \bar m)| \left(\frac{\Im[m]}{N\eta}\right)^r.
\end{equation}
\end{claim}

\begin{proof}
  \eqref{e:Gkbd} follows directly from the spectral decomposition and the delocalization of the eigenvectors \eqref{e:evbound}:
  \begin{align}\begin{split} \label{e:Gkbd-pf}
  \left|\frac{1}{N^{k-1}} (G^{k})_{ab}\right|
  &= \left|\frac{1}{N^{k-1}} \sum_{\alpha} \frac{u_{\alpha}(a)u_{\alpha}(b)}{(\lambda_{\alpha}-z)^{k}}\right|
  \prec \frac{1}{(N\eta)^{k-2}} \frac{1}{N^{2}} \sum_{\alpha} \frac{1}{|\lambda_\alpha-z|^2}
  =  \frac{\Im[m]}{(N\eta)^{k-1}},
  \\
  \left|\frac{1}{N^{2k-1}} (|G|^{2k})_{ab}\right|
  &= \frac{1}{N^{2k-1}} \left|\sum_{\alpha} \frac{u_{\alpha}(a)u_{\alpha}(b)}{|\lambda_{\alpha}-z|^{2k}}\right|
  \prec \frac{1}{(N\eta)^{2k-2}} \frac{1}{N^{ 2}} \sum_{\alpha} \frac{1}{|\lambda_\alpha-z|^{2}}
  =  \frac{\Im[m]}{(N\eta)^{2k-1}}.
\end{split}\end{align}
The claim \eqref{e:GGbd} follows from Young's inequality and \eqref{e:Gkbd} by taking $k=2$,
\begin{align}
\frac{1}{N}\sum_{j=1}^N |G_{aj}G_{jb}|
\leq \frac{1}{N}\sum_{j=1}^N (|G_{aj}|^2+|G_{jb}|^2)=\frac{1}{N}\left((|G|^2)_{aa}+|G|^2_{bb}\right)\prec \frac{\Im[m]}{N\eta}.
\end{align}
For \eqref{e:d2G},  we notice that from \eqref{e:dG}, one can directly verify that none of the derivatives
  $\partial_{ij},\partial_{kl},\partial_{ik},\partial_{jl}$   produces a diagonal term,
  and \eqref{e:d2G} follows from \eqref{e:Lambound}.

  For  \eqref{e:d2m}, one can check using \eqref{e:dm} that $(\del_{ij}^{kl})^s m$ is a sum of terms in the following form
  \begin{align}\label{e:someterm}
  \frac{1}{N}\sum_{a=1}^NG_{ai_1}G_{i_2i_3}\cdots G_{i_{2s}a} 
  \end{align}
  where $i_1,i_2,\cdots, i_{2s}\in \{i,j,k,l\}$. Thanks to \eqref{e:Lambound} and \eqref{e:GGbd}, we can bound \eqref{e:someterm} as
  \begin{align}\label{e:dermbound}
  \left|\frac{1}{N}\sum_{a=1}^NG_{ai_1}G_{i_2i_3}\cdots G_{i_{2s}a} \right|\prec \left|\frac{1}{N}\sum_{a=1}^NG_{ai_1}G_{i_{2s}a} \right|\prec \frac{\Im[m]}{N\eta},
  \end{align}
  and the claim \eqref{e:d2m} follows. For \eqref{e:dPbd}, the derivative $(\del_{ij}^{kl})^s U(m)$ is a sum of terms of the form
  \begin{align} \label{e:U_derterm}
  U^{(r, \bar r)}(m, \bar m) (\del_{ij}^{kl})^{s_1}m(\del_{ij}^{kl})^{s_2}m\cdots (\del_{ij}^{kl})^{s_r}m
  (\del_{ij}^{kl})^{\bar s_1} \bar m(\del_{ij}^{kl})^{\bar s_2}\bar m\cdots (\del_{ij}^{kl})^{\bar s_{\bar r}}\bar m
  ,
  \end{align}
  where $r + \bar r\geq 1$, $s_1,s_2,\cdots,s_r, \bar s_1, \bar s_2,\cdots, \bar s_{\bar r} \geq 1$ and $s_1+\cdots+s_r + \bar s_1+\cdots+ \bar s_{\bar r} =s$. 
Thanks to \eqref{e:d2m} we have 
  \begin{align}\label{e:someterm2}
\abs{\text{\eqref{e:U_derterm}}}\prec |U^{(r, \bar r)}(m, \bar m)| \left(\frac{\Im[m]}{N\eta}\right)^{r+ \bar r}.
  \end{align}
The claim \eqref{e:dPbd} follows from \eqref{e:someterm2}.

\end{proof}

\begin{claim}\label{c:DGDm}
Let $U$ be a fixed polynomial of the Stieltjes transform $m$ and its complex conjugate $\bar m$.
For any indices $i,j,k,l,a,b$ and fixed positive integer $\fb>0$, we have
\begin{align}\label{e:exp1}
&D_{ij}^{kl} G_{ab}(A)
=\sum_{n=1}^{\fb-1} \frac{(-1)^n}{(d-1)^{n/2}} (G(\xi_{ij}^{kl}G)^{n})_{ab} + \OO_\prec(d^{-\fb/2})
=\OO_\prec(d^{-1/2}),\\
&D_{ij}^{kl} m(A)
= \OO_\prec\left(\frac{\Im[m]}{d^{1/2}N\eta}\right),\label{e:exp2}\\
&D_{ij}^{kl}U(A)
= \OO_\prec \pBB{ \frac{1}{d^{1/2}}\max_{s + \bar s\geq 1}|U^{(s, \bar s)}(m, \bar m)| \pbb{\frac{\Im[m]}{N\eta}}^{s + \bar s} }.\label{e:exp3}
\end{align}
\end{claim}

\begin{proof}
  By the Taylor expansion \eqref{e:D-expand},
  \begin{equation}\label{e:diffseries}
    D_{ij}^{kl} G_{ab}(A)
    =\sum_{n=1}^{\fb-1} \frac{1}{n!(d-1)^{n/2}}(\del_{ij}^{kl})^nG_{ab}
    +\frac{1}{\fb!(d-1)^{\fb/2}}(\del_{ij}^{kl})^\fb G_{ab}(A+\theta \xi_{ij}^{kl}),
  \end{equation}
  for some random $\theta\in [0,1]$. Thanks to Lemma \ref{l:basicestimates} and Remark \ref{r:basicestimates},
  \begin{align}\label{e:dgbound}
(\del_{ij}^{kl})^nG_{ab}=(-1)^n n!(G(\xi_{ij}^{kl}G)^n)_{ab}\prec 1,  \quad(\del_{ij}^{kl})^\fb G_{ab}(A+\theta \xi_{ij}^{kl})\prec 1.
  \end{align}
  The expression \eqref{e:exp1} follows from the bound \eqref{e:dgbound}. The estimate \eqref{e:exp2} follows from averaging \eqref{e:diffseries} and using \eqref{e:dermbound}. The estimate \eqref{e:exp3} follows from \eqref{e:exp2} and the discrete product rule \eqref{e:D-product}.
\end{proof}

As an application of Claim \ref{c:DGDm}, we have the following estimate, which says essentially that when acting with discrete derivatives on a product of the form $FU$, where $U$ is a polynomial in $m$ and $F$ a polynomial in the Green's function entries (recall Definition \ref{d:evaluation}), the main contribution is given by differentiating $F$.
\begin{claim}\label{c:taylorexp}
  Let $F$ be a fixed polynomial of Green's function entries $\{G_{ij}\}_{i,j\in \qq{N}}$, and $U$ a fixed polynomial in the Stieltjes transform $m$ and its complex conjugate $\bar m$. 
For any indices $b\geq 1$, $i_1,j_1, k_1,l_1,\cdots, i_b, j_b,k_b,l_b$ and positive integer $\fb\geq 1$, we have
\begin{align}\begin{split}\label{e:expFU}
  &\phantom{{}={}}FU\left(A+\sum_{a=1}^b\xi_{i_aj_a}^{k_al_a}\right)-FU(A)
  = \left( \sum_{n=1}^{\fb-1} \frac{1}{n!(d-1)^{n/2}} \left(\sum_{a=1}^b\del_{i_aj_a}^{k_al_a}\right)^n F(A)\right)U(A)\\
  &\phantom{{}={}}+\OO_\prec\left(\frac{|U(A)|}{d^{\fb/2}} + \frac{|F(A+\sum_{a=1}^b\xi_{i_aj_a}^{k_al_a})|}{d^{1/2}}\max_{s + \bar s\geq 1}|U^{(s, \bar s)}(m, \bar m)| \pbb{\frac{\Im[m]}{N\eta}}^{s + \bar s} \right).
\end{split}\end{align}
\end{claim}

\begin{proof}
We denote $\xi = \sum_{a=1}^b \xi_{i_aj_b}^{k_a l_a}$, and rewrite the left-hand side of \eqref{e:expFU} as 
\begin{align}\begin{split}\label{e:diffdecompose}
FU\left(A+\xi\right)-FU(A)
&=\left(F\left(A+\xi\right)-F(A)\right)U(A)
+F\left(A+\xi\right)\left(U\left(A+\xi\right)-U(A)\right).
\end{split}\end{align}
Since $F$ is a polynomial of Green's function entries $\{G_{ij}\}_{i,j\in \qq{N}}$, by the same argument as for \eqref{e:exp1}, we have
\begin{align}\label{e:fbound}
F\left(A+\xi\right)-F(A)= \sum_{n=1}^{\fb-1} \frac{1}{n!(d-1)^{n/2}}\left(\sum_{a=1}^b\del_{i_aj_a}^{k_al_a}\right)^n F(A)+\OO_\prec \left(d^{-\fb/2}\right).
\end{align}
Similarly, since $U$ is a polynomial of the Stieltjes transform $m$ and its complex conjugate $\bar m$, by  \eqref{e:exp3}, we have
\begin{align}\label{e:ubound}
U\left(A+\xi\right)-U(A)=\OO_\prec\left( \frac{1}{d^{1/2}}\max_{s + \bar s\geq 1}|U^{(s, \bar s)}(m, \bar m)| \pbb{\frac{\Im[m]}{N\eta}}^{s + \bar s} \right).
\end{align}
The claim \eqref{e:expFU} follows from  plugging \eqref{e:fbound} and \eqref{e:ubound} into \eqref{e:diffdecompose}.
\end{proof}

We conclude this section with an elementary result for the operator $\prec$, which we shall use tacitly throughout the following sections.

\begin{lemma} \label{lem:prec_exp}
Suppose that $A$ and $B$ are nonnegative random variables satisfying $A \leq N^C$ and $B \geq N^{-C}$ for some constant $C > 0$. Then $A \prec B$ implies $\bE [A] \prec \bE [B]$.
\end{lemma}

\section{Self-consistent equation in expectation}
\label{sec:P-construct}

In this section we derive the self-consistent equation  \emph{in expectation} for the Stieltjes transform $m$;
in Section~\ref{sec:P-moments} below we shall extend this self-consistent equation to a high probability estimate.

\begin{proposition}\label{p:DSE}
Suppose that Assumption \ref{ass:main} holds.
For every fixed integer $\fa \geq 1$, there exists a polynomial, depending on $d$ and $\fa$ but not $N$, and whose degree depends on $\fa$ only,
\begin{align} \label{e:P_split}
P_\fa(z,w)=1+zw+Q_\fa(w),
\end{align}
where
\begin{align} \label{e:Q_a}
Q_\fa(w)=\frac{dw^2}{d-1} +  \frac{1}{d}\left(a_3 w^3 +a_4 w^4+\cdots \right),
\end{align}
is a polynomial with bounded coefficients $a_3,a_4,\dots$ such that, for any $z \in \mathbf D$,
\begin{align}\label{e:defP}
\bE[P_\fa(z, m)]\prec \frac{1}{d^{\fa/2}}+\frac{\bE[\Im[m]]}{N\eta}+\frac{d^{3/2}\Lambdao}{N}.
\end{align}
\end{proposition}

\begin{remark}
  An expansion \emph{in expectation} similar to Proposition~\ref{p:DSE} (with different coefficients) would be possible for the Erd\H{o}s--R\'enyi graph
  in the same regime of expected degree.
  The essential difference between the random regular and the Erd\H{o}s--R\'enyi graph is in
  the high moment estimates in Section~\ref{sec:P-moments}, using which we convert the expansion in
  expectation to one in high probability.
  For the random regular graph, there are fundamental cancellations arising from the degree constraint, which imply stronger concentration than possible for the Erd\H{o}s--R\'enyi graph; these cancellations are manifest only in the high probability expansion.
  We emphasize that without concentration, the self-consistent  equation in expectation does not lead to a closed equation for $m$ (or its expectation), and hence does not provide useful spectral information.
\end{remark}

We shall show that the estimate \eqref{e:defP} results from the switching invariance of random regular graphs.
 It may be viewed as an approximate \emph{Schwinger--Dyson Equation} for the random regular graph ensemble; in statistical mechanics and field theory, such equations are typically derived by integration by parts.

Before giving the proof of Proposition~\ref{p:DSE}, we explain the mechanism behind it. Starting from \eqref{e:GHexp} we obtain
\begin{equation*}
1+zm=\sum_{ij}\frac{A_{ij}G_{ij}}{N(d-1)^{1/2}} + \text{(error)},
\end{equation*}
where we use (error) to denote a small error that we do not keep track of in this sketch.
Taking the expectation and using Corollary~\ref{c:intbp1}, recalling that $\sum_i G_{ij} = 0$
and recalling the notation \eqref{e:D-defn}, we get
\begin{equation*}
\E[1+zm] =  \frac{1}{N^2 d (d - 1)^{1/2}} \sum_{ijkl} \E \qb{A_{ik} A_{jl} D_{ij}^{kl} G_{ji}} + \text{(error)}.
\end{equation*}
Using the integration by parts formula from Corollary~\ref{c:intbp} we therefore obtain
\begin{multline} \label{e:exp_sketch}
\E[1+zm] = \frac{d}{N^4 (d - 1)^{1/2}} \sum_{ijkl} \E \qb{D_{ij}^{kl} G_{ji}}
\\
+ \frac{1}{N^4 d^3 (d - 1)^{1/2}} \sum_{ijklrstu} \E \qbb{A_{ir} A_{ks} A_{jt} A_{lu} \pB{(D_{ij}^{kl} G_{ji}) (A + \xi_{ik}^{rs} + \xi_{jl}^{tu}) - (D_{ij}^{kl} G_{ji})(A)}}
+ \text{(error)}.
\end{multline}
For the first term on the right-hand side of \eqref{e:exp_sketch}, we use the Taylor expansion \eqref{e:D-expand} to expand $D_{ij}^{kl} G_{ji}$ as a polynomial in the entries of $G$, up to a small error. The leading term is $-(d - 1)^{-1/2} G_{jj} G_{ii}$, and it yields the first term of \eqref{e:Q_a}. The other terms are polynomials that either contain off-diagonal entries of $G$ or a higher order. Similarly, for the second term of \eqref{e:exp_sketch}, we keep on reapplying inductively the integration by parts formula from Corollary \ref{c:intbp}, and expand all discrete derivatives using Taylor's formula \eqref{e:D-expand}.

This procedure results in a proliferation of terms that contain products of factors of the form
\begin{equation} \label{e:sketch_term}
\frac{1}{d^{\fo/2}}\frac{1}{N^{b+c}d^b}\sum_{\bf m}\sum_{\bf ij}\bE\left[\prod_{a=1}^bA_{i_aj_a}F_{\bf ijm}\right]
\end{equation}
where $F$ is a polynomial in the entries of $G$. Each application of Corollary \ref{c:intbp} yields a main term (second term on the right-hand side of \eqref{e:intbp}) with no factors of $A$ and another term (first term on the right-hand side of \eqref{e:intbp}) that is (after Taylor expansion of the discrete derivatives) of \emph{higher order}, in the sense of both the degree of $F_{\bf ijm}$ and the power of $d^{-1/2}$ in front of it. Any main term that contains one or more off-diagonal can be shown to either vanish or be small enough. Hence, we only need to keep track of terms of the form \eqref{e:sketch_term} with $b = 0$ and $\chi_F = 1$
 (recall Definition~\ref{d:evaluation}).

Such a term is in general not a polynomial in $m$; consider for instance the term $\frac{1}{N} \sum_{i} \E[G_{ii}^2]$. An important ingredient of our argument is to rewrite such a term as a corresponding polynomial in $m$, up to higher order terms; for instance,
\begin{equation*}
\frac{1}{N} \sum_{i} \E[G_{ii}^2] = \frac{1}{N^2} \sum_{ij} \E[G_{ii} G_{jj}] + \text{(higher order)} + \text{(error)},
\end{equation*}
whereby $\frac{1}{N^2} \sum_{ij} \E[G_{ii} G_{jj}] = \E[m^2]$.
To explain how this works, consider some polynomial $X_i$ in the Green function entries (think of e.g.\ $X_i = G_{ii}$). We want to replace $\E[G_{ii} X_i]$ with $\E [G_{jj} X_i]$ up to higher order terms and small errors. From \eqref{e:GHexp} we get the equations
\begin{equation*}
1 - \frac{1}{N} + z G_{ii} = (HG)_{ii} , \qquad 1 - \frac{1}{N} + z G_{jj} = (HG)_{jj}.
\end{equation*}
Multiplying the first by $G_{jj} X_i$ and the second by $G_{ii} X_i$ and taking the difference, we obtain
\begin{align*}
(G_{ii} - G_{jj}) X_i &= \pb{G_{ii} (HG)_{jj} - G_{jj} (HG)_{ii}} X_i + \text{(error)}
\\
&= \frac{1}{(d - 1)^{1/2}} \sum_k \pb{A_{jk} G_{ii} G_{kj} - A_{ik} G_{ki} G_{jj}} X_i + \text{(error)}.
\end{align*}
We take the expectation, average over $ij$, and apply the integration by parts formula of Corollary~\ref{c:intbp} twice. Recalling that $\sum_k G_{ik} = 0$, we therefore obtain
\begin{multline*}
\frac{1}{N^2} \sum_{ij} \E [(G_{ii} - G_{jj}) X_i]
\\
= \frac{d}{(d - 1)^{1/2} N^5} \sum_{ijkrs} \E \qb{D_{jk}^{rs}( G_{ii} G_{kj} X_i) - D_{ik}^{rs}(G_{ki} G_{jj} X_i)} + \text{(higher order)} + \text{(error)}.
\end{multline*}
We expand the discrete derivatives using \eqref{e:D-expand}. This yields
\begin{align*}
&\mspace{-40mu}
\frac{1}{N^2} \sum_{ij} \E [(G_{ii} - G_{jj}) X_i]
\\
&=
\frac{d}{(d - 1) N^5} \sum_{ijkrs} \E \qb{\del_{jk}^{rs}( G_{ii} G_{kj} X_i) - \del_{ik}^{rs}(G_{ki} G_{jj} X_i)} + \text{(higher order)} + \text{(error)}
\\
&=
\frac{d}{(d - 1) N^5} \sum_{ijkrs} \E \qb{- G_{ii} G_{kk} G_{jj}X_i + G_{kk} G_{ii} G_{jj} X_i} + \text{(higher order)} + \text{(error)},
\end{align*}
where we used that all other terms are small because they contain off-diagonal terms. Thus, the leading order terms cancel exactly. We may therefore continue this expansion iteratively on the terms (higher order), which will stop after a finite number of steps, as the order, and hence the power of $d^{-1/2}$, increases at each iteration. 

We conclude this informal discussion by noting that the algorithm sketched above that generates the polynomial $Q_{\fa}$, while explicit, is quite complicated and tracking the actual coefficients of \eqref{e:Q_a} that it generates is difficult. This issue will be addressed in the next section, where, instead of tracking the coefficients explicitly, we characterize them indirectly by showing that up to a small error term the Stieltjes transform of the Kesten--McKay law $\md$ is a root of $P_{\fa}(z, \cdot)$, which will imply that the coefficients of $P_{\fa}$ are close to those of \eqref{e:md-sce}. This concludes the outline of the proof of Proposition \ref{p:DSE}.

The rest of this section is devoted to the proof of Proposition \ref{p:DSE}. Throughout, we fix an integer $\fa \geq 1$. The spectral parameter $z$ is always taken in the set $\bld D$, and our estimates are uniform in $z$. We shall always work under Assumption \ref{ass:main}.

\subsection{Estimates for moments of the Green's function}

We begin with a definition of a family of fundamental terms that form the backbone of our expansion. They are classified by the order of the variable $d^{-1/2}$ and the degree of the polynomial in the entries of $G$. Both quantities are important to keep track of. The former because it will allow us to stop the recursive application of identities yielding high order terms after a fixed number, $\fa$, of steps, up to an error term of order $d^{-\fa/2}$. The latter is important to ensure that the polynomial in $m$ that we shall ultimately generate will have a large enough degree.

For the following statements, we recall
that for multi-indices ${\bf i} \in \qq{N}^b$ and ${\bf j} \in \qq{N}^c$, we denote by ${\bf i} {\bf j}\in \qq{N}^{b+c}$ their concatenation
(and analogously for ${\bf ijk}$).
Together with this notation, we recall from Definition~\ref{d:evaluation} that,
for a polynomial $F$ in $(2b+c)^2$ variables and $\mathbf i, \mathbf j \in \qq{N}^b$ and $\mathbf m \in \qq{N}^c$,
we write $F_{\bf ijm}$ for its evaluation in the Green's function entries,
and that for a polynomial $U$ in $m, \bar m$, we abbreviate $U(m, \bar m)$ by $U$.

\begin{definition} \label{d:order_terms}
\begin{enumerate}
\item
For $\fo \in \N$, we define the expressions 
\begin{align}\label{e:newterm}
\cal T_{\fo}(F,U) \deq \frac{1}{d^{\fo/2}}\frac{1}{N^{b+c}d^b}\sum_{\bf m}\sum_{\bf ij}\bE\left[\prod_{a=1}^bA_{i_aj_a}F_{\bf ijm}(A) U(A)\right].
\end{align}
\item
For a given polynomial $U = U(m, \bar m)$ and integers $\fo, \fd \in \N$, we use the symbol
\begin{equation*}
\cal T_{\fo, \fd}(U)
\end{equation*}
to denote a finite linear combination of terms of the form $\alpha \cal T_{\tilde \fo}(\tilde F, U)$, where $\tilde \fo \geq \fo$, $\deg(\tilde F) \geq \fd$, and $\alpha = C \pb{d/(d-1)}^r$ for some $r \in \Z/2$ and a deterministic constant $C \in \R$.
\end{enumerate}
\end{definition}

 Note that $\cal T_\fo$ depends on $\fo\in\N$ only by a multiplicative factor $d^{-\fo/2}$.
From \eqref{e:Lambound}, we have $|F_{\bf ijm}|\prec 1$, and thus by Lemma \ref{lem:prec_exp} we find the a priori estimate
\begin{align} \label{e:T_estimate}
|\cal T_{\fo}(F,U)| \prec \frac{1}{d^{\fo/2}}\frac{1}{N^{b+c}d^b}\sum_{\bf m}\sum_{\bf ij}\bE\left[\prod_{a=1}^bA_{i_aj_a}|U(A)|\right]= \frac{\bE[|U(A)|]}{d^{\fo/2}}.
\end{align}
Proposition~\ref{p:DSE} will follow from much more precise estimates that follow from an inductive application of the following proposition,
which extracts the leading term from \eqref{e:newterm} and shows that it can be expressed as a monomial in the trace $m = \frac1N\Tr G$
rather than the individual Green's function entries.
To prove Proposition~\ref{p:DSE}, we only need the special case $U=1$;
we allow for a general $U$ for later use in Section~\ref{sec:P-moments}.

Aside from the notation introduced in Definition \ref{d:order_terms}, recall
that $\cal C(F,A)$ was defined in Definition~\ref{def:C},
and that $F_{\bf i}$, $U$, and $\chi_F$ were defined in Definition~\ref{d:evaluation}.
 Also recall $\Lambda_o$ from \eqref{Lambda_assumptions}.

\begin{proposition}\label{p:reduceterm}
Fix $\fo \in \N$. Let $F$ be a fixed monic monomial in $(2b+c)^2$ abstract variables, with degree $\deg(F)$.
Then
\begin{align}\label{e:reduceterm}\begin{split}
&\phantom{{}={}}\frac{1}{d^{\fo/2}}\frac{1}{N^{b+c}d^b}\sum_{\bf m}\sum_{\bf ij}\bE\left[\prod_{a=1}^bA_{i_aj_a}F_{\bf ijm}U\right]
=\frac{\chi_{F}\bE[m^{\deg(F)}U]}{d^{\fo/2}} + \cal T_{\fo + 1, \deg(F) + 1}(U) 
  \\
&\phantom{{}={}}+\frac{1}{d^{\fo/2}}\OO_\prec\left( \frac{\bE[\Im[m]|U|]}{N\eta}+\frac{\bE[|U|]}{d^{\fa/2}} +\frac{d\bE[ \cal C(U,A)]}{N} + { \Lambdao^{2-\chi_F}} \max_{s + \bar s\geq 1} \E \qBB{|U^{(s, \bar s)}| \pbb{\frac{\Im[m]}{N\eta}}^{s + \bar s}} \right).
\end{split}\end{align}
\end{proposition}

To prove Proposition~\ref{p:reduceterm}, we shall use the following claims.
The first claim states that the averages of monomials with more than one off-diagonal Green's function terms
are subleading.

\begin{claim}\label{c:two-off}
Let $F$ be a fixed monomial in $b^2$ abstract variables with at least two off-diagonal entries.
Then
\begin{equation}\label{e:two-off}
\left|\frac{1}{N^{b}}\sum_{\bf i}\bE\left[F_{\bf i}U\right]\right|\prec\frac{\bE[\Im[m]|U|]}{N\eta}.
\end{equation}
\end{claim}

\begin{proof}
Using $|G_{xy}| \prec 1$ and $U \prec 1$ from \eqref{e:Lambound} to bound all except the two off-diagonal factors of $G$ in $F_{\bf i}$ and then using the Cauchy--Schwarz inequality,
we have
\begin{align}\begin{split}\label{e:two-off2}
\left|\frac{1}{N^{b}}\sum_{\bf i}\bE\left[F_{\bf i}U\right]\right|
\prec\frac{1}{N^{b}}\sum_{\bf i}\bE\left[|G_{i_1i_2}|^2|U|\right]
= \frac{1}{N^{2}}\sum_{i_1,i_2}\bE\left[|G_{i_1i_2}|^2|U|\right]
= \frac{\bE[\Im[m]|U|]}{N\eta},
\end{split}\end{align}
where the last equality is the Ward identity \eqref{e:Ward}.
\end{proof}

The following claim separates the leading order term of $\cal T_{\fo}(F,U)$ plus other terms of higher order and much small error terms. It says that, to leading order, each factor of $A$ in \eqref{e:newterm} can be replaced with its expectation $d/N$.

\begin{claim}\label{c:keyexp1}
Fix $\fo \in \N$. Let $F$ be a fixed monomial in $(2b+c)^2$ abstract variables and let $U$ be a fixed polynomial in $m$.
Then
\begin{align}\label{e:keyexp1}
\begin{split}
&\phantom{{}={}} \frac{1}{d^{\fo/2}}\frac{1}{N^{b+c}d^{b}}\sum_{\bf m}\sum_{\bf ij}\bE\left[\prod_{a=1}^bA_{i_aj_a}F_{\bf ijm}U\right]\\
&=\frac{1}{d^{\fo/2}}\frac{1}{N^{2b+c}}\sum_{\bf m}\sum_{\bf ij}\bE\left[F_{\bf ijm}U\right]+
 \cal T_{\fo+1, \deg(F)+1}(U)
\\
&\phantom{{}={}}+\frac{1}{d^{\fo/2}}\OO_{\prec}\left(\frac{\bE[|U|]}{d^{\fa/2}} +\frac{d\bE[ \cal C(U,A)]}{N} + \frac{\Lambdao^{1-\chi_F}}{d^{1/2}} \max_{s + \bar s\geq 1} \E \qBB{|U^{(s, \bar s)}| \pbb{\frac{\Im[m]}{N\eta}}^{s + \bar s}}\right).
\end{split}\end{align}
\end{claim}

\begin{proof}
We prove the statement for $\fo=0$; the general statement follows by multiplying both sides by $1/d^{\fo/2}$. 
By Corollary~\ref{c:intbp}, the left-hand side of \eqref{e:keyexp1} is
\begin{align}\begin{split}
    &\phantom{{}={}}
    \frac{1}{N^{2b+c}}\sum_{\bf m}\sum_{\bf ij}\bE\left[F_{\bf ijm}U(A)\right]\\
    &+
    \frac{1}{N^{2b+c}d^{2b}}\sum_{\bf m}\sum_{\bf ijkl}\bE\left[\prod_{a=1}^{b}A_{i_ak_a}A_{j_al_a}\left(F_{\bf ijm}U\left(A+\sum_{a=1}^b\xi_{i_aj_a}^{k_al_a}\right)-F_{\bf ijm}U(A)\right)\right]+\OO_\prec\left(\frac{d\bE[ \cal C(U,A)]}{N}\right),
\label{e:intbp2}\end{split}\end{align}
where we used Remark \ref{r:basicestimates} to estimate $\abs{F_{\bf ijm}} \prec 1$ as well as Lemma \ref{lem:prec_exp}.
For the second term in \eqref{e:intbp2}, we use Claim \ref{c:taylorexp} with $F = F_{\bf ijm}$. The term resulting from the first term on the right-hand side of \eqref{e:expFU} gives rise to $\cal T_{\fo+1, \deg(F)+1}(U)$. For the error terms, we note that, by Remark \ref{r:basicestimates} we have
\begin{align}
\frac{1}{N^{2b+c}d^{2b}}\sum_{\bf m}\sum_{\bf ijkl}\prod_{a=1}^{b}A_{i_ak_a}A_{j_al_a}\left|F_{\bf ijm}\left(A+\sum_{a=1}^b\xi_{i_aj_a}^{k_al_a}\right)\right |\prec \frac{d}{N}+\Lambdao^{1-\chi_F} \leq 2 \Lambdao^{1-\chi_F},
\end{align}
where in the last step we used Assumption \ref{ass:main}. In summary, the error terms resulting from the application of Claim \ref{c:taylorexp} to the second term of \eqref{e:intbp2} are bounded by
\begin{equation*}
\OO_\prec\left(\frac{\bE[|U|]}{d^{\fa/2}} + \frac{\Lambdao^{1-\chi_F}}{d^{1/2}} \max_{s + \bar s\geq 1} \E \qBB{|U^{(s, \bar s)}| \pbb{\frac{\Im[m]}{N\eta}}^{s + \bar s}}\right).
\end{equation*}
The proof is therefore complete.
\end{proof}

The following claim is a decoupling argument: when averaging over an index $i$ that appears in a diagonal Green's function entry $G_{ii}$ and possibly many other places as well, up to some error terms we can replace $G_{ii}$ with $G_{i'i'}$, where $i'$ is a new summation index that appears in no other place, over which we take the average. For example, this allows us to convert an expression of the form $\frac{1}{N} \sum_i (G_{ii})^2$ to a polynomial in $m$ of the form $m^2$.

\begin{claim}\label{c:keyexp2}
Fix $\fo \in \N$. Let $F$ be a fixed monomial in $(1+c)^2$ abstract variables.
Then
\begin{equation}\begin{split}\label{e:keyexp2}
&\phantom{{}={}}\frac{1}{d^{\fo/2}}\frac{1}{N^{2+c}}\sum_{ii'\bf m}\bE[G_{ii}F_{i\bf m} U ]\\
&=\frac{1}{d^{\fo/2}}\frac{1}{N^{2+c}}\sum_{ii'\bf m}\bE[G_{i'i'}F_{i\bf m}U]+  \cal T_{\fo + 1, \deg(F) + 3}(U) 
\\
&\phantom{{}={}}+\frac{1}{d^{\fo/2}}\OO_\prec\left( \frac{\bE[\Im[m]|U|]}{N\eta}+\frac{\bE[|U|]}{d^{\fa/2}}+ \frac{d^{3/2}\Lambdao\bE[\cal C(U,A)]}{N}+{ \Lambdao^{2-\chi_F}}\max_{s + \bar s\geq 1} \E \qBB{|U^{(s, \bar s)}| \pbb{\frac{\Im[m]}{N\eta}}^{s + \bar s}}\right).
\end{split}\end{equation}
\end{claim}

\begin{proof} 
We prove the statement for $\fo=0$, the general statement follows by multiplying both sides by $1/d^{\fo/2}$. By the definition of the Green's function \eqref{e:GHexp}, we have
\begin{align}
\label{e:GHi'term}&\left(1-\frac{1}{N}\right)=-zG_{i'i'}+\sum_{j=1}^N\frac{A_{i'j}G_{ji'}}{\sqrt{d-1}},\\
\label{e:GHiterm}&\left(1-\frac{1}{N}\right)=-zG_{ii}+\sum_{j=1}^N\frac{A_{ij}G_{ji}}{\sqrt{d-1}}.
\end{align}
Multiplying \eqref{e:GHi'term} and \eqref{e:GHiterm} by $G_{ii}F_{i \bf m}U$ and $G_{i'i'}F_{i \bf m}U$ respectively,
averaging over the indices, and then taking the difference, we get 
\begin{align}\begin{split}\label{e:difterm}
&\frac{1}{N^{2+c}}\sum_{ii' \bf m}\bE[G_{ii}F_{i\bf m}U]
=\frac{1}{N^{2+c}}\sum_{ii'\bf m}\bE[G_{i'i'}F_{i\bf m}U]+\OO_{\prec}\left(\frac{\bE[|U|]}{N}\right)\\
&+\frac{1}{N^{2+c}(d-1)^{1/2}}\sum_{ii'j\bf m}
\left(\bE[A_{i'j}G_{ji'}G_{ii}F_{i\bf m}U]
-\bE[A_{ij}G_{ji}G_{i'i'}F_{i\bf m}U]\right),
\end{split}\end{align}
where we used that $|G_{ii}F_{i\bf m}|\prec 1$ and $|G_{i'i'}F_{i\bf m}|\prec 1$.
We shall show that the difference of the two terms on the right-hand side of \eqref{e:difterm} is of order $\fo$ greater than $0$, up to small error terms.

Using Corollary~\ref{c:intbp} we find
\begin{align}\begin{split}\label{e:difterm1}
&\phantom{{}={}}\frac{1}{N^{2+c}(d-1)^{1/2}}\sum_{ii'j\bf m}
\bE[A_{i'j}G_{ji'}G_{ii}F_{i\bf m}U]\\
&=\frac{1}{N^{3+c}d(d-1)^{1/2}}\sum_{ii'jkl\bf m}
\bE[A_{i'k}A_{jl}D_{i'j}^{kl}(G_{ji'}G_{ii}F_{i\bf m}U)]+\OO_{\prec}\left(\frac{d^{3/2}\Lambdao\bE[\cal C(U,A)]}{N}\right),
\end{split}\end{align}
where we used that $\sum_{i'} G_{ji'}=0$, so that the main term in \eqref{e:intbp} vanishes, and from $\abs{G_{ji'}}$ we gain an off-diagonal factor that is estimated by $\Lambdao$. 
For the first term on the right-hand side of \eqref{e:difterm1}, by the discrete product rule \eqref{e:D-product} and Claim \ref{c:DGDm},
\begin{align}\begin{split}\label{e:difterm1.5}
&\phantom{{}={}}D_{i'j}^{kl}(G_{ji'}G_{ii}F_{i\bf m}U)
=D_{i'j}^{kl}(G_{ji'}G_{ii}F_{i\bf m})U
+G_{ji'}G_{ii}F_{i\bf m}D_{i'j}^{kl}(U)
+D_{i'j}^{kl}(G_{ji'}G_{ii}F_{i\bf m})D_{i'j}^{kl}(U)\\
&=D_{i'j}^{kl}(G_{ji'}G_{ii}F_{i\bf m})U
+\OO_\prec\left(\left|G_{ji'}G_{ii}F_{i\bf m}
+D_{i'j}^{kl}(G_{ji'}G_{ii}F_{i\bf m})\right|\frac{1}{\sqrt{d}}\max_{s + \bar s\geq 1} |U^{(s, \bar s)}| \pbb{\frac{\Im[m]}{N\eta}}^{s + \bar s}\right).
\end{split}\end{align}
We notice that $G_{ji'}G_{ii}F_{i\bf m}$ contains at least $2-\chi_F$ off-diagonal entries and $D_{i'j}^{kl}(G_{ji'}G_{ii}F_{i\bf m})$ contains at least $1-\chi_F$ off-diagonal entries. Thus, by plugging \eqref{e:difterm1.5} into \eqref{e:difterm1}, we get
\begin{align}\begin{split}\label{e:difterm2}
&\phantom{{}={}}\frac{1}{N^{2+c}(d-1)^{1/2}}\sum_{ii'j\bf m}
\bE[A_{i'j}G_{ji'}G_{ii}F_{i\bf m}U]\\
&=\frac{1}{N^{3+c}d(d-1)^{1/2}}\sum_{ii'jkl\bf m}
\bE[A_{i'k}A_{jl}D_{i'j}^{kl}(G_{ji'}G_{ii}F_{i\bf m})U]\\
&\phantom{{}={}}+\OO_{\prec}\left(\frac{d^{3/2}\Lambdao\bE[\cal C(U,A)]}{N}+\Lambda_o^{2-\chi_F}\max_{s + \bar s\geq 1} \E \qBB{|U^{(s, \bar s)}| \pbb{\frac{\Im[m]}{N\eta}}^{s + \bar s}}\right)\\
&=\frac{1}{N^{3+c}d}\sum_{n=1}^{\fa}\frac{1}{n!(d-1)^{(n+1)/2}}\sum_{ii'jkl\bf m}\bE[A_{i'k}A_{jl}(\del_{i'j}^{kl})^n(G_{ji'}G_{ii}F_{i\bf m})U]\\
&\phantom{{}={}}+\OO_{\prec}
\left(
\frac{\bE[|U|]}{d^{\fa/2}}+\frac{d^{3/2}\Lambdao\bE[\cal C(U,A)]}{N}
+\Lambda_o^{2-\chi_F} \max_{s + \bar s\geq 1} \E \qBB{|U^{(s, \bar s)}| \pbb{\frac{\Im[m]}{N\eta}}^{s + \bar s}}
\right),
\end{split}
\end{align}
where the last step follows by Taylor expansion, as in Claim \ref{c:taylorexp}.
The remaining derivative $(\del_{i'j}^{kl})^n(G_{ji'}G_{ii}F_{i\bf m})$ is again a polynomial in $\{G_{xy}\}_{x,y\in ii'jkl\bf m}$,
and thus this term is in the form $\cal T_{n - 1, \deg(F) + 2 + n}(U)$. Treating the terms $n = 1$ and $n \geq 2$ separately, we get
\begin{align}\label{e:difterm3}
\begin{split}
  &\frac{1}{N^{3+c}d(d-1)}\sum_{ii'jkl\bf m}\bE[A_{i'k}A_{jl}\del_{i'j}^{kl}(G_{ji'}G_{ii}F_{i\bf m})U]+
 \cal T_{1, \deg(F) + 4}(U) 
  \\
 &+ \OO_{\prec}\left(\frac{\bE[|U|]}{d^{\fa/2}}+\frac{d^{3/2}\Lambdao\bE[\cal C(U,A)]}{N}
 + \Lambda_o^{2-\chi_F} \max_{s + \bar s\geq 1} \E \qBB{|U^{(s, \bar s)}| \pbb{\frac{\Im[m]}{N\eta}}^{s + \bar s}}
 \right).
\end{split}
\end{align}
By Claim \ref{c:keyexp1}, the first term in \eqref{e:difterm3} can be expanded as
\begin{align}
\begin{split}\label{e:difterm4}
  &\phantom{{}={}}\frac{1}{N^{3+c}d(d-1)}\sum_{ii'jkl\bf m}\bE[A_{i'k}A_{jl}\del_{i'j}^{kl}(G_{ji'}G_{ii}F_{i\bf m})U] 
  \\
&=\frac{d}{N^{5+c}(d-1)}\sum_{ii'jkl\bf m}\bE[\del_{i'j}^{kl}(G_{ji'}G_{ii}F_{i\bf m})U]
+  \cal T_{1, \deg(F) + 3}(U)
\\
&\phantom{{}={}}+\OO_{\prec}\left(\frac{\bE[|U|]}{d^{\fa/2}} +\frac{d\bE[ \cal C(U,A)]}{N} + \frac{\Lambdao^{1-\chi_F}}{d^{1/2}}\max_{s + \bar s\geq 1} \E \qBB{|U^{(s, \bar s)}| \pbb{\frac{\Im[m]}{N\eta}}^{s + \bar s}}\right).
\end{split}
\end{align}
Moreover, $\del_{i'j}^{kl}(G_{ji'}G_{ii}F_{i\bf m})=G_{ji'}\del_{i'j}^{kl}(G_{ii}F_{i\bf m})+\del_{i'j}^{kl}(G_{ji'})G_{ii}F_{i\bf m}$.
The first term $G_{ji'}\del_{i'j}^{kl}(G_{ii}F_{i\bf m})$ contains at least two off-diagonal Green's function entries.
Thus, by the Claim~\ref{c:two-off}, we have
\begin{align}\label{e:difterm5}
\left|\frac{d}{N^{5+c}(d-1)}\sum_{ii'jkl\bf m}\bE[G_{ji'}\del_{i'j}^{kl}(G_{ii}F_{i\bf m})U]\right|
\prec \frac{\bE[\Im[m]|U|]}{N\eta}.
\end{align}
To analyse the second term, we write $\del_{i'j}^{kl}(G_{ji'})
=-(G\xi_{i'j}^{kl}G)_{ji'}=-G_{jj}G_{i'i'}+G_{jj}G_{li'}+G_{jk}G_{i'i'}-G_{ji'}G_{ji'}-G_{jk}G_{li'}-G_{jl}G_{ki'}+G_{ji'}G_{ki'}+G_{jl}G_{ji'}$.
Since the row and column sums of $G$ are zero,  all but the first and fourth terms vanish when taking the average over the indices $i'jkl$. The fourth term has
two off-diagonal Green's function entries and can therefore be estimated using Claim~\ref{c:two-off}.
Thus we have
\begin{align}\begin{split}\label{e:difterm6}
&\phantom{{}={}}\frac{d}{N^{5+c}(d-1)}\sum_{ii'jkl\bf m}\bE[\del_{i'j}^{kl}(G_{ji'})G_{ii}F_{i\bf m}U]\\
&=
-\frac{d}{N^{5+c}(d-1)}\sum_{ii'jkl\bf m}\bE[G_{jj}G_{i'i'}G_{ii}F_{i\bf m}U]
+\OO_\prec\left( \frac{\bE[\Im[m]|U|]}{N\eta}\right).
\end{split}\end{align}
By combining the estimates \eqref{e:difterm2}, \eqref{e:difterm3}, \eqref{e:difterm4} and \eqref{e:difterm5}, we get
\begin{align}\begin{split}\label{e:mainterm1}
&\phantom{{}={}}\frac{1}{N^{2+c}(d-1)^{-1/2}}\sum_{ii'j\bf m}
\bE[A_{i'j}G_{ji'}G_{ii}F_{i\bf m}U]\\
&= -\frac{d}{N^{5+c}(d-1)}\sum_{ii'jkl\bf m}\bE[G_{jj}G_{i'i'}G_{ii}F_{i\bf m}U]
+  \cal T_{1, \deg(F) + 3}(U)
\\
&\phantom{{}={}}+\OO_\prec\left( \frac{\bE[\Im[m]|U|]}{N\eta}+\frac{\bE[|U|]}{d^{\fa/2}}+ \frac{d^{3/2}\Lambdao\bE[\cal C(U,A)]}{N}+{ \Lambdao^{2-\chi_F}}\max_{s + \bar s\geq 1} \E \qBB{|U^{(s, \bar s)}| \pbb{\frac{\Im[m]}{N\eta}}^{s + \bar s}}\right).
\end{split}\end{align}
Analogously, repeating the above argument for the last term in \eqref{e:difterm}, we find
\begin{align}
\begin{split}\label{e:mainterm2}
&\phantom{{}={}}\frac{1}{N^{2+c}(d-1)^{1/2}}\sum_{ii'j\bf m}
\bE[A_{ij}G_{ji}G_{i'i'}F_{i\bf m}]\\
&=-\frac{d}{N^{5+c}(d-1)}\sum_{ii'jkl\bf m}\bE[G_{jj}G_{i'i'}G_{ii}F_{i\bf m}U]
+ \cal T_{1, \deg(F) + 3}(U)
\\
&\phantom{{}={}}+\OO_\prec\left( \frac{\bE[\Im[m]|U|]}{N\eta}+\frac{\bE[|U|]}{d^{\fa/2}}+ \frac{d^{3/2}\Lambdao\bE[\cal C(U,A)]}{N}+{ \Lambdao^{2-\chi_F}}\max_{s + \bar s\geq 1} \E \qBB{|U^{(s, \bar s)}| \pbb{\frac{\Im[m]}{N\eta}}^{s + \bar s}}\right).
\end{split}
\end{align}
Since the first terms on the right-hand sides of \eqref{e:mainterm1} and \eqref{e:mainterm2} are the same,
they cancel upon taking their difference, and 
our claim \eqref{e:keyexp2} follows by combining \eqref{e:difterm}, \eqref{e:mainterm1} and \eqref{e:mainterm2}. (Note that the error term on the right-hand side of \eqref{e:difterm} can be absorbed into the third error term of \eqref{e:keyexp2}.)
\end{proof}

\begin{proof}[Proof of Proposition \ref{p:reduceterm}]
We prove the statement for $\fo=0$, the general statement follows by multiplying both sides by $1/d^{\fo/2}$.
By Claim \ref{c:keyexp1}, we have
\begin{align}\begin{split}\label{e:monexp1}
&\phantom{{}={}}\frac{1}{N^{b+c}d^b}\sum_{\bf m}\sum_{\bf ij}\bE\left[\prod_{a=1}^bA_{i_aj_a}F_{\bf ijm}U\right]\\
&=\frac{1}{N^{2b+c}}\sum_{\bf m}\sum_{\bf ij}\bE\left[F_{\bf ijm}U\right]
+ \cal T_{1, \deg(F) + 1}(U)
\\
&\phantom{{}={}}+\OO_{\prec}\left(\frac{\bE[|U|]}{d^{\fa/2}} +\frac{d\bE[ \cal C(U,A)]}{N} + \frac{\Lambdao^{1-\chi_F}}{d^{1/2}}\max_{s + \bar s\geq 1} \E \qBB{|U^{(s, \bar s)}| \pbb{\frac{\Im[m]}{N\eta}}^{s + \bar s}}\right).
\end{split}\end{align}
We now estimate the first term on the right-hand side, distinguishing three cases.

\paragraph{Case 1} The monomial $F_{\bf ijm}$ has more than one off-diagonal Green's function factors. Then, by Claim~\ref{c:two-off},
\begin{align}\label{e:case1}
\frac{1}{N^{2b+c}}\sum_{\bf m}\sum_{\bf ij}\bE\left[F_{\bf ijm}U\right]\prec\frac{\bE[\Im[m]|U|]}{N\eta}.
\end{align}

\paragraph{Case 2} The monomial $F_{\bf ijm}$ contains exactly one off-diagonal Green's function factor.
Then, without loss of generality, we assume that
$
F_{\bf ijm}
=G_{m_1m_1}^rG_{m_1m_2} \tilde F_{{\bf ij}m_2\cdots m_c},
$
where $\tilde F_{{\bf ij}m_2\cdots m_c}$ is a monomial in terms of the Green's function entries $\{G_{xx}\}_{x\in\{i_1,j_1,\cdots,i_b,j_b,m_2,\cdots,m_c\}}$
and $r \in \N$.
If $r=0$ then $F_{\bf ijm}$ vanishes upon taking the average over $\bf m$ since $\sum_{m_1} G_{m_1m_2} = 0$.
For $r\geq 1$, we introduce new indices $m_1^1,m_1^2,\cdots,m_1^r$,
and repeatedly use Claim~\ref{c:keyexp2} to replace $G_{m_1m_1}^r$ by $G_{m_1^1m_1^1}G_{m_1^2m_1^2}\cdots G_{m_1^rm_1^r}$.
This way we obtain
\begin{align}\begin{split}\label{e:case2}
&\phantom{{}={}}\frac{1}{N^{2b+c}}\sum_{\bf m}\sum_{\bf ij}\bE\left[G_{m_1m_1}^rG_{m_1m_2}\tilde F_{{\bf ij}m_2 \cdots m_c}U\right]\\
&=\frac{1}{N^{2b+c+r}}\sum_{m_1^1, \cdots,m_1^r} \sum_{\bf m}\sum_{\bf ij}\bE\left[\prod_{a=1}^rG_{m_1^am_1^a}G_{m_1m_2}\tilde F_{{\bf ij}m_2\cdots m_c}U\right]
+ \cal T_{1, \deg(F) + 2}(U)
\\
&\qquad
+\OO_\prec\left( \frac{\bE[\Im[m]|U|]}{N\eta}+\frac{\bE[|U|]}{d^{\fa/2}}+ \frac{d^{3/2}\Lambdao\bE[\cal C(U,A)]}{N} +\Lambdao^2 \max_{s + \bar s\geq 1} \E \qBB{|U^{(s, \bar s)}| \pbb{\frac{\Im[m]}{N\eta}}^{s + \bar s}}\right)\\
&=  \cal T_{1, \deg(F) + 2}(U)
\\
&\qquad +\OO_\prec\left( \frac{\bE[\Im[m]|U|]}{N\eta}+\frac{\bE[|U|]}{d^{\fa/2}}+ \frac{d^{3/2}\Lambdao\bE[\cal C(U,A)]}{N} + \Lambdao^2 \max_{s + \bar s\geq 1} \E \qBB{|U^{(s, \bar s)}| \pbb{\frac{\Im[m]}{N\eta}}^{s + \bar s}}\right),
\end{split}\end{align}
where in the last equality we used that $\sum_{m_1} G_{m_1 m_2} = 0$.

\paragraph{Case 3} The monomial $F_{\bf ijm}$ contains only diagonal Green's function terms.
Then, by the same argument as in Case~2, we can repeatedly use Claim~\ref{c:keyexp2} to get
\begin{align}\begin{split}\label{e:case3}
&\phantom{{}={}}\frac{1}{N^{2b+c}}\sum_{\bf m}\sum_{\bf ij}\bE\left[F_{\bf ijm} U\right]\\
&=\bE[m^{\deg(F)}U]+ \cal T_{1, \deg(F) + 2}(U)
\\&\phantom{{}={}}+\OO_\prec\left( \frac{\bE[\Im[m]|U|]}{N\eta}+\frac{\bE[|U|]}{d^{\fa/2}}+ \frac{d^{3/2}\Lambdao\bE[\cal C(U,A)]}{N}+\Lambdao \max_{s + \bar s\geq 1} \E \qBB{|U^{(s, \bar s)}| \pbb{\frac{\Im[m]}{N\eta}}^{s + \bar s}}\right).
\end{split}\end{align}
By \eqref{e:monexp1}  and putting  the three cases, \eqref{e:case1}, \eqref{e:case2} and \eqref{e:case3}, together, the proposition follows.
\end{proof}

\subsection{Proof of Proposition~\ref{p:DSE}}

We prove the following proposition, from which Proposition~\ref{p:DSE} will follow easily by taking $U=1$. The general form of Proposition~\ref{p:constructQ} will be used in Section~\ref{sec:P-moments}.
\begin{proposition}\label{p:constructQ}
Suppose that Assumption~\ref{ass:main} holds.
For every fixed integer $\fa \geq 1$, there exists a polynomial, depending on $d$ and $\fa$ but not $N$, and whose degree depends on $\fa$ only,
\begin{equation*}
Q_\fa(w)=\frac{dw^2}{d-1} +  \frac{1}{d}\left(a_3 w^3 +a_4 w^4+\cdots \right),
\end{equation*}
with bounded coefficients $a_3,a_4,\dots$
such that, for any $z \in \mathbf D$,
\begin{align}\label{e:7firstterm}
\begin{split}
&\phantom{{}={}}\frac{1}{N^2d(d-1)^{1/2}}\sum_{i jkl}\bE[A_{ik}A_{jl}(D_{ij}^{kl}G_{ij})U] + 
\bE[Q_\fa U]\\
& = \OO_\prec \pa{ \frac{\bE[\Im[m]|U|]}{N\eta}+\frac{\bE[|U|]}{d^{\fa/2}} +\frac{d\bE[ \cal C(U,A)]}{N} + \frac{\Lambdao^2}{d^{1/2}}\max_{s + \bar s\geq 1} \E \qBB{|U^{(s, \bar s)}| \pbb{\frac{\Im[m]}{N\eta}}^{s + \bar s}}}.
\end{split}
\end{align}
\end{proposition}

\begin{proof}
By \eqref{e:exp1} and Lemma \ref{l:basicestimates},
\begin{align}\begin{split}\label{e:firstexp11}
&\phantom{{}={}}\frac{1}{N^2} \sum_{ijkl}\frac{\bE[A_{ik}A_{j l} (D_{ij}^{kl} G_{ij})U]}{d(d-1)^{1/2}} =\sum_{ijkl}\frac{\bE[A_{ik}A_{jl}(\del_{ij}^{kl}G_{ij})U]}{d(d-1)N^2}\\
&+\sum_{ijkl}\frac{\bE[A_{ik}A_{jl}((\del_{ij}^{kl})^2G_{ij})U]}{2d(d-1)^{3/2}N^2}
+ \cal T_{2,4}(U)
+\OO_\prec\left(\frac{\bE[|U|]}{d^{\fa/2}}\right).
\end{split}\end{align}
For the first term on the right-hand side of \eqref{e:firstexp11}, the derivative $\del_{ij}^{kl}G_{ij}$ is given by \eqref{e:dG}.
For the last four terms, $G_{il}G_{jj}+G_{ii}G_{kj}+G_{ik}G_{ij}+G_{ij}G_{lj}$, using that $\sum_j G_{ij}=0$, we have
\begin{align}
\sum_{ijkl}\frac{1}{d(d-1)N^2}\bE[A_{ik}A_{jl}(G_{il}G_{jj}+G_{ii}G_{kj}+G_{ik}G_{ij}+G_{ij}G_{lj})U]
=0.
\end{align}
For the term $-G_{ij}G_{ij}$, using \eqref{e:GGbd}, we have
\begin{align}
-\sum_{ijkl}\frac{1}{d(d-1)N^2}\bE[A_{ik}A_{jl}G_{ij}G_{ij}U]
=-\sum_{ij}\frac{d}{(d-1)N^2}\bE[G_{ij}G_{ij}U]\prec \frac{\bE[\Im[m]|U|]}{N\eta}.
\end{align}
For the term $-G_{il}G_{kj}$, we use $(H-z)G=P_\perp$ and \eqref{e:GGbd} to get
\begin{align}\begin{split}
&\phantom{{}={}}-\sum_{ijkl}\frac{1}{d(d-1)N^2}\bE[A_{ik}A_{jl}G_{il}G_{kj}U]
=-\frac{1}{d(d-1)N^2}\bE[\Tr(AGAG)U]\\
&=-\frac{1}{dN^2}\bE[\Tr(z^2G^2+2zG+P_\perp)U]\prec \frac{\bE[\Im[m]|U|]}{dN\eta}  + \frac{\E[\abs{U}]}{dN}  \prec \frac{\bE[\Im[m]|U|]}{dN\eta}.
\end{split}\end{align}
Summarizing, we can rewrite the first term on the right-hand side of \eqref{e:firstexp11} as
\begin{align}\begin{split}\label{e:secondterm222}
&\phantom{{}={}}\sum_{ijkl}\frac{\bE[A_{ik}A_{jl}(\del_{ij}^{kl}G_{ij})U]}{d(d-1)N^2}
=-\frac{d}{d-1}\bE[m^2 U]-\sum_{ijkl}\frac{\bE[A_{ik}A_{jl}G_{ik}G_{jl}U]}{d(d-1)N^2}+\OO_\prec\left(\frac{\bE[\Im[m]|U|]}{N\eta} \right).
\end{split}\end{align}
By Corollary~\ref{c:intbp}, $\sum_j G_{ij}=0$ and the trivial extension of the product rule \eqref{e:D-product} and \eqref{e:exp3} to differences in the direction $\xi_{ik}^{i'k'}+\xi_{jl}^{j'l'}$, the second term on the right-hand side of \eqref{e:secondterm222} is
\begin{align}\begin{split}\label{e:secondterm2222}
    &\phantom{{}={}}
   \sum_{ijkl}\frac{\bE[A_{ik}A_{jl}G_{ik}G_{jl}U]}{d(d-1)N^2}=\OO_\prec\left(\frac{d\bE[ \cal C(U,A)]}{N}\right)\\
    &\phantom{{}={}}+
    \frac{1}{d^3(d-1)N^4}\sum_{ii'jj'kk'll'}\bE\left[A_{ii'}A_{kk'}A_{jj'}A_{ll'} 
    \left(G_{ik}G_{jl}U\left(A+\xi_{ik}^{i'k'}+\xi_{jl}^{j'l'}\right)-G_{ik}G_{jl}U(A)\right)\right]\\
    &=\frac{1}{d^3 (d-1)  N^4}\sum_{ii'jj'kk'll'}\bE\left[A_{ii'}A_{kk'}A_{jj'}A_{ll'} 
   \left(G_{ik}G_{jl}\left(A+\xi_{ik}^{i'k'}+\xi_{jl}^{j'l'}\right)- G_{ik}G_{jl}(A)\right)U\right]\\
    &\phantom{{}={}}+\OO_\prec\left(\frac{d\bE[ \cal C(U,A)]}{N} + \frac{\Lambdao^2}{d^{1/2}} \max_{s + \bar s\geq 1} \E \qBB{|U^{(s, \bar s)}| \pbb{\frac{\Im[m]}{N\eta}}^{s + \bar s}}\right)\\
    &=\frac{1}{d^3(d-1)^{3/2}N^4}\sum_{ii'jj'kk'll'}\bE\left[A_{ii'}A_{kk'}A_{jj'}A_{ll'} 
   (\del_{ik}^{i'k'}+\del_{jl}^{j'l'})(G_{ik}G_{jl})U\right]\\
    &\phantom{{}={}}+ \cal T_{2,4}(U)
    +\OO_\prec\left(\frac{\bE[|U|]}{d^{(\fa+1)/2}} +\frac{d\bE[ \cal C(U,A)]}{N} + \frac{\Lambdao^2}{d^{1/2}}\max_{s + \bar s\geq 1} \E \qBB{|U^{(s, \bar s)}| \pbb{\frac{\Im[m]}{N\eta}}^{s + \bar s}}\right),
\end{split}\end{align}
where in the first equality, we used  Remark \ref{r:basicestimates} and that $G_{ik}G_{jl}$ contains two off-diagonal terms, and in the second equality we used Claim \ref{c:taylorexp}.
For the first term on the right-hand side of \eqref{e:secondterm2222}, we notice that $ (\del_{ik}^{i'k'}+\del_{jl}^{j'l'})(G_{ik}G_{jl})$ contains at least one off-diagonal term. By Proposition \ref{p:reduceterm} we get
\begin{align}\begin{split}\label{e:secondterm22222}
&\phantom{{}={}}\frac{1}{d^3(d-1)^{3/2}N^4}\sum_{ii'jj'kk'll'}\bE\left[A_{ii'}A_{kk'}A_{jj'}A_{ll'} 
   (\del_{ik}^{i'k'}+\del_{jl}^{j'l'})(G_{ik}G_{jl})U\right]
= \cal T_{2,4}(U)\\
   &\phantom{{}={}}+\OO_\prec\left(\frac{\bE[\Im[m]|U|]}{N\eta}+\frac{\bE[|U|]}{d^{(\fa+1)/2}} +\frac{d\bE[ \cal C(U,A)]}{N} + \frac{ \Lambdao^2}{d^{1/2}}\max_{s + \bar s\geq 1} \E \qBB{|U^{(s, \bar s)}| \pbb{\frac{\Im[m]}{N\eta}}^{s + \bar s}}\right).
\end{split}\end{align}
It follows by combining \eqref{e:secondterm222}, \eqref{e:secondterm2222}, and \eqref{e:secondterm22222} that
\begin{align}\begin{split}\label{e:secondterm222222}
&\phantom{{}={}}\sum_{ijkl}\frac{\bE[A_{ik}A_{jl} (\del_{ij}^{kl}G_{ij})U]}{d(d-1)N^2}
=-\frac{d}{d-1}\bE[m^2 U]+ \cal T_{2,4}(U)\\
&\phantom{{}={}}+\OO_\prec\left(\frac{\bE[\Im[m]|U|]}{N\eta}+\frac{\bE[|U|]}{d^{(\fa+1)/2}} +\frac{d\bE[ \cal C(U,A)]}{N} + \frac{\Lambdao^2}{d^{1/2}}\max_{s + \bar s\geq 1} \E \qBB{|U^{(s, \bar s)}| \pbb{\frac{\Im[m]}{N\eta}}^{s + \bar s}}\right).\end{split}\end{align}

For the second term on the right-hand side of \eqref{e:firstexp11}, we notice that $(\del_{ij}^{kl})^2G_{ij}$ contains at least one off-diagonal term. By Proposition \ref{p:reduceterm} we get
\begin{align}\begin{split}\label{e:secondterm22}
&\phantom{{}={}}\sum_{ijkl}\frac{\bE[A_{ik}A_{jl}((\del_{ij}^{kl})^2G_{ij})U]}{2d(d-1)^{3/2}N^2}
=\cal T_{2,3}(U)\\
   &\phantom{{}={}}+\OO_\prec\left(\frac{\bE[\Im[m]|U|]}{N\eta}+\frac{\bE[|U|]}{d^{(\fa+1)/2}} +\frac{d\bE[ \cal C(U,A)]}{N} + \frac{ \Lambdao^2}{d^{1/2}}\max_{s + \bar s\geq 1} \E \qBB{|U^{(s, \bar s)}| \pbb{\frac{\Im[m]}{N\eta}}^{s + \bar s}}\right).
\end{split}\end{align}

We plug \eqref{e:secondterm222222} and \eqref{e:secondterm22} into \eqref{e:firstexp11}, which yields
\begin{align} \label{e:p57_end}
\begin{split}
&\phantom{{}={}}\frac{1}{N^2} \sum_{ijkl}\frac{\bE[A_{ik}A_{j l} (D_{ij}^{kl} G_{ij})U]}{d(d-1)^{1/2}} =-\frac{d}{d-1}\bE[m^2 U]+ \cal T_{2,3}(U)\\
&\phantom{{}={}}+\OO_\prec\left(\frac{\bE[\Im[m]|U|]}{N\eta}+\frac{\bE[|U|]}{d^{\fa/2}} +\frac{d\bE[ \cal C(U,A)]}{N} + \frac{\Lambdao^2}{d^{1/2}}\max_{s + \bar s\geq 1} \E \qBB{|U^{(s, \bar s)}| \pbb{\frac{\Im[m]}{N\eta}}^{s + \bar s}}\right).
\end{split}
\end{align}

Next, we apply Proposition \ref{p:reduceterm} repeatedly to the term $\cal T_{2,3}(U)$ on the right-hand side of \eqref{e:p57_end}. For $\fo \geq 2$ and $\fd \geq 3$, Proposition \ref{p:reduceterm} yields
\begin{multline} \label{e:p53induct}
\cal T_{\fo, \fd}(U) = \frac{1}{d^{\fo/2}} \E [p U] + \cal T_{\fo + 1, \fd + 1}(U)
\\
+ \frac{1}{d}\OO_\prec\left(\frac{\bE[\Im[m]|U|]}{N\eta}+\frac{\bE[|U|]}{d^{\fa/2}} +\frac{d\bE[ \cal C(U,A)]}{N} + {\Lambdao}\max_{s + \bar s\geq 1} \E \qBB{|U^{(s, \bar s)}| \pbb{\frac{\Im[m]}{N\eta}}^{s + \bar s}}\right),
\end{multline}
where $p$ is a polynomial in $m$, with bounded coefficients depending only on $d$, of degree at least $\fd$. Applying \eqref{e:p53induct} to $\cal T_{2,3}(U)$ in \eqref{e:p57_end} and then iterating \eqref{e:p53induct} $\fa$ times concludes the proof, by \eqref{e:T_estimate}.
\end{proof}

\begin{proof}[Proof of Proposition~\ref{p:DSE}]
Let $Q_\fa(w)$ be as constructed in Proposition \ref{p:constructQ}. By taking the normalized trace on both sides of \eqref{e:GHexp}, we have
\begin{align}
1+zm=\frac{1}{N}+\sum_{ij}\frac{A_{ij}G_{ij}}{N(d-1)^{1/2}}=\frac{1}{N}+\sum_{i\neq j}\frac{A_{i j}G_{ij}}{N(d-1)^{1/2}},
\end{align}
where the last equality follows since $A_{ii}=0$. By Corollary~\ref{c:intbp1} with $F=G_{ij}$, we have for $i\neq j$,  $|F(A)|+\max_{kl}|F(A+\xi_{ij}^{kl})|\prec \Lambdao$ (since $\Lambdao \geq d^{-1/2}$, see also Remark \ref{r:basicestimates}), and 
\begin{align}\begin{split}\label{e:firstexp11b}
\phantom{{}={}}\bE[1+zm]
&=\frac{1}{N(d-1)^{1/2}}\sum_{i\neq j}\bE[A_{ij}G_{ij}]+\frac{1}{N}\\
&=\frac{1}{N^2} \sum_{ijkl}\left(\frac{1}{d(d-1)^{1/2}}\bE[A_{ik}A_{j l} D_{ij}^{kl} G_{ij}]\right)+\OO_{\prec}\left(\frac{d^{3/2}\Lambdao}{N}\right).
\end{split}\end{align}
By Proposition \ref{p:constructQ} with $U=1$, we have
\begin{align}\label{e:firstexp12}
\frac{1}{N^2} \sum_{ijkl}\left(\frac{1}{d(d-1)^{1/2}}\bE[A_{ik}A_{j l} D_{ij}^{kl} G_{ij}]\right)
 + 
\bE[Q_\fa]
\prec \frac{\bE[\Im[m]]}{N\eta}+\frac{1}{d^{\fa / 2 }} +\frac{d}{N} .
\end{align}
The claim \eqref{e:defP} follows from combining \eqref{e:firstexp11b} and \eqref{e:firstexp12}, and $\Lambdao\geq 1/\sqrt{d}$,
\begin{align}
\bE\left[1+zm+Q_\fa(m)\right]=\OO_\prec\left(\frac{1}{d^{\fa/2}}+ \frac{\bE[\Im[m]]}{N\eta}+\frac{d^{3/2}\Lambdao}{N}\right).
\end{align}
This finishes the proof of Proposition \ref{p:DSE}.
\end{proof}

\section{Identification of the self-consistent equation}
\label{sec:P-identification}

The algorithm that generates the polynomial $P_{\fa}$ from Proposition \ref{p:DSE} is explicit but quite complicated, so that explicitly tracking the resulting coefficients of $P_\fa$ is a hopeless task beyond the first few orders. In this section we characterize these coefficients (asymptotically)
as those of the power series $P_\infty(z,w)$ from \eqref{e:md-sce}, characterizing the Stieltjes transform $\md$ of the Kesten--McKay law.

\begin{proposition}\label{p:idfP}
Uniformly in $z\in \bC_+$,
the polynomial $P_\fa(z,w) = 1 + zw + Q_\fa(w)$ constructed in Proposition~\ref{p:DSE} satisfies
\begin{equation}
P_\fa(z, \md(z))= \OO(d^{-\fa/2}),
\end{equation}
where $\md$ is the Stieltjes transform of the Kesten--McKay law, given by \eqref{e:md}.
\end{proposition}

\begin{corollary}\label{c:Pproperty}
Let $P_\fa$ be the polynomial constructed in Proposition \ref{p:DSE}. Then
$P_{\fa}(z,w)-P_{\infty}(z,w)  = Q_\fa(w)-Q_\infty(w)$
is a power series in $w$, which converges on the whole complex plane. Each of its coefficients is of order $\OO(d^{-\fa/2})$.
\end{corollary}

\begin{definition}
 We write $P_\fa'(z,w) \deq \partial_w P_\fa(z,w)$, and similarly $P_\fa^{(k)}(z,w)$ for the $k$-th derivative in the variable $w$.
\end{definition}

\begin{corollary}\label{c:mdproperty}
The polynomial $P_\fa$ constructed in Proposition \ref{p:DSE} satisfies
\begin{align}\label{e:derPa}
  |P_\fa'(z,\md(z))|\asymp \sqrt{|\kappa|+\eta}+ \OO(d^{-\fa/2}),
  \quad
  P_\fa''(z,\md(z))=2+ \OO ( d^{-1/2}),
  \quad
  P_{\fa}'''(z,\md(z)) = \OO(1),
\end{align}
where $z=2+\kappa+\ri \eta$ or $z=-2-\kappa+\ri\eta$ for $\eta\leq \fK, -2\leq \kappa\leq \fK$, where the constant $\fK$ is from \eqref{e:D}. 
\end{corollary}

\subsection{Proof of Proposition~\ref{p:idfP}}

The main ingredient of the proof of the proposition is the \emph{ideal Green's function} $\wh G $ introduced in the following definition.

\begin{definition}
For $z \in \C_+$ define the ideal Green's function  $\wh G(z) = (\wh G_{ij}(z))_{i,j \in [N]}$ through
\begin{equation} \label{e:defiG}
\wh G_{ij} \deq
\md \, \delta_{ij} - \frac{\md \msc}{\sqrt{d - 1}} A_{ij}.
\end{equation}
\end{definition}

Note that while Proposition~\ref{p:idfP} is deterministic, $\wh G$ is random. However,
we remark that when $i$ and $j$ have distance at most $1$ in the random regular graph defined by $A$,
the ideal Green's function $\wh G_{ij}$ coincides with the Green's function of the infinite $d$-regular tree
\begin{equation*}
G^{\text{tree}}_{ij} = m_d \pbb{- \frac{\msc}{\sqrt{d-1}}}^{\dist(i,j)}
\end{equation*}
(see \cite[Proposition~5.1]{MR3962004}) for vertices $i$ and $j$ with the same distance, while it is set to be $0$ for all pairs of vertices $i,j$ with greater distance.
Because it agrees with the tree Green's function locally,
the ideal Green's function is a random matrix that shares key algebraic properties with the true Green's function of the random graph,
while its normalized trace is equal to the deterministic $\md$.
As a consequence, we shall show that $\md = \frac{1}{N} \Tr \wh G$ satisfies the same self-consistent equation as $m = \frac{1}{N} \Tr G$, up to small error terms, which will imply Proposition~\ref{p:idfP}.

\begin{lemma} \label{l:whG_est}
The ideal Green's function $\wh G$ has the following properties, uniformly in $z \in \C_+$.
\begin{enumerate}
\item For any $i \in \qq{N}$, $|\wh G_{ii}|=\OO(1)$, and for $i\neq j$, $|\wh G_{ij}|=\OO((d-1)^{-1/2})$;
\item $\sum_{i=1}^N \wh G_{ii}=N \md$;
\item $\sum_{i=1}^N |\wh G_{ij}|=\sum_{j=1}^N |\wh G_{ij}|=\OO(d^{1/2})$;
\item $\sum_{i=1}^N |\wh G_{ij}|^2=\sum_{j=1}^N |\wh G_{ij}|^2=\OO(1)$.
\item For any $i\in \qq{N}$,
\begin{align}\label{e:multiHG}
1=\pb{(H-z)\wh G}_{ii}=-z\wh G_{ii}+\sum_{j = 1}^N\frac{A_{ij}\wh G_{ij}}{\sqrt{d-1}}.
\end{align}
\end{enumerate}
\end{lemma}

\begin{proof}
  Properties (i)--(iv) are immediate consequences of the definition and $\md,\msc = \OO(1)$.
  The identity \eqref{e:multiHG} follows from \eqref{e:md}.
\end{proof}

Crucially, under perturbation by a switching $\xi_{ij}^{kl}$ of $A$,
the ideal Green's function $\wh G$ satisfies the same resolvent expansion as the real Green's function $G$.
The intuitive reason behind this behaviour is the following. Take two edges, $ij$ and $kl$, that are switchable.
Then, in the limit where $\dist(ij,kl)$ tends to infinity, the Green's functions $\wh G$ and $G^{\text{tree}}$, when restricted to the vertices $ijkl$, have identical behaviour under the switching $\xi_{ij}^{kl}$.

To state this precisely, we introduce some notation. For an $N \times N$ matrix $M = (M_{ij})_{i,j \in \qq{N}}$ and $V \subset \qq{N}$ , we denote by $M \vert_V \deq (M_{ij})_{i,j \in V}$ the submatrix induced by the set $V$. In particular, $A \vert_V$ is the adjacency matrix of the graph $A$ restricted to the vertex set $V$. We also frequently abbreviate subsets of $\qq{N}$ as $\{i,j\} \equiv ij$, and so on.

\begin{claim}\label{c:restrictR2}
Let $i,j,k,l$ be distinct indices such that $A|_{ijkl}=\Delta_{ik}+\Delta_{jl}$ (i.e.\ $\chi_{ik}^{jl}(A) = 1$).
Then
\begin{align}\label{e:restrictR2}
\wh{G}(A+\xi_{ij}^{kl})|_{ijkl}
= \left(\pb{\wh{G}(A)\vert_{ijkl}}^{-1}+(d-1)^{-1/2}\xi_{ij}^{kl} \vert_{ijkl}\right)^{-1}.
\end{align}
\end{claim}

\begin{proof}
By the definition of $\wh G(A)$, we have
\begin{align}\label{e:invtG1}
\wh{G}(A)|_{ijkl}=
\left[\begin{array}{cccc}
\md & 0& \frac{-\md \msc}{\sqrt{d-1}} & 0\\
0 & \md & 0 & \frac{-\md \msc}{\sqrt{d-1}}\\
\frac{-\md \msc}{\sqrt{d-1}} & 0 & \md & 0\\
0 & \frac{-\md \msc}{\sqrt{d-1}} & 0 & \md
\end{array}
\right]
=\left[\begin{array}{cccc}
\frac{1}{\msc} & 0 & \frac{1}{\sqrt{d-1}} & 0\\
0 & \frac{1}{\msc} & 0 & \frac{1}{\sqrt{d-1}}\\
\frac{1}{\sqrt{d-1}} & 0 & \frac{1}{\msc} & 0\\
0 & \frac{1}{\sqrt{d-1}} & 0 & \frac{1}{\msc}
\end{array}
\right]^{-1},
\end{align}
where in the last equality we used the identity
 \eqref{e:mdmscquad} and
$1+z\msc +\msc^2=0$.
Analogously,
\begin{align}\label{e:invtG2}
\wh{G}(A+\xi_{ij}^{kl})|_{ijkl}
=\left[\begin{array}{cccc}
\frac{1}{\msc} & \frac{1}{\sqrt{d-1}} & 0 & 0\\
\frac{1}{\sqrt{d-1}} & \frac{1}{\msc} & 0 & 0\\
0 & 0 & \frac{1}{\msc} & \frac{1}{\sqrt{d-1}}\\
0 & 0 & \frac{1}{\sqrt{d-1}} & \frac{1}{\msc}
\end{array}
\right]^{-1}
=
\pB{\pb{\wh{G}(A)|_{ijkl}}^{-1}+(d-1)^{-1/2} \xi_{ij}^{kl} \vert_{ijkl}}^{-1},
\end{align}
where the last equality follows from \eqref{e:invtG1}.
This proves \eqref{e:restrictR2}.
\end{proof}

By definition, the ideal Green's function of course has a trivial perturbation expansion under switchings since it is an affine linear function of $A$. However, when perturbed by the specific direction and magnitude of a switching, its behaviour can be written in a more complicated but more useful way, which matches precisely the corresponding resolvent expansion for switchings of the true Green's function. Indeed, as a consequence of Claim \ref{c:restrictR2}, we obtain the following resolvent expansion for the transformation by switching of the ideal Green's function.
\begin{claim}\label{c:restrictR3}
For $b,c \in \N$ consider tuples ${\bf i} = (i_1,\dots,i_b)$, ${\bf j} = (j_1,\dots,j_b)$,  ${\bf k} = (k_1,\dots,k_b)$, ${\bf l} = (l_1,\dots,l_b)$, ${\bf m} = (m_1,\dots,m_c)$ such that the indices $\bf i \bf j\bf k\bf l \bf m$ are distinct and $A|_{\bf i \bf j \bf k \bf l\bf m}=A|_{\bf m}+\sum_{a=1}^b\left(\Delta_{ i_a k_a }+\Delta_{ j_a l_a }\right)$. Then 
\begin{align}\label{e:restrictR3}
\pBB{\wh{G}\left(A+\sum_{a=1}^b\xi_{i_aj_a}^{k_al_a}\right)-\wh{G}(A)} \Biggr\vert_{\bf ijklm} =\pBB{\sum_{n\geq 1}\frac{1}{(d-1)^{n/2}}\wh G(A)\left(-\sum_{a=1}^b\xi_{i_aj_a}^{k_al_a}\wh G(A)\right)^n} \Biggr\vert_{\bf ijklm}.
\end{align}
\end{claim}

\begin{proof}
We shall prove that
\begin{equation}\label{e:restrictR4}
\wh{G}\left(A+\sum_{a=1}^b\xi_{i_aj_a}^{k_al_a}\right) \Biggr|_{\bf ijklm}
    =\left(\pB{\wh{G}(A)|_{\bf ijklm}}^{-1}+(d-1)^{-1/2}\sum_{a=1}^b\xi_{i_aj_a}^{k_al_a} \vert_{\bf ijklm} \right)^{-1},
\end{equation}
and the claim \eqref{e:restrictR3} then follows from the resolvent expansion, which converges for large enough $d$ because $\wh G_{ij} = \OO(1)$ by Lemma \ref{l:whG_est}.
By assumption, $A|_{\bf i \bf j \bf k \bf l\bf m}$ is a block matrix with $b+1$ blocks, indexed by $\{i_1,j_1,k_1,l_1\}, \{i_2,j_2,k_2,l_2\},\cdots, \{i_b, j_b, k_b,l_b\}$ and $\{m_1,\cdots,m_c\}$.
By our definition \eqref{e:defiG} of ideal Green's function $\wh G$, both $\wh{G}(A+\sum_{a=1}^b\xi_{i_aj_a}^{k_al_b})|_{\bf i \bf j \bf k\bf l \bf m}$ and $\wh{G}(A)|_{\bf i \bf j \bf k \bf l \bf m}$ have this same block structure.
Thus both sides of \eqref{e:restrictR4} vanish except for the submatrix indexed by $\{i_1,j_1,k_1,l_1\}, \{i_2,j_2,k_2,l_2\},\cdots, \{i_b, j_b, k_b,l_b\}$ and $\{m_1,\cdots,m_c\}$,
and \eqref{e:restrictR4} follows from \eqref{e:restrictR2}.
\end{proof}

Claim \ref{c:restrictR3} says that the ideal Green's function $\wh G$ has exactly the same behaviour under a switching $A \mapsto A+\sum_{a=1}^b\xi_{i_aj_a}^{k_al_a}$ as the true Green's function $G$. To make this precise, we need the following definition, which is to be compared with Definitions \ref{d:evaluation} and \ref{d:order_terms}.

\begin{definition} \label{d:idF}
\begin{enumerate}
\item
For any polynomial $F$ in $r^2$ abstract variables and $\bld i \in \qq{N}^r$, denote by $\wh F_{\bld i} = F(\{\wh G_{i_s i_t}\}_{i_s, i_t = 1}^r)$ the corresponding polynomial in the ideal Green's function entries. Hence, $\wh F_{\bld i}$ is obtained from $F_{\bld i}$ by replacing every factor $G_{ij}$ in $F_{\bld i}$ by $\wh G_{ij}$.
\item
For a polynomial $\wh F_{\bf ijm}$ in the ideal Green's function entries and $\fo \in \N$, we define $\wh {\cal T}_{\fo}(F, 1)$ as the right-hand side of \eqref{e:newterm} with $U = 1$ and $F_{\bf ijm}$ replaced with $\wh F_{\bf ijm}$.
\end{enumerate}
\end{definition}

With this definition, we can rewrite the right-hand side of \eqref{e:restrictR3} as
\begin{equation*}
\left. \left( \sum_{n\geq 1} \frac{1}{n!(d-1)^{n/2}} \left(\sum_{a=1}^b\del_{i_aj_a}^{k_al_a}\right)^n G(A)\right)^{\!\!\wh{\;}} \; \right \vert_{\bf ijklm}
\end{equation*}
Thus, Claim \ref{c:restrictR3} implies the following result, which is the analogue of Claim \ref{c:taylorexp}.

\begin{claim} \label{c:restrictR3id}
Let $\bf ijklm$ satisfy the assumptions of Claim \ref{c:restrictR3}, and let $F_{\bf ijm}$ be a polynomial in the Green's function entries. Then for any positive integer $\fa \geq 1$ we have
\begin{equation*} 
\wh F_{\bf ijm} \pBB{A + \sum_{a=1}^b\xi_{i_aj_a}^{k_al_a}} - \wh F_{\bf ijm} (A) = \left( \sum_{n=1}^{\fa-1} \frac{1}{n!(d-1)^{n/2}} \left(\sum_{a=1}^b\del_{i_aj_a}^{k_al_a}\right)^n F_{\bf ijm}(A)\right)^{\!\!\wh{\;}} + \OO \pbb{\frac{1}{d^{\fa/2}}}.
\end{equation*}
\end{claim}

Tha main ingredient in the proof of Proposition \ref{p:idfP} is the following result, which is the analogue of Proposition \ref{p:reduceterm} for the ideal Green's function.

\begin{proposition} \label{p:reduceterm2}
Fix $\fo \in \N$. Let $F$ be a fixed monic monomial in $(2b+c)^2$ abstract variables, with degree $\deg(F)$, and let $\wh F_{\bf ijm}$ be the corresponding monomial in the ideal Green's function entries from Definition \ref{d:idF}(i). Recall the definition of $\chi_F$ from Definition \ref{d:evaluation}.
Then
\begin{align}\label{e:reduceterm2}\begin{split}
\frac{1}{d^{\fo/2}}\frac{1}{N^{b+c}d^b}\sum_{\bf m}\sum_{\bf ij}\bE\left[\prod_{a=1}^bA_{i_aj_a}\wh F_{\bf ijm}\right]
=\frac{\chi_{F} \, \md^{\deg(F)} }{d^{\fo/2}} + \wh {\cal T}_{\fo + 1, \deg(F) + 1}(1) +  \frac{1}{d^{\fo/2}} \OO\left(\frac{1}{d^{\fa/2}}+\frac{d}{N}\right),
\end{split}\end{align}
where the term $\wh {\cal T}_{\fo + 1, \deg(F) + 1}(1)$ is equal to the corresponding term ${\cal T}_{\fo + 1, \deg(F) + 1}(1)$ from \eqref{e:reduceterm} with $G$ replaced by $\wh G$ (see Definition \ref{d:idF}(ii)).
\end{proposition}

\begin{proof}
The proof involves repeating the proof of Proposition \ref{p:reduceterm} with $U = 1$ almost verbatim, replacing Claim \ref{c:taylorexp} and \eqref{e:GHiterm} with Claim \ref{c:restrictR3id} with \eqref{e:multiHG} respectively.

More precisely, the proof of Proposition \ref{p:reduceterm} relies only on the following ingredients: $\sum_i G_{ij} = 0$, Claim \ref{c:two-off}, Claim \ref{c:keyexp1}, and Claim \ref{c:keyexp2}. Each of these has the following analogue for the ideal Green's function $\wh G$.
We replace $\sum_i G_{ij} = 0$ with Lemma \ref{l:whG_est}(iii). For Claim \ref{c:two-off}, we replace \eqref{e:two-off} with $U = 1$ by
\begin{equation*}
\left|\frac{1}{N^{b}}\sum_{\bf i}\bE\left[\wh F_{\bf i}\right]\right| = \OO \pbb{\frac{1}{N}},
\end{equation*}
as follows from Lemma \ref{l:whG_est}(iv). For Claim \ref{c:keyexp1}, we replace \eqref{e:keyexp1} with $U = 1$ by
\begin{multline*}
\frac{1}{d^{\fo/2}}\frac{1}{N^{b+c}d^{b}}\sum_{\bf m}\sum_{\bf ij}\bE\left[\prod_{a=1}^bA_{i_aj_a}\wh F_{\bf ijm}\right]
\\
=\frac{1}{d^{\fo/2}}\frac{1}{N^{2b+c}}\sum_{\bf m}\sum_{\bf ij}\bE\left[\wh F_{\bf ijm}\right]+
 \wh{\cal T}_{\fo+1, \deg(F)+1}(1)
+\frac{1}{d^{\fo/2}}\OO\left(\frac{1}{d^{\fa/2}} +\frac{d}{N}\right),
\end{multline*}
where $\wh{\cal T}_{\fo+1, \deg(F)+1}(1)$ is obtained from $\cal T_{\fo+1, \deg(F)+1}(1)$ in \eqref{e:keyexp1} by replacing $G$ with $\wh G$. Here we used Claim \ref{c:restrictR3id} instead of Claim \ref{c:taylorexp}.
Finally, for Claim \ref{c:keyexp2}, we replace \eqref{e:keyexp2} with $U = 1$ by
\begin{equation*}
\frac{1}{d^{\fo/2}}\frac{1}{N^{2+c}}\sum_{ii'\bf m}\bE[\wh G_{ii} \wh F_{i\bf m}]
=\frac{1}{d^{\fo/2}}\frac{1}{N^{2+c}}\sum_{ii'\bf m}\bE[\wh G_{i'i'} \wh F_{i\bf m}] + \wh {\cal T}_{\fo + 1, \deg(F) + 3}(1)
+\frac{1}{d^{\fo/2}}\OO \left(\frac{1}{d^{\fa/2}}+ \frac{d}{N}\right),
\end{equation*}
where $\wh{\cal T}_{\fo+1, \deg(F)+3}(1)$ is obtained from $\cal T_{\fo+1, \deg(F)+3}(1)$ in \eqref{e:keyexp2} by replacing $G$ with $\wh G$. Here we used Claim \ref{c:restrictR3id} and \eqref{e:multiHG} instead of Claim \ref{c:taylorexp} and \eqref{e:GHi'term}, \eqref{e:GHiterm}, respectively.

This concludes the proof.
\end{proof}

\begin{proof}[Proof of Proposition \ref{p:idfP}]
  The polynomial $P_\fa(z,w)$ in Proposition \ref{p:DSE} was constructed by repeatedly applying Proposition \ref{p:reduceterm}.
  Using Proposition \ref{p:reduceterm2} instead, we can repeat its proof verbatim to obtain
  \begin{align}
    P_\fa(z, \md)=\OO\left(\frac{1}{d^{\fa/2}}+\frac{d}{N}\right).
  \end{align}
Since the left-hand side does not depend on $N$ (recall Proposition \ref{p:DSE}), taking the limit $N \to \infty$ yields the claim.
\end{proof}

\subsection{Proof of Corollaries~\ref{c:Pproperty} and \ref{c:mdproperty}}

\begin{proof}[Proof of Corollary~\ref{c:Pproperty}]
We define the power series
\begin{align}
R(w)=P_{\fa}(z,w)-P_{\infty}(z,w).
\end{align}
It follows from Proposition \ref{p:idfP} that for any $z\in \bC$, 
\begin{align}
R(\md(z))= P_\fa(z,\md(z)) = \OO(d^{-\fa/2}).
\end{align}
Next, we remark that the disk $\{w \in \C \col \abs{w} \leq \frac{d - 1}{d}\}$ lies in the image $\md(\C)$. Indeed, from \eqref{e:sc_sce} we find that the image $\msc(\C)$ is the closed unit disk in $\C$, and the above claim then follows from the identity $\md = \frac{(d - 1) \msc}{d - 1 - \msc^2}$, which follows from \eqref{e:md} and \eqref{e:sc_sce}.
As a consequence, we have that for any $w$ such that $\abs{w} \leq \frac{d-1}{d}$,
\begin{align}
R(w)= \OO(d^{-\fa/2}).
\end{align}
For any fixed degree $k$, the coefficients of $w^k$ in the infinite series $R(w)$ is given by
\begin{align}
\frac{1}{2\pi\ri}\oint_{|w|=(d-1)/d} \frac{R(w)}{w^{k+1}} \, \dd w = \OO(d^{-\fa/2}),
\end{align}
since $d \geq 2$. This finishes the proof.
\end{proof}

\begin{proof}[Proof of Corollary~\ref{c:mdproperty}]
By Corollary~\ref{c:Pproperty}, we have
\begin{align}\label{e:derP12}
P_\fa^{(k)}(z,\md(z))=P_\infty^{(k)}(z,\md(z)) + \OO(d^{-\fa/2}), \quad k=1,2.
\end{align}
Since $P_\infty(z,\md(z))=0$  for $z\in \C_+$, by the chain rule,
\begin{align}
0=\del_zP_\infty(z,\md(z))=
\md(z)+P_\infty'(z,\md(z)) \del_z \md(z).
\end{align}
Rearranging the above expression gives
\begin{align}\label{e:derPinf}
P_\infty'(z,\md(z))=-\frac{\md(z)}{\del_z \md(z)},
\end{align}
so that an elementary analysis of \eqref{e:md} and \eqref{e:sc_sce} yields
\begin{equation*}
\absb{P_\infty'(z,\md(z))} \asymp \sqrt{|\kappa|+\eta},
\end{equation*}
and the first relation of \eqref{e:derPa} follows from combining \eqref{e:derP12} and \eqref{e:derPinf}.
The second relation of \eqref{e:derPa} follows from \eqref{e:derP12} and 
\begin{align}
P_\infty''(z,\md(z))=2 +\OO(d^{-1/2}).
\end{align}
This completes the proof.
\end{proof}

\section{Moment estimate for the self-consistent equation}
\label{sec:P-moments}

In order to establish the self-consistent equation for $m$ in the sense of high probability, in this section we derive a recursive moment estimate for the high moments of $P_\fa(z,m)$, where $P_\fa$ is the polynomial constructed in Proposition \ref{p:DSE}. In the next section, we establish eigenvalue rigidity estimates using a careful analysis of this recursive moment estimate and an iteration argument.

\begin{proposition}\label{p:Pe}
Suppose that Assumption \ref{ass:main} holds, and that $\Lambdad \geq \Lambdao\geq 1/\sqrt{d}$. Let $P_\fa(z,w)$ be the polynomial constructed in Proposition \ref{p:DSE}. Fix $r \in \N$.
Abbreviating $P_\fa \equiv P_\fa(z, m(z))$, we have for any $z \in \bld D$,
\begin{align}\begin{split}\label{e:Pe}
  &\phantom{{}={}}\bE[|P_\fa|^{2r}]
  \prec \frac{\Lambdad^2}{d^{1/2}}\bE\left[ \frac{\Im[m]|P_\fa'|}{N\eta}|P_\fa|^{2r-2}\right]
  +\bE\left[\frac{\Im[m]|P_\fa'|}{(N\eta)^2}|P_\fa|^{2r-2}\right]\\
  &+\frac{\Lambdao}{d^{1/2}}\bE\left[\frac{\Im[m]|P_\fa'|^2}{(N\eta)^3}|P_\fa|^{2r-3}\right]
  +\left(\frac{1}{d^{\fa/2}}+\frac{d^{3/2}\Lambdao}{N}\right)\bE[|P_\fa|^{2r-1}]
    +\max_{1 \leq s \leq 2r}\bE\left[\left(\frac{\Im[m]}{N\eta}\right)^s|P_\fa|^{2r-s}\right]
    \\
   &+\left(\frac{\Lambdad}{d}+\frac{\Lambdao^2}{d^{1/2}}+\frac{d\Lambdao}{N}\right)\max_{1 \leq s \leq 2r - 1}\bE\left[ \left(\frac{\Im[m]|P_\fa'|}{N\eta}\right)^s|P_\fa|^{2r-s-1}\right].
\end{split}
\end{align}
\end{proposition} 

The rest of this section is devoted to the proof of Proposition~\ref{p:Pe}.
We begin by writing $P_\fa(z, m)=1+mz+Q_\fa(m)$ (see \eqref{e:P_split}) and then apply the identity (see \eqref{e:GHexp})
\begin{equation*}
1+zm=\frac{1}{N} \sum_{ij}H_{ij}G_{ij} + \frac{1}{N}
\end{equation*}
to obtain
\begin{align}
\begin{split}\label{e:moment0}
\bE[|P_\fa|^{2r}]
&=\bE[(1+z m)P_\fa^{r-1}\bar P_\fa^{r}]+
\bE[Q_\fa P_\fa^{r-1}\bar P_\fa^{r}]\\
&=\frac{1}{(d-1)^{1/2}N}\sum_{ij}\bE[A_{ij}G_{ij}P_\fa^{r-1}\bar P_\fa^{r}]+
\bE[Q_\fa P_\fa^{r-1}\bar P_\fa^{r}]+\OO\left(\frac{\bE[|P_\fa|^{2r-1}]}{N}\right).
\end{split}
\end{align}
For the first term on the second line of \eqref{e:moment0}, we use Corollary \ref{c:intbp1} with the random variable $F=G_{ij}P_\fa^{r-1}P_\fa^r$ and note that for $i\neq j$
\begin{equation*}
|F(A)|+\max_{kl}|F(A+\xi_{ij}^{kl})|
\prec \Lambdao \cal C(P_\fa^{r-1}\bar P_\fa^r, A),
\end{equation*}
where we recall the definition \eqref{def_CA} and we used Remark~\ref{r:basicestimates}.
This yields
\begin{align}\begin{split}\label{e:moment}
&\phantom{{}={}}\frac{1}{(d-1)^{1/2}N}\sum_{ij}\bE[A_{ij}G_{ij}P_\fa^{r-1}\bar P_\fa^{r}]
=\frac{d}{(d-1)^{1/2}N^2}\sum_{i\neq j}\bE[G_{ij}P_\fa^{r-1}\bar P_\fa^r]\\
&+\frac{1}{N^2d(d-1)^{1/2}}\sum_{i\neq j}\sum_{kl}\bE[A_{ik}A_{jl}D_{ij}^{kl}(G_{ij}P_\fa^{r-1}\bar P_\fa^{r})]+
\OO_\prec \left(\frac{d^{3/2}\Lambdao}{N}\bE[ \cal C(P_\fa^{r-1}\bar P_\fa^r, A)]\right)\\
&=\frac{1}{N^2d(d-1)^{1/2}}\sum_{i jkl}\bE[A_{ik}A_{jl}D_{ij}^{kl}(G_{ij}P_\fa^{r-1}\bar P_\fa^{r})]
+\OO_\prec \left(\frac{d^{3/2}\Lambdao}{N}\bE[ \cal C(P_\fa^{r-1}\bar P_\fa^r, A)]\right),
\end{split}\end{align}
where in the last equality we used $\Lambdao\geq 1/\sqrt{d}$. To estimate the error term, we remark that for $U=P_\fa^{r-1}\bar P_\fa^r$ we have
\begin{align} \label{e:Uss_est}
\max_{s + \bar s \geq 1}\bE\left[|U^{(s, \bar s)}| \left(\frac{\Im[m]}{N\eta}\right)^{s + \bar s} \right]\prec \max_{1 \leq s \leq 2r - 1} \bE\left[\left(\frac{|P_\fa'|\Im[m]}{N\eta}+\left(\frac{\Im[m]}{N\eta}\right)^2\right)^s|P_\fa|^{2r-1-s}\right],
\end{align}
as can be seen after some elementary algebra. Using \eqref{e:exp3} for $U=P_\fa^{r-1}\bar P_\fa^r$ we get
\begin{align} \label{e:C_est}
\cal C(P_\fa^{r-1}\bar P_\fa^r, A)
\prec |P_\fa|^{2r-1}+\frac{1}{\sqrt{d}}\max_{1 \leq s \leq 2r - 1} \bE\left[\left(\frac{|P_\fa'|\Im[m]}{N\eta}+\left(\frac{\Im[m]}{N\eta}\right)^2\right)^s|P_\fa|^{2r-1-s}\right].
\end{align}
By the discrete product rule \eqref{e:D-product}, the first term in the last line of \eqref{e:moment} equals
\begin{align}
\begin{split}\label{e:1term}
&\bE[A_{ik}A_{jl}D_{ij}^{kl}(G_{ij}P_\fa^{r-1}\bar P_\fa^{r})]
=\bE[A_{ik}A_{jl}D_{ij}^{kl}(G_{ij})P_\fa^{r-1}\bar P_\fa^{r}] \\
&+\bE[A_{ik}A_{jl}G_{ij}D_{ij}^{kl}(P_\fa^{r-1}\bar P_\fa^{r})] 
+\bE[A_{ik}A_{jl}D_{ij}^{kl}(G_{ij})D_{ij}^{kl}(P_\fa^{r-1}\bar P_\fa^{r})]. 
\end{split}\end{align}
At this point we note the crucial cancellation, by Proposition \ref{p:constructQ}, of the first term on the right-hand side of \eqref{e:1term} with $\bE[Q_\fa P_\fa^{r-1}\bar P_\fa^r]$ from \eqref{e:moment0}. 
Indeed, by Proposition \ref{p:constructQ} for $U=P_\fa^{r-1}\bar P_\fa^r$, recalling \eqref{e:Uss_est} and \eqref{e:C_est}, we find
\begin{align}\label{e:7firsttermb}
\begin{split}
&\frac{1}{N^2d(d-1)^{1/2}}\sum_{i jkl}\bE[A_{ik}A_{jl}D_{ij}^{kl}(G_{ij})P_\fa^{r-1}\bar P_\fa^{r}]+
\bE[Q_\fa P_\fa^{r-1}\bar P_\fa^{r}]
\prec 
\bE\left[\left(\frac{1}{d^{\fa/2}}+\frac{\Im[m]}{N\eta}+\frac{d}{N}\right) |P_\fa|^{2r-1}\right]\\
&+ \left(\frac{d^{1/2}}{N}+\frac{\Lambdao^2}{d^{1/2}}\right)\max_{1 \leq s \leq 2r - 1} \bE\left[\left(\frac{|P_\fa'|\Im[m]}{N\eta}+\left(\frac{\Im[m]}{N\eta}\right)^2\right)^s|P_\fa|^{2r-1-s}\right].
\end{split}
\end{align}

What remains, therefore, is to estimate the contributions of the second and third terms on the right-hand side of \eqref{e:1term}.
For the second term on the right-hand side of \eqref{e:1term}, we claim that
\begin{align}\begin{split}\label{e:7secondterm}
&\frac{1}{N^2d(d-1)^{1/2}}\sum_{i jkl}\bE[A_{ik}A_{jl}G_{ij}D_{ij}^{kl}(P_\fa^{r-1}\bar P_\fa^{r})]
 \prec  \bE\left[\frac{\Lambdao\Lambdad}{d^{1/2}}\frac{\Im[m]|P_\fa'|}{N\eta}|P_\fa|^{2r-2}\right]\\
&+\bE\left[\frac{\Im[m]|P_\fa'|}{(N\eta)^2}|P_\fa|^{2r-2}\right]
+\frac{\Lambdao}{d}\max_{1 \leq s \leq 2r - 1}\bE\left[\left(\frac{|P_\fa'|\Im[m]}{N\eta}+\left(\frac{\Im[m]}{N\eta}\right)^2\right)^s|P_\fa|^{2r-1-s}\right]\\
    &+\frac{\Lambdao}{d^{1/2}}\bE\qa{\pa{\frac{\Im[m]}{N\eta}}^{2}  |P_\fa|^{2r-2}}
+\frac{\Lambdao}{d^{1/2}}\bE\left[\frac{\Im[m]|P_\fa'|^2}{(N\eta)^3}|P_\fa|^{2r-3}\right].
\end{split}\end{align}
Essentially, this estimate will arise from the three off-diagonal Green's function entries obtained from $G_{ij}$ and the derivative on $P_\fa^{r-1}\bar P_\fa^r$.

For the third term on the right-hand side of \eqref{e:1term}, we claim that
\begin{align}\label{e:7thirdterm}
\begin{split}
 &\phantom{{}={}}\frac{1}{N^2d(d-1)^{1/2}}\bE[A_{ik}A_{jl}D_{ij}^{kl}(G_{ij})D_{ij}^{kl}(P_\fa^{r-1}\bar P_\fa^{r})]
\prec \frac{\Lambdad^2}{d^{1/2}}\bE\left[\frac{\Im[m]|P_\fa'|}{N\eta}|P_\fa|^{2r-2}\right]+\frac{d}{N}\bE[|P_\fa|^{2r-1}]\\
&+\left(\frac{d^{1/2}}{N}+\frac{\Lambdad}{d}+\frac{\Lambdao^2}{d^{1/2}}\right)\max_{1 \leq s \leq 2r - 1}\bE\left[\left(\frac{\Im[m]|P_\fa'|}{N\eta}+\left(\frac{\Im[m]}{N\eta}\right)^2\right)^s|P_\fa|^{2r-s-1}\right]
+
\bE\left[\frac{\Im[m]}{N\eta}|P_\fa|^{2r-1}\right]
.
\end{split}
\end{align}
This estimate will arise from certain special cancellations arising from the $d$-regular graph structure.

Proposition \eqref{p:Pe} follows from combining the estimates \eqref{e:moment0}, \eqref{e:moment}, \eqref{e:Uss_est}, \eqref{e:1term}, \eqref{e:7firsttermb}, \eqref{e:7secondterm} and \eqref{e:7thirdterm}.
In the remainder of this section, we prove these estimates \eqref{e:7secondterm} and \eqref{e:7thirdterm}.

\subsection{Proof of  \eqref{e:7secondterm}}
By \eqref{e:D-expand}, left-hand side of \eqref{e:7secondterm} can be written as
\begin{align}\begin{split}\label{e:second2}
&\frac{\OO(1)}{N^2d^{3/2}}\sum_{ijkl}\bE[A_{ik}A_{jl}G_{ij}D_{ij}^{kl}(P_\fa^{r-1}\bar P_\fa^{r})] 
=\sum_{n=1}^{\fb-1}  \frac{\OO(1)}{N^2d^{(n+3)/2}}\sum_{ijkl}\bE[A_{ik}A_{jl}G_{ij}(\del_{ij}^{kl})^n(P_\fa^{r-1}\bar P_\fa^{r})]\\
&+\frac{\OO(1)}{N^2d^{(\fb+3)/2}}\sum_{ijkl}\bE\left[A_{ik}A_{jl}G_{ij} \pB{(\del_{ij}^{kl})^\fb (P_\fa^{r-1}\bar P_\fa^{r})(A+\theta \xi_{ij}^{kl})}\right],
\end{split}\end{align}
for some random $\theta \in [0,1]$.
In fact, for the terms corresponding to $n\geq 3$ in \eqref{e:second2}, we have the following simple estimate.
\begin{claim} \label{c:Cl71}For the terms in \eqref{e:second2} with $n\geq 3$,
\begin{align}\begin{split}
   &\phantom{{}={}} \frac{1}{N^2d^{(n+3)/2}}\sum_{ijkl}\bE[A_{ik}A_{jl}G_{ij}(\del_{ij}^{kl})^n(P_\fa^{r-1}\bar P_\fa^{r})]\\
    &\prec \frac{\Lambdao}{d^{(n-1)/2}}\max_{1 \leq s \leq 2r - 1}\bE\left[\left(\frac{|P_\fa'|\Im[m]}{N\eta}+\left(\frac{\Im[m]}{N\eta}\right)^2\right)^s|P_\fa|^{2r-1-s}\right].
\end{split}\end{align}
\end{claim}

\begin{proof}
Thanks to \eqref{e:dPbd} and chain rule, we have
\begin{align} \label{e:delnP}
(\del_{ij}^{kl})^n(P_\fa^{r-1}\bar P_\fa^{r})\prec \max_{1 \leq s \leq 2r - 1} \left(\frac{|P_\fa'|\Im[m]}{N\eta}+\left(\frac{\Im[m]}{N\eta}\right)^2\right)^s|P_\fa|^{2r-1-s}.
\end{align}
Using $\max_{i\neq j}|G_{ij}|\prec \Lambdao$  and $\max_i |G_{ii}| \prec 1$, it leads to
\begin{align}\begin{split}
&\hphantom{{}\prec{}}\frac{1}{N^2d^{(n+3)/2}}\sum_{ijkl}\bE[A_{ik}A_{jl}G_{ij}(\del_{ij}^{kl})^n(P_\fa^{r-1}\bar P_\fa^{r})]\\
&\prec  \frac{1}{N^2d^{(n+3)/2}}\left(\max_{1 \leq s \leq 2r - 1}\sum_{i\neq j}\sum_{kl}\bE\left[A_{ik}A_{jl}\Lambdao\left(\frac{|P_\fa'|\Im[m]}{N\eta}+\left(\frac{\Im[m]}{N\eta}\right)^2\right)^s|P_\fa|^{2r-1-s}\right]\right.\\
&\qquad\left.+\max_{1 \leq s \leq 2r - 1}\sum_{ikl}\bE\left[A_{ik}A_{il}\left(\frac{|P_\fa'|\Im[m]}{N\eta}+\left(\frac{\Im[m]}{N\eta}\right)^2\right)^s|P_\fa|^{2r-1-s}\right]\right)\\
&\prec \frac{\Lambdao}{d^{(n-1)/2}}\max_{1 \leq s \leq 2r - 1}\bE\left[\left(\frac{|P_\fa'|\Im[m]}{N\eta}+\left(\frac{\Im[m]}{N\eta}\right)^2\right)^s|P_\fa|^{2r-1-s}\right],
\end{split}\end{align}
where in the last inequality we used that $\sum_{i \neq jkl} A_{ik}A_{jl} = N^2 d^2$ and $\sum_{ikl} A_{ik}A_{il} = N d^2$
  and $\Lambdao \geq 1/N$.
\end{proof}

Moreover, by choosing $\fb$ large enough, depending on $r$, it follows that the second line of \eqref{e:second2} is bounded by $N^{-2r}$.

In the following, we estimate the terms on the right-hand side of \eqref{e:second2}, corresponding to $n=1$ and $n=2$.

\begin{claim}\label{c:Cl72}
For the term in \eqref{e:second2} with $n=1$,
\begin{equation}
  \frac{1}{N^2d^2}\sum_{ijkl}\bE[A_{ik}A_{jl}G_{ij}\del_{ij}^{kl}(P_\fa^{r-1}\bar P_\fa^{r})]
  \prec
  \bE\left[\frac{\Im[m]|P_\fa'|}{(N\eta)^2}|P_\fa|^{2r-2}\right].
\end{equation}
\end{claim}

\begin{proof}
For the derivative $\del_{ij}^{kl}(P_\fa^{r-1}\bar P_\fa^{r})$, we have
\begin{align}
\del_{ij}^{kl}(P_\fa^{r-1}\bar P_\fa^{r})
=(r-1)\del_{ij}^{kl}m P_\fa' P_\fa^{r-2}\bar P_\fa^r+(\cdots),
\end{align}
where $(\cdots)$ denotes analogous terms with complex conjugates obtained by applying the derivatives to $\bar P_\fa$ instead of $P_\fa$. We estimate the error from the first term, for which we can first sum over the indices $i,j,k,l$.
 By \eqref{e:dm},
\begin{align}\begin{split}\label{e:sumijkl}
\frac{1}{N^2d^2}\sum_{ijkl}A_{ik}A_{jl}G_{ij}\del_{ij}^{kl}m
=\frac{2}{N^3d^2}\sum_{ijkl}A_{ik}A_{jl}G_{ij}(-(G^2)_{ij}-(G^2)_{kl}+(G^2)_{ik}+(G^2)_{jl}).
\end{split}\end{align}
There are four terms on the right-hand side of \eqref{e:sumijkl}. For the first term, using \eqref{e:Gkbd},
\begin{align}\label{e:71term}
\frac{1}{N^3d^2}\sum_{ijkl}A_{ik}A_{jl}G_{ij}(G^2)_{ij}
  =\frac{1}{N^3}\sum_{ij}G_{ij}(G^2)_{ij}=
  \frac{1}{N^3}\Tr G^3\prec\frac{\Im[m]}{(N\eta)^2}.
\end{align}
For the second term on the right-hand side of \eqref{e:sumijkl}, using that $(H-z)G=P_{\perp}$ and \eqref{e:Gkbd},
\begin{align}\label{e:72term}
\begin{split}
&\phantom{{}={}}\frac{1}{N^3d^2}\sum_{ijkl}A_{ik}A_{jl}G_{ij}(G^2)_{kl}
=\frac{1}{N^3d^2}\Tr(GAGAG)\\
&=\frac{1}{N^3d^2}\Tr(G(A-\sqrt{d-1}z)G(A-\sqrt{d-1}z)G)\\
&+\frac{2}{N^3d^2}\Tr(G(\sqrt{d-1}z)G(A-\sqrt{d-1}z)G)\\
&+\frac{1}{N^3d^2}\Tr(G(\sqrt{d-1}z)G(\sqrt{d-1}z)G)\\
&=\frac{\OO(1)}{N^3d}\Tr G
+\frac{\OO(1)}{N^3d}\Tr G^2
+\frac{\OO(1)}{N^3d}\Tr G^3=\OO_\prec\left(\frac{\Im[m]}{d(N\eta)^2}\right).
\end{split}
\end{align}
For the last two terms on the right-hand side of \eqref{e:sumijkl},  since $\sum_j G_{ij} = 0$,
\begin{align}\begin{split}\label{e:73term}
  &\frac{1}{N^3d^2}\sum_{ijkl}A_{ik}A_{jl}G_{ij}(G^2)_{ik}
  = \frac{1}{N^3d}\sum_{ijk}A_{ik}G_{ij}(G^2)_{ik}
  = 0,\\
   &\frac{1}{N^3d^2}\sum_{ijkl}A_{ik}A_{jl}G_{ij}(G^2)_{jl}
  = \frac{1}{N^3d}\sum_{ijl}A_{jl}G_{ij}(G^2)_{jl}
  = 0.
\end{split}\end{align}
By combining expressions \eqref{e:71term}, \eqref{e:72term} and \eqref{e:73term}, we obtain the estimate
\begin{align}
\frac{1}{N^2d^2}\sum_{ijkl}A_{ik}A_{jl}G_{ij}\del_{ij}^{kl}m\prec \frac{\Im[m]}{(N\eta)^2},
\end{align}
for \eqref{e:sumijkl}, and the Claim \ref{c:Cl72} follows.
\end{proof}

\begin{claim}\label{c:Cl73}
  For the term in \eqref{e:second2} with $n=2$,
  \begin{align}\begin{split}
    &\phantom{{}={}}\frac{1}{N^2d^{5/2}}\sum_{ijkl}\bE[A_{ik}A_{jl}G_{ij}(\del_{ij}^{kl})^2(P_\fa^{r-1}\bar P_\fa^{r})]
    \prec
    \bE\left[\frac{\Lambdao\Lambdad}{d^{1/2}}\frac{\Im[m]|P_\fa'|}{N\eta}|P_\fa|^{2r-2}\right]\\
    &+\frac{\Lambdao}{d^{1/2}}\bE\qa{\pa{\frac{\Im[m]}{N\eta}}^{2}  |P_\fa|^{2r-2}}
+\frac{\Lambdao}{d^{1/2}}\bE\left[\frac{\Im[m]|P_\fa'|^2}{(N\eta)^3}|P_\fa|^{2r-3}\right]
    +\frac{\Lambdao}{d}\bE\left[\left(\frac{\Im[m]|P_\fa'|}{N\eta}\right)^2|P_\fa|^{2r-3}\right].
 \end{split} \end{align}
\end{claim}

\begin{proof}
For the derivative $\del_{ij}^{kl}(P_\fa^{r-1}\bar P_\fa^{r})$, we have
  \begin{equation}\label{e:derPP}
    (\del_{ij}^{kl})^2(P_\fa^{r-1}\bar P_\fa^{r})
    = ((\del_{ij}^{kl})^2 m) P_\fa' P_\fa^{r-2}\bar P_\fa^r
    + (\del_{ij}^{kl}m)^2 P_\fa'' P_\fa^{r-2}\bar P_\fa^r
    + ((\del_{ij}^{kl}m) P_\fa')^2 P_\fa^{r-3}\bar P_\fa^r + (\cdots),
  \end{equation}
  where $(\cdots)$ denotes analogous terms with complex conjugates obtained by applying the derivatives to $\bar P_\fa$ instead of $P_\fa$.
  We consider the terms separately.

For the first term in \eqref{e:derPP}, we use the explicit formula
  \begin{equation} \label{e:partial2m}
    (\del_{ij}^{kl})^2m
    =
    \frac{2}{N}\sum_{a=1}^N \partial_{ij}^{kl} (-G_{ia}G_{ja}-G_{ka}G_{la}+G_{ia}G_{ka}+G_{ja}G_{la})
    = \frac{2}{N} \sum_{a=1}^{N} (G_{ii}G_{ja}^2 + \cdots),
  \end{equation}
  where $\cdots$ denotes $31$ other terms obtained by applying the product rule for differentiation for $\partial_{ij}^{kl}$.
  Using that $\sum_j G_{ij} =0$ and \eqref{e:Gkbd}, the first term gives
  \begin{align}\begin{split}
    &\phantom{{}={}}\frac{1}{N^3 d^{5/2}} \sum_{ijkla} A_{ik}A_{jl}G_{ij} G_{ii}G_{ja}^2
    =
      \frac{1}{N^3 d^{1/2}} \sum_{ij} G_{ij} G_{ii}(G^2)_{jj},
    \\
    &=
      \frac{1}{N^3 d^{1/2}} \sum_{ij} G_{ij} (G_{ii}-m)(G^2)_{jj}
    \prec \frac{\Lambdao\Lambdad}{d^{1/2}} \frac{\Im[m]}{N\eta}.
    \end{split}
  \end{align}
An analogous calculation can be performed for all terms on the right-hand side of \eqref{e:partial2m}. Indeed, every such term has three factors of $G$, exactly two of which have an index $a$, and the third remaining factor is either diagonal (in which case the same argument as above applies) or off-diagonal (in which case we gain $\Lambdao$ instead of $\Lambdad$). Since $\Lambdad\geq \Lambdao$, it leads to the estimate 
  \begin{align}\label{e:721term}
  \frac{1}{N^2d^{5/2}}\sum_{ijkl}\bE[A_{ik}A_{jl}G_{ij}((\del_{ij}^{kl})^2 m) P_\fa' P_\fa^{r-2}\bar P_\fa^r
]\prec \bE\left[\frac{\Lambdao\Lambdad}{d^{1/2}}\frac{\Im[m]|P_\fa'|}{N\eta}|P_\fa|^{2r-2}\right].
  \end{align}

For the second term in \eqref{e:derPP}, we can directly apply \eqref{e:d2m} to get the bound
  \begin{align}\label{e:722term}
 \frac{1}{N^2d^{5/2}}\sum_{ijkl}\bE[A_{ik}A_{jl}G_{ij}(\del_{ij}^{kl}m)^2 P_\fa'' P_\fa^{r-2}\bar P_\fa^r
]    \prec \frac{\Lambdao}{d^{1/2}}\bE\qa{\pa{\frac{\Im[m]}{N\eta}}^{2} |P_\fa''| |P_\fa|^{2r-2}},
  \end{align}
which is enough by $P_\fa'' \prec 1$.

For the third term in \eqref{e:derPP}, we have $(\del_{ij}^{kl}m)^2=4((G^2)_{ij}+(G^2)_{kl}-(G^2)_{ik}-(G^2)_{jl})^2/N^2$. There are ten different terms, which we estimate one by one.
For the term $(G^2)_{ij}(G^2)_{ij}$, we use \eqref{e:Gkbd} to get
\begin{align}\begin{split}
&\phantom{{}={}}\frac{1}{N^4d^{5/2}}\sum_{ijkl}A_{ik}A_{jl}G_{ij}(G^2)_{ij}(G^2)_{ij}
=\frac{1}{N^4d^{1/2}}\sum_{ij}G_{ij}(G^2)_{ij}(G^2)_{ij}\\
&\prec\frac{1}{N^4d^{1/2}}\sum_{ij}\Lambdao\left|(G^2)_{ij}\right|^2
+\frac{1}{N^4d^{1/2}}\sum_{i}\left|(G^2)_{ii}\right|^2\\
&=\frac{1}{N^4d^{1/2}}\Lambdao\Tr|G|^4
+\frac{1}{Nd^{1/2}}\left(\frac{\Im[m]}{N\eta}\right)^2
\prec \frac{\Lambdao}{d^{1/2}}\frac{\Im[m]}{(N\eta)^3}+\frac{1}{Nd^{1/2}}\left(\frac{\Im[m]}{N\eta}\right)^2.
\end{split}\end{align}
For the term $(G^2)_{ij}(G^2)_{kl}$, we use \eqref{e:Gkbd} and $(H-z)G=P_\perp$ to get
\begin{align}\begin{split}
&\phantom{{}={}}\frac{1}{N^4d^{5/2}}\sum_{ijkl}A_{ik}A_{jl}G_{ij}(G^2)_{ij}(G^2)_{kl}
=\frac{1}{N^4d^{5/2}}\sum_{ij}G_{ij}(G^2)_{ij}(AG^2A)_{ij}\\
&=\frac{d-1}{N^4d^{5/2}}\sum_{ij}G_{ij}(G^2)_{ij}(z^2G^2+2zG+P_\perp)_{ij}
\\
&\prec\frac{1}{N^2d^{3/2}}\sum_{ij}|G_{ij}|\frac{\Im[m]^2}{(N\eta)^2}
\prec \frac{\Lambdao}{d^{3/2}}\left(\frac{\Im[m]}{N\eta}\right)^2,
\end{split}\end{align}
where in the second equality we used that $AG^2A=(d-1)(H-z+z)G^2(H-z+z)=(d-1)(z^2G^2+2zG+P_\perp)$.
For the term $(G^2)_{ij}(G^2)_{ik}$, we use \eqref{e:Gkbd} and $(H-z)G=P_\perp$ to get
\begin{align}\begin{split}
&\phantom{{}={}}\frac{1}{N^4d^{5/2}}\sum_{ijkl}A_{ik}A_{jl}G_{ij}(G^2)_{ij}(G^2)_{ik}
=\frac{1}{N^4d^{3/2}}\sum_{ij}G_{ij}(G^2)_{ij}(AG^2)_{ii}\\
&=\frac{(d-1)^{1/2}}{N^4d^{3/2}}\sum_{ij}G_{ij}(G^2)_{ij}(zG^2+G)_{ii}
\\
&\prec\frac{1}{N^2d}\sum_{ij}|G_{ij}|\frac{\Im[m]^2}{(N\eta)^2}
\prec \frac{\Lambdao}{d}\left(\frac{\Im[m]}{N\eta}\right)^2,
\end{split}\end{align}
where in the second equality we used that $AG^2=(d-1)^{1/2}(H-z+z)G^2=(d-1)^2(zG^2+G)$.
The term $(G^2)_{ij}(G^2)_{jl}$ can be estimated in the same way.
For the term $(G^2)_{kl}(G^2)_{kl}$, we use \eqref{e:Gkbd} and $(H-z)G=P_\perp$ to get
\begin{align}\begin{split}
&\phantom{{}={}}\frac{1}{N^4d^{5/2}}\sum_{ijkl}A_{ik}A_{jl}G_{ij}(G^2)_{kl}(G^2)_{kl}
=\frac{1}{N^4d^{5/2}}\sum_{kl}(AGA)_{kl}(G^2)_{kl}(G^2)_{kl}\\
&=\frac{d-1}{N^4d^{5/2}}\sum_{kl}(z^2G+zP_\perp+HP_\perp)_{kl}(G^2)_{kl}(G^2)_{kl}
\\
&\prec\frac{1}{N^2d^{3/2}}\sum_{kl}|(z^2G+zP_\perp+HP_\perp)_{kl}|\frac{\Im[m]^2}{(N\eta)^2}
\prec \frac{\Lambdao}{d^{3/2}}\left(\frac{\Im[m]}{N\eta}\right)^2,
\end{split}\end{align}
where in the second equality we used that $AGA=(d-1)(H-z+z)G(H-z+z)=(d-1)(z^2G+zP_\perp+HP_\perp)$, and in the last inequality, we used $\Lambdao\geq 1/\sqrt{d}$ and $\sum_{kl}|(z^2G+zP_\perp+HP_\perp)_{kl}|\prec N^2\Lambdao+d^{1/2}N\leq N^2\Lambdao$.
For the term $(G^2)_{kl}(G^2)_{ik}$, we use \eqref{e:Gkbd} and $(H-z)G=P_\perp$ to get
\begin{align}\begin{split}
&\phantom{{}={}}\frac{1}{N^4d^{5/2}}\sum_{ijkl}A_{ik}A_{jl}G_{ij}(G^2)_{kl}(G^2)_{ik}
=\frac{1}{N^4d^{5/2}}\sum_{ikl}A_{ik}(AG)_{il}(G^2)_{kl}(G^2)_{ik}\\
&=\frac{(d-1)^{1/2}}{N^4d^{5/2}}\sum_{ikl}A_{ik}(zG+P_\perp)_{il}(G^2)_{kl}(G^2)_{ik}\\
&\prec\frac{1}{N^2d^{2}}\sum_{ikl}A_{ik}|(zG+P_\perp)_{il}|
\frac{\Im[m]^2}{(N\eta)^2}
\prec \frac{\Lambdao}{d}\left(\frac{\Im[m]}{N\eta}\right)^2,
\end{split}\end{align}
where in the second line we used $AG=(d-1)^{1/2}(H-z+z)G=(d-1)^{1/2}(zG+P_\perp)$, and in the last inequality we used $\sum_{l}|(zG+P_\perp)_{il}|\leq N\Lambdao$. The term $(G^2)_{kl}(G^2)_{jl}$ can be estimated in the same way. For the term $(G^2)_{ik}(G^2)_{ik}$, we use that $\sum_{j}G_{ij}=0$ to get
\begin{align}\begin{split}
&\phantom{{}={}}\frac{1}{N^4d^{5/2}}\sum_{ijkl}A_{ik}A_{jl}G_{ij}(G^2)_{ik}(G^2)_{ik}
=\frac{1}{N^4d^{3/2}}\sum_{ijk}A_{ik}G_{ij}(G^2)_{ik}(G^2)_{ik}
=0.
\end{split}\end{align}
The term $(G^2)_{jl}(G^2)_{jl}$ can be estimated in the same way.
For the term $(G^2)_{ik}(G^2)_{jl}$, we use \eqref{e:Gkbd} and $(H-z)G=P_\perp$ to get
\begin{align}\begin{split}
&\phantom{{}={}}\frac{1}{N^4d^{5/2}}\sum_{ijkl}A_{ik}A_{jl}G_{ij}(G^2)_{ik}(G^2)_{jl}
=\frac{1}{N^4d^{5/2}}\sum_{ij}G_{ij}(AG^2)_{ii}(AG^2)_{jj}\\
&=\frac{(d-1)}{N^4d^{5/2}}\sum_{ij}G_{ij}(zG^2+G)_{ii}(zG^2+G)_{jj}\\
&\prec\frac{1}{N^2d^{3/2}}\sum_{ij}|G_{ij}|
\frac{\Im[m]^2}{(N\eta)^2}
\prec \frac{\Lambdao}{d^{3/2}}\left(\frac{\Im[m]}{N\eta}\right)^2.
\end{split}\end{align}
where in the second equality, we used $AG^2=(d-1)^{1/2}(H-z+z)G^2=(d-1)^{1/2}(zG^2+G)$.

We combine the above estimates together, and find that the third term in \eqref{e:derPP} is bounded by
\begin{align}\begin{split}\label{e:723term}
&\phantom{{}={}}\frac{1}{N^2d^{5/2}}\sum_{ijkl}\bE[A_{ik}A_{jl}G_{ij}((\del_{ij}^{kl}m) P_\fa')^2 P_\fa^{r-3}\bar P_\fa^r]\\
&\prec \frac{\Lambdao}{d^{1/2}}\bE\left[\frac{\Im[m]|P_\fa'|^2}{(N\eta)^3}|P_\fa|^{2r-3}\right]
    +\frac{\Lambdao}{d}\bE\left[\left(\frac{\Im[m]|P_\fa'|}{N\eta}\right)^2|P_\fa|^{2r-3}\right].
\end{split}\end{align}
The Claim \ref{c:Cl73} follows from combining \eqref{e:721term}, \eqref{e:722term} and \eqref{e:723term}.
\end{proof}

\subsection{Proof of \eqref{e:7thirdterm}}
By \eqref{e:D-expand}, the left-hand side of \eqref{e:7thirdterm} can be written as 
\begin{align}\begin{split}\label{e:third2}
&\phantom{{}={}}\frac{1}{N^2d(d-1)^{1/2}}\sum_{ijkl}\bE[A_{ik}A_{jl}D_{ij}^{kl}(G_{ij})D_{ij}^{kl}(P_\fa^{r-1}\bar P_\fa^{r})]\\
&=\sum_{n_1=1}^{\fb_1-1}\frac{1}{n_1!N^2d(d-1)^{(n_1+1)/2}}
\sum_{ijkl}\bE\left[A_{ik}A_{jl}(\del_{ij}^{kl})^{n_1}(G_{ij})D_{ij}^{kl}(P_\fa^{r-1}\bar P_\fa^{r})\right]\\
&+\frac{1}{\fb_1!N^2d(d-1)^{(\fb_1+1)/2}}
\sum_{ijkl}\bE\left[A_{ik}A_{jl}\pB{(\del_{ij}^{kl})^{\fb_1}(G_{ij})(A+\theta \xi_{ij}^{kl})}D_{ij}^{kl}(P_\fa^{r-1}\bar P_\fa^{r})\right],
\end{split}\end{align}
for some random $\theta\in [0,1]$. As above, by choosing $\fb_1$ large enough, depending on $r$, we find that the last line of \eqref{e:third2} is bounded by $N^{-2r}$.
Moreover, for terms corresponding to $n_1\geq 3$ in \eqref{e:third2}, we have the following simple estimate.
\begin{claim} \label{c:Cl733}For the terms $n_1\geq 3$ on the right-hand side of \eqref{e:third2},
\begin{align}\begin{split}
   &\phantom{{}={}} \frac{1}{N^2d(d-1)^{(n_1+1)/2}}
\sum_{ijkl}\bE\left[A_{ik}A_{jl}(\del_{ij}^{kl})^{n_1}(G_{ij})D_{ij}^{kl}(P_\fa^{r-1}\bar P_\fa^{r})\right]\\
    &\prec \frac{1}{d^{n_1/2}}\max_{1 \leq s \leq 2r - 1}\bE\left[\left(\frac{|P_\fa'|\Im[m]}{N\eta}+\left(\frac{\Im[m]}{N\eta}\right)^2\right)^s|P_\fa|^{2r-1-s}\right].
\end{split}\end{align}
\end{claim}
\begin{proof}
Thanks to \eqref{e:D-expand} and \eqref{e:delnP}, we have
\begin{align}\label{e:DPPterm}
D_{ij}^{kl}(P_\fa^{r-1}\bar P_\fa^{r})\prec \max_{1 \leq s \leq 2r - 1}\left(\frac{|P_\fa'|\Im[m]}{N\eta}+\left(\frac{\Im[m]}{N\eta}\right)^2\right)^s\frac{|P_\fa|^{2r-1-s}}{d^{1/2}}.
\end{align}
Combining with $(\partial_{ij}^{kl})^{n_1} G_{ij} \prec 1$ we obtain
\begin{align}\begin{split}
&\phantom{{}={}}
\frac{1}{N^2d(d-1)^{(n_1+1)/2}}
\sum_{ijkl}\bE\left[A_{ik}A_{jl}(\del_{ij}^{kl})^{n_1}(G_{ij})D_{ij}^{kl}(P_\fa^{r-1}\bar P_\fa^{r})\right]\\
&\prec \frac{1}{N^2d^{(n_1+4)/2}}
\max_{1 \leq s \leq 2r - 1}\sum_{ijkl}\bE\left[A_{ik}A_{jl}\left(\frac{|P_\fa'|\Im[m]}{N\eta}+\left(\frac{\Im[m]}{N\eta}\right)^2\right)^s|P_\fa|^{2r-1-s}\right]\\
&=\frac{1}{d^{n_1/2}}
\max_{1 \leq s \leq 2r - 1}\bE\left[\left(\frac{|P_\fa'|\Im[m]}{N\eta}+\left(\frac{\Im[m]}{N\eta}\right)^2\right)^s|P_\fa|^{2r-1-s}\right].
\qedhere
\end{split}
\end{align}
\end{proof}

In the following, we estimate the terms on the right-hand side of \eqref{e:third2}, corresponding to $n_1=1$ and $n_1=2$.

\begin{claim} \label{c:Cl731}For the term $n_1=1$ on the right-hand side of \eqref{e:third2},
\begin{align}\begin{split}
   &\phantom{{}={}} \frac{1}{N^2d(d-1)}
\sum_{ijkl}\bE\left[A_{ik}A_{jl}\del_{ij}^{kl}G_{ij} D_{ij}^{kl}(P_\fa^{r-1}\bar P_\fa^{r})\right]
\prec
    \frac{\Lambdad^2}{d^{1/2}}\bE\left[\frac{\Im[m]|P_\fa'|}{N\eta}|P_\fa|^{2r-2}\right]+\frac{d}{N}\bE[|P_\fa|^{2r-1}]\\
&+\left(\frac{d^{1/2}}{N}+\frac{\Lambdad}{d}+\frac{\Lambdao^2}{d^{1/2}}\right)\max_{1 \leq s \leq 2r - 1}\bE\left[\left(\frac{\Im[m]|P_\fa'|}{N\eta}+\left(\frac{\Im[m]}{N\eta}\right)^2\right)^s|P_\fa|^{2r-s-1}\right]
+
\bE\left[\frac{\Im[m]}{N\eta}|P_\fa|^{2r-1}\right].
\end{split}\end{align}
\end{claim}

\begin{proof}
The derivative $\del_{ij}^{kl}G_{ij}$ is given by
\begin{align}\label{e:derG-bis}
  \del_{ij}^{kl}G_{ij}=-G_{ii}G_{jj}+G_{il}G_{jj}+G_{ii}G_{kj}-G_{ij}G_{ij}-G_{ik}G_{lj}-G_{il}G_{kj}+G_{ik}G_{ij}+G_{ij}G_{lj}
  .
\end{align}
We shall show that the biggest term is $-G_{ii}G_{jj}$, and that the other terms are smaller.
We first estimate those terms in \eqref{e:derG-bis} which contain two off-diagonal terms,
i.e., $-G_{ij}G_{ij}-G_{ik}G_{lj}-G_{il}G_{kj}+G_{ik}G_{ij}+G_{ij}G_{lj}$. For the term $G_{ij}G_{ij}$,
using \eqref{e:DPPterm} to bound $D_{ij}^{kl}(P_\fa^{r-1}\bar P_\fa^r)$,
\begin{align}\begin{split}\label{e:twooffd}
&\phantom{{}={}}\frac{1}{N^2d(d-1)}
\sum_{ijkl}\bE\left[A_{ik}A_{jl}G_{ij}G_{ij}D_{ij}^{kl}(P_\fa^{r-1}\bar P_\fa^{r})\right]\\
&\prec \frac{1}{N^2d(d-1)}\left(\sum_{i\neq j}\sum_{kl}\bE\left[A_{ik}A_{jl}\Lambdao^2 |D_{ij}^{kl}(P_\fa^{r-1}\bar P_\fa^{r})|\right]+\sum_{ikl}\bE\left[A_{ik}A_{il}|D_{ii}^{kl}(P_\fa^{r-1}\bar P_\fa^{r})|\right]\right)\\
&\prec\frac{\Lambdao^2}{d^{1/2}}\max_{1 \leq s \leq 2r - 1}\bE\left[\left(\frac{|P_\fa'|\Im[m]}{N\eta}+\left(\frac{\Im[m]}{N\eta}\right)^2\right)^s|P_\fa|^{2r-1-s}\right].
\end{split}\end{align}
The same argument applies to the other terms with two off-diagonal indices.

We next estimate those terms in \eqref{e:derG-bis}, which contain exactly one off-diagonal term,
i.e., $G_{il}G_{jj}+G_{ii}G_{kj}$. For the term $G_{il}G_{jj}$,
\begin{align}\begin{split}\label{e:secondTaylor}
&\phantom{{}={}}\frac{1}{N^2d(d-1)}
\sum_{ijkl}\bE\left[A_{ik}A_{jl}G_{il}G_{jj}D_{ij}^{kl}(P_\fa^{r-1}\bar P_\fa^{r})\right]\\
&=\sum_{n_2= 1}^{\fb_2-1}\frac{1}{n_2!N^2d(d-1)^{1+n_2/2}}
\sum_{ijkl}\bE\left[A_{ik}A_{jl}G_{il}G_{jj}(\del_{ij}^{kl})^{n_2}(P_\fa^{r-1}\bar P_\fa^{r})\right]\\
&+\frac{1}{\fb_2!N^2d(d-1)^{1+\fb_2/2}}
\sum_{ijkl}\bE\left[A_{ik}A_{jl}G_{il}G_{jj}\pB{(\del_{ij}^{kl})^{\fb_2}(P_\fa^{r-1}\bar P_\fa^{r})(A+\theta\xi_{ij}^{kl})}\right],
\end{split}
\end{align}
for some random $\theta\in [0,1]$. As above, by choosing $\fb_2$ large enough, depending on $r$, we find that the last line of \eqref{e:secondTaylor} is bounded by $N^{-2r}$. Moreover, for terms corresponding to $n_2\geq 2$ in \eqref{e:third2}, by the same argument as \eqref{e:twooffd}
and $\max_{i\neq l}|G_{il}|\prec \Lambdao$, we have the simple estimate
\begin{align}\begin{split}
   &\phantom{{}={}} \frac{1}{N^2d(d-1)^{1+n_2/2}}
\sum_{ijkl}\bE\left[A_{ik}A_{jl}G_{il}G_{jj}(\del_{ij}^{kl})^{n_2}(P_\fa^{r-1}\bar P_\fa^{r})\right]\\
    &\prec \frac{\Lambdao}{d^{n_2/2}}\max_{1 \leq s \leq 2r - 1}\bE\left[\left(\frac{|P_\fa'|\Im[m]}{N\eta}+\left(\frac{\Im[m]}{N\eta}\right)^2\right)^s|P_\fa|^{2r-1-s}\right].
\end{split}\end{align}
Using the above estimate in \eqref{e:secondTaylor}, we get
\begin{align} \label{e:yetanotherterm}
\begin{split}
&\phantom{{}={}}\frac{1}{N^2d(d-1)}
\sum_{ijkl}\bE\left[A_{ik}A_{jl}G_{il}G_{jj}D_{ij}^{kl}(P_\fa^{r-1}\bar P_\fa^{r})\right]\\
&=\frac{1}{N^2d(d-1)^{3/2}}
\sum_{ijkl}\bE\left[A_{ik}A_{jl}G_{il}G_{jj}\del_{ij}^{kl}(P_\fa^{r-1}\bar P_\fa^{r})\right]\\
&+\OO_\prec \left( \frac{\Lambdao}{d}\max_{1 \leq s \leq 2r - 1}\bE\left[\left(\frac{|P_\fa'|\Im[m]}{N\eta}+\left(\frac{\Im[m]}{N\eta}\right)^2\right)^s|P_\fa|^{2r-1-s}\right]\right).
\end{split}\end{align}
The first term on the right-hand side of \eqref{e:yetanotherterm}, arising from $n_2 = 1$, needs to be estimated more precisely.
It can be written as
\begin{align}\begin{split}\label{e:onederterm}
&\phantom{{}={}}\frac{1}{N^2d(d-1)^{3/2}}
\sum_{ijkl}\bE\left[A_{ik}A_{jl}G_{il}G_{jj}\del_{ij}^{kl}(P_\fa^{r-1}\bar P_\fa^{r})\right]\\
&=\frac{\OO(1)}{N^2d^{5/2}}
\sum_{ijkl}\bE\left[A_{ik}A_{jl}G_{il}G_{jj}\left((r-1)\del_{ij}^{kl}m P_\fa' P_\fa^{r-2}\bar P_\fa^{r}+r\del_{ij}^{kl}\bar m \bar P_\fa' |P_\fa|^{2r-2}\right)\right].
\end{split}\end{align}
We estimate $\sum_{ikl}A_{ik}A_{jl}G_{il}\del_{ij}^{kl}m$, and the other term $\sum_{ikl}A_{ik}A_{jl}G_{il}\del_{ij}^{kl}\bar m$
can be estimated in the same way.
Recalling \eqref{e:dm}, we have
\begin{align}\label{e:dm2}
&\del_{ij}^{kl}m
=
\frac{2}{N}(-(G^2)_{ij}-(G^2)_{kl}+(G^2)_{ik}+(G^2)_{jl}).
\end{align}
We first estimate the terms in \eqref{e:dm2} which do not contain the index $l$.
For the terms $(G^2)_{ij}+(G^2)_{ik}$,
using the definition of the Green's function $(H-z)G=P_\perp$ and then using \eqref{e:GGbd},
\begin{align}\begin{split}\label{e:tt1}
&\phantom{{}={}}\frac{1}{N^3d^{5/2}}\sum_{ijkl}A_{ik}A_{jl}G_{il}(-(G^2)_{ij}+(G^2)_{ik})
=\frac{1}{N^3d^{5/2}}\sum_{ijk}A_{ik}(AG)_{ij}(-(G^2)_{ij}+(G^2)_{ik})\\
&=\frac{(d-1)^{1/2}}{N^3d^{5/2}}\sum_{ijk}A_{ik}(zG+P_\perp)_{ij}(-(G^2)_{ij}+(G^2)_{ik})\\
&\prec \frac{1}{N^2d^{2}}\sum_{ijk}A_{ik}|(zG+P_\perp)_{ij}|\frac{\Im[m]}{N\eta}
 \prec \frac{\Lambdao}{d}\frac{\Im[m]}{N\eta},
\end{split}\end{align}
where in the second equality we used that $AG=(d-1)^{1/2}(H-z+z)G=(d-1)^{1/2}(zG+P_\perp)$.
For the terms which do not contain the index $i$, $(G^2)_{jl}-(G^2)_{lk}$, analogously,
\begin{align}\begin{split}\label{e:tt2}
\phantom{{}={}}\frac{1}{N^3d^{5/2}}\sum_{ijkl}A_{ik}A_{jl}G_{il}((G^2)_{jl}-(G^2)_{lk})
&\prec  \frac{\Lambdao}{d}\frac{\Im[m]}{N\eta}.
\end{split}\end{align}
We plug the estimates \eqref{e:tt1} and \eqref{e:tt2} into \eqref{e:onederterm}, and get
\begin{align}\label{e:twooffdiagonal}
\frac{1}{N^2d(d-1)^{3/2}}
\sum_{ijkl}\bE\left[A_{ik}A_{jl}G_{il}G_{jj}\del_{ij}^{kl}(P_\fa^{r-1}\bar P_\fa^{r})\right]
\prec\frac{\Lambdao}{d}\bE\left[\frac{\Im[m]|P_\fa'|}{N\eta}|P_\fa|^{2r-2}\right]
.
\end{align} 
The same argument applies to the term $G_{ii}G_{kj}$, and we have the same estimate as \eqref{e:twooffdiagonal}.

To estimate the term $-G_{ii}G_{jj}$ in \eqref{e:derG-bis} we use the following Claim \ref{c:Clcancel}, which concludes the proof.
\end{proof}

The following claim uses a cancellation that exploits that the graph is regular.
\begin{claim}\label{c:Clcancel}
We have
\begin{align}\begin{split}\label{e:Bterm}
&\phantom{{}={}}\frac{1}{N^2d(d-1)}
\sum_{ijkl}\bE\left[A_{ik}A_{jl}G_{ii}G_{jj}D_{ij}^{kl}(P_\fa^{r-1}\bar P_\fa^{r})\right]\prec
\frac{\Lambdad^2}{d^{1/2}}\bE\left[\frac{\Im[m]|P_\fa'|}{N\eta}|P_\fa|^{2r-2}\right]+\frac{d}{N}\bE[|P_\fa|^{2r-1}]\\
&+\left(\frac{d^{1/2}}{N}+\frac{\Lambdad}{d}\right)\max_{1 \leq s \leq 2r - 1}\bE\left[\left(\frac{\Im[m]|P_\fa'|}{N\eta}+\left(\frac{\Im[m]}{N\eta}\right)^2\right)^s|P_\fa|^{2r-s-1}\right]
+
\bE\left[\frac{\Im[m]}{N\eta}|P_\fa|^{2r-1}\right]
.
\end{split}\end{align}
\end{claim}

\begin{proof}
Using $\sum_{ij} (A_{ij}-d/N) = 0$ and Corollary \ref{c:intbp1} with the random variable $F=m^2 P_\fa^{r-1}P_\fa^r$
we have
\begin{align}\begin{split}\label{e:BtermEst}
&\phantom{{}={}}
0=\frac{1}{N(d-1)}\sum_{ijkl}\bE[(A_{ij}-\frac{d}{N})m^2 (P_\fa^{r-1}\bar P_\fa^r)]\\
&=\frac{1}{N^{2}d(d-1)}\sum_{ijkl}\bE[A_{ik}A_{jl}D_{ij}^{kl}(m^2 (P_\fa^{r-1}\bar P_\fa^r))]+\OO_\prec \left(\frac{d}{N}\bE\left[\cal C(P_\fa^{r-1}\bar P_\fa^r, A)\right]\right)\\
&=\frac{1}{N^{ 2}d(d-1)}\sum_{ijkl}\bE[A_{ik}A_{jl}m^2 D_{ij}^{kl}(P_\fa^{r-1}\bar P_\fa^r)]+\OO_\prec\left(\frac{d}{N}\bE\left[\cal C(P_\fa^{r-1}\bar P_\fa^r, A)\right]+\max_{1 \leq s \leq 2r}\bE\left[\pbb{\frac{\Im[m]}{N\eta}}^s|P_\fa|^{2r-s}\right]\right),
\end{split}\end{align}
where we used \eqref{e:exp3}, the discrete derivative rule \eqref{e:D-product}, and \eqref{e:delnP}. The error $\cal C(P_\fa^{r-1}\bar P_\fa^r, A)$ is estimated in \eqref{e:C_est}.

Therefore, subtracting \eqref{e:BtermEst} from the left-hand side of \eqref{e:Bterm}, we get
\begin{align}\begin{split}\label{e:aterm}
&\phantom{{}={}}\frac{1}{N^2d(d-1)}
\sum_{ijkl}\bE\left[A_{ik}A_{jl}G_{ii}G_{jj}D_{ij}^{kl}(P_\fa^{r-1}\bar P_\fa^{r})\right]\\
&=\frac{1}{N^2d(d-1)}
\sum_{ijkl}\bE\left[A_{ik}A_{jl}(G_{ii}G_{jj}-m^2)D_{ij}^{kl}(P_\fa^{r-1}\bar P_\fa^{r})\right]\\
&\phantom{{}={}}+\OO_\prec\left(\frac{d}{N}\bE\left[\cal C(P_\fa^{r-1}\bar P_\fa^r, A)\right]+\max_{1 \leq s \leq 2r}\bE\left[\pbb{\frac{\Im[m]}{N\eta}}^s|P_\fa|^{2r-s}\right]\right)
.
\end{split}\end{align}
Using \eqref{e:D-expand}, we rewrite the first term on the right-hand side as
\begin{align}\begin{split}\label{e:lastexp}
&\phantom{{}={}}\frac{\OO(1)}{N^2d^2}
\sum_{ijkl}\bE\left[A_{ik}A_{jl}(G_{ii}G_{jj}-m^2)D_{ij}^{kl}(P_\fa^{r-1}\bar P_\fa^{r})\right]\\
&=\sum_{n_2 = 1}^{\fb_2 - 1} \frac{1}{n_2!N^2d^{n_2/2+2}}
\sum_{ijkl}\bE\left[A_{ik}A_{jl}(G_{ii}G_{jj}-m^2)(\del_{ij}^{kl})^{n_2}(P_\fa^{r-1}\bar P_\fa^{r})\right]
\\
&\phantom{{}={}}
+ \frac{1}{\fb_2!N^2d^{\fb_2/2+2}}
\sum_{ijkl}\bE\left[A_{ik}A_{jl}(G_{ii}G_{jj}-m^2) \pB{(\del_{ij}^{kl})^{\fb_2}(P_\fa^{r-1}\bar P_\fa^{r})(A + \theta \xi_{ij}^{kl})}\right]
\end{split}\end{align}
for some random $\theta\in [0,1]$. As above, by choosing $\fb_2$ large enough, depending on $r$, we find that the last line of \eqref{e:lastexp} is bounded by $N^{-2r}$. Moreover, for terms corresponding to $n_2\geq 2$ in \eqref{e:lastexp}, we have the following simple estimate.
Using  $|G_{ii} - m| \prec \Lambdad$, \eqref{e:delnP}, and that $\sum_k A_{ik} = \sum_l A_{jl} = d$, we find
\begin{align}\begin{split}\label{e:bterm}
&\phantom{{}={}}\frac{1}{N^2d^{n_2/2+2}}
\sum_{ijkl}\bE\left[A_{ik}A_{jl}(G_{ii}G_{jj}-m^2)(\del_{ij}^{kl})^{n_2}(P_\fa^{r-1}\bar P_\fa^{r})\right]\\
&\prec\frac{\OO(1)}{N^2d^{n_2/2+2}}\max_{1 \leq s \leq 2r - 1}\sum_{ijkl}\bE\left[A_{ik}A_{jl}\Lambdad\left(\frac{\Im[m]|P_\fa'|}{N\eta}+\left(\frac{\Im[m]}{N\eta}\right)^2\right)^s|P_\fa|^{2r-s-1}\right] \\
&\prec\frac{\Lambdad}{d^{n_2/2}}\max_{1 \leq s \leq 2r - 1}\bE\left[\left(\frac{\Im[m]|P_\fa'|}{N\eta}+\left(\frac{\Im[m]}{N\eta}\right)^2\right)^s|P_\fa|^{2r-s-1}\right].
\end{split}\end{align}

For the term in \eqref{e:lastexp} corresponding to $n_2=1$, the estimate is more involved. We start by writing
\begin{align}\begin{split}
  &\phantom{{}={}}\frac{\OO(1)}{N^2d^{5/2}}
    \sum_{ijkl}\bE\left[A_{ik}A_{jl}(G_{ii}G_{jj}-m^2)\del_{ij}^{kl} (P_\fa^{r-1}\bar P_\fa^{r}))\right]
  \\
  &=\frac{\OO(1)}{N^2d^{5/2}}
    \sum_{ijkl}\bE\left[A_{ik}A_{jl}(G_{ii}G_{jj}-m^2)\left(r(\del_{ij}^{kl}\bar m)\bar P_\fa' |P_\fa|^{2r-2}+(r-1)(\del_{ij}^{kl} m) P_\fa' \bar P_\fa^2|P_\fa|^{2r-3} \right)\right].
\end{split}\end{align}
We estimate the term $\sum_{kl}A_{ik}A_{jl}\del_{ij}^{kl}m$; its complex conjugate $\sum_{kl}A_{ik}A_{jl}\del_{ij}^{kl}\bar m$ is estimated analogously.
We use \eqref{e:dm} and
estimate the resulting four terms one by one.
For the term $-(G^2)_{ij}$,
\begin{align}\label{e:t1}\begin{split}
&\phantom{{}={}}\frac{1}{N^3d^{5/2}}
    \sum_{ijkl}\bE\left[A_{ik}A_{jl}(G_{ii}G_{jj}-m^2)(G^2)_{ij}\bar P_\fa' |P_\fa|^{2r-2}\right]\\
    &=\frac{1}{N^3d^{1/2}}
    \sum_{ij}\bE\left[(G_{ii}G_{jj}-m^2)(G^2)_{ij}\bar P_\fa' |P_\fa|^{2r-2}\right]\\
    &=\frac{1}{N^3d^{1/2}}
    \sum_{ij}\bE\left[((G_{ii}-m)m+(G_{jj}-m)m+(G_{ii}-m)(G_{jj}-m))(G^2)_{ij}\bar P_\fa' |P_\fa|^{2r-2}\right]\\
     &=\frac{1}{N^3d^{1/2}}
    \sum_{ij}\bE\left[(G_{ii}-m)(G_{jj}-m)(G^2)_{ij}\bar P_\fa' |P_\fa|^{2r-2}\right]\\&\prec \frac{\Lambdad^2}{d^{1/2}}\bE\left[\frac{\Im[m]|P_\fa'|}{N\eta}|P_\fa|^{2r-2}\right],
   \end{split}
\end{align}
where in the third equality, we used that the first two terms vanish after summing over index $j$ and $i$ respectively because $\sum_j (G^2)_{ij} = 0 = \sum_i (G^2)_{ij}$, and in the last inequality, we used $|G_{ii}-m|\prec \Lambdad$ and \eqref{e:Gkbd}.
For the term $(G^2)_{ik}$,
\begin{align}\begin{split}\label{e:t2}
&\phantom{{}={}}\frac{1}{N^3d^{5/2}}
    \sum_{ijkl}\bE\left[A_{ik}A_{jl}(G_{ii}G_{jj}-m^2)(G^2)_{ik}\bar P_\fa' |P_\fa|^{2r-2}\right]\\
&=\frac{1}{N^3d^{3/2}}\sum_{ij}\bE\left[(G_{ii}G_{jj}-m^2)(AG^2)_{ii}\bar P_\fa' |P_\fa|^{2r-2}\right]\\
&=\frac{(d-1)^{1/2}}{N^3d^{3/2}}\sum_{ij}\bE\left[(G_{ii}G_{jj}-m^2)(zG^2+G)_{ii}\bar P_\fa' |P_\fa|^{2r-2}\right]\\
&\prec\frac{\Lambdad}{d}\bE\left[\frac{\Im[m]|P_\fa'|}{N\eta}|P_\fa|^{2r-2}\right],
\end{split}\end{align}
where we used that $AG^2=(d-1)^{1/2}(H-z+z)G^2=(d-1)^{1/2}(zG^2+G)$ and \eqref{e:Gkbd}.
We have the same estimate for the term $(G^2)_{jl}$,
\begin{align}\begin{split}\label{e:t3}
\frac{1}{N^3d^{5/2}}
    \sum_{ijkl}\bE\left[A_{ik}A_{jl}(G_{ii}G_{jj}-m^2)(G^2)_{jl}\bar P_\fa' |P_\fa|^{2r-2}\right]\prec\frac{\Lambdad}{d}\bE\left[\frac{\Im[m]|P_\fa'|}{N\eta}|P_\fa|^{2r-2}\right].
\end{split}\end{align}
For the term $-(G^2)_{kl}$, we use \eqref{e:Gkbd}
\begin{align}\begin{split}\label{e:t4}
&\phantom{{}={}}\frac{1}{N^3d^{5/2}} \sum_{ijkl}\bE\left[A_{ik}A_{jl}(G_{ii}G_{jj}-m^2)(G^2)_{kl}\bar P_\fa' |P_\fa|^{2r-2}\right]\\
&=\frac{1}{N^3d^{5/2}} \sum_{ij}\bE\left[(G_{ii}G_{jj}-m^2)(AG^2A)_{ij}\bar P_\fa' |P_\fa|^{2r-2}\right]\\
&=\frac{d-1}{N^3d^{5/2}} \sum_{ij}\bE\left[(G_{ii}G_{jj}-m^2)(z^2G^2+2zG+P_\perp)_{ij}\bar P_\fa' |P_\fa|^{2r-2}\right]\\
&\prec \frac{\Lambdad}{d^{3/2}}\bE\left[\frac{\Im[m]|P_\fa'|}{N\eta}|P_\fa|^{2r-2}\right],
\end{split}\end{align}
where in the third equality, we used $AG^2A=(d-1)(H-z+z)G^2(H-z+z)=(d-1)(z^2G^2+2zG+P_\perp)$.
We combine the estimates \eqref{e:t1}, \eqref{e:t2}, \eqref{e:t3} and \eqref{e:t4}, and use that $\Lambdad\geq 1/\sqrt{d}$,
\begin{align}\begin{split}\label{e:cterm}
&\phantom{{}={}}\frac{\OO(1)}{N^2d^{5/2}}
    \sum_{ijkl}\bE\left[A_{ik}A_{jl}(G_{ii}G_{jj}-m^2)\del_{ij}^{kl} (\bar P_\fa |P_\fa|^{2r-2}))\right]
\prec \frac{\Lambdad^2}{d^{1/2}}\bE\left[\frac{\Im[m]|P_\fa'|}{N\eta}|P_\fa|^{2r-2}\right].
\end{split}\end{align}
The Claim \ref{c:Clcancel} follows from combining \eqref{e:aterm}, \eqref{e:bterm} and \eqref{e:cterm}.
\end{proof}

\begin{claim} \label{c:Cl732}For the term $n_1=2$ on the right-hand side of \eqref{e:third2},
\begin{align}\begin{split}
   &\phantom{{}={}} \frac{1}{N^2d(d-1)^{3/2}}
\sum_{ijkl}\bE\left[A_{ik}A_{jl}(\del_{ij}^{kl})^{2}(G_{ij})D_{ij}^{kl}(P_\fa^{r-1}\bar P_\fa^{r})\right]\\
    &\prec \frac{\Lambdao}{d}\max_{1 \leq s \leq 2r - 1}\bE\left[\left(\frac{|P_\fa'|\Im[m]}{N\eta}+\left(\frac{\Im[m]}{N\eta}\right)^2\right)^s|P_\fa|^{2r-1-s}\right].
\end{split}\end{align}
\end{claim}
\begin{proof}
For $i\neq j$ we use $(\partial_{ij}^{kl})^2G_{ij} \prec \Lambdao$ from \eqref{e:d2G},
\begin{align}\begin{split}
&\phantom{{}={}}\frac{1}{N^2d(d-1)^{3/2}}
\sum_{ijkl}\bE\left[A_{ik}A_{jl}(\del_{ij}^{kl})^{2}(G_{ij})D_{ij}^{kl}(P_\fa^{r-1}\bar P_\fa^{r})\right]\\
&\prec \frac{\Lambdao}{N^2d^{5/2}}
\sum_{ijkl}\bE\left[A_{ik}A_{jl}|D_{ij}^{kl}(P_\fa^{r-1}\bar P_\fa^{r})|\right]+\frac{1}{N^2d^{5/2}}
\sum_{ikl}\bE\left[A_{ik}A_{il}|D_{ii}^{kl}(P_\fa^{r-1}\bar P_\fa^{r})|\right]\\
&\prec\frac{\Lambdao}{d}
\max_{1 \leq s \leq 2r - 1}\bE\left[\left(\frac{|P_\fa'|\Im[m]}{N\eta}+\left(\frac{\Im[m]}{N\eta}\right)^2\right)^s|P_\fa|^{2r-1-s}\right],
\end{split}\end{align}
where we used \eqref{e:DPPterm} for the second inequality.
\end{proof}

\section{Analysis of self-consistent equation and proof of Theorem~\ref{t:eigloc}} \label{sec:P-prop}

In this section we analyse the recursive moment estimate \eqref{e:Pe} from Proposition \ref{p:Pe}, 
around the spectral edges $\pm2$, and obtain an improved estimate for the Stieltjes transform $m(z)$.

In the following, we focus on the right spectral edge; an analogous argument applies to the left edge.
For a fixed integer $\fa \geq 1$ and the same large constant $\fK>0$ as in \eqref{e:D}, we define the spectral domain for the right edge
(around the point $z=2$) by
\begin{align}\label{e:defcD}\begin{split}
\cDe &= \Biggl\{2+\kappa+\ri\eta:  0\leq\eta\leq \fK, 0\leq 2+\kappa\leq \fK\\
&\quad |\kappa|+\eta\geq \frac{1}{d^{ \fa/2}}+\frac{d}{N}+\frac{d^3}{N^2}+\frac{d^{3/2}}{N(N\eta)^{1/2}}+\frac{1}{d^{3/2}N\eta}+\frac{1}{(N\eta)^2}\Biggr\}.
\end{split}\end{align}
 Also recall $\md$ from \eqref{e:md}.

\begin{theorem}\label{thm:edgerigidity}
Fix an integer $\fa \geq 1$. For $1\ll d\ll N^{2/3}$, the following holds uniformly for any $z=2+\kappa+\ri \eta\in \cDe$.
\begin{itemize}
\item If $\kappa\geq 0$ then
\begin{align}\begin{split}\label{e:stateoutS}
&\phantom{{}={}}|m(z)-\md(z)|\prec 
\left(\frac{1}{d}+\frac{d}{N}+\frac{1}{(dN\eta)^{1/2}}\right)^{1/3}\frac{1}{N\eta^{2/3}(|\kappa|+\eta)^{1/3}}\\
&+\frac{1}{\sqrt{|\kappa|+\eta}}\left(\frac{1}{d^{\fa/2}}+\frac{d}{N}+\frac{d^3}{N^2}+\frac{1}{N^{1/2}d^{3/2}}+\frac{d^{3/2}}{N^{3/2}\eta^{1/2}}+\frac{1}{N\eta^{1/2}}+\frac{1}{d^{3/2}N\eta}+\frac{1}{(N\eta)^2}\right).
\end{split}\end{align}
\item If $\kappa\leq 0$ then
\begin{align}\label{e:stateinS}
|m( z)-\md( z)|\prec
\frac{1}{(N\eta)^{1/2}} \left(\frac{1}{N\eta}+\frac{d^{5/2}}{N^2}+\frac{1}{d^{3/2}}\right)^{1/2}+\frac{1}{\sqrt{|\kappa|+\eta}}\left(\frac{1}{d^{\fa/2}}+\frac{d}{N}+\frac{d^3}{N^2}+\frac{d^{3/2}}{N\sqrt{N\eta}}\right).
\end{align}
\end{itemize}
The analogous statement holds around the left edge.
\end{theorem}

In this section, we use the notation $X \lesssim Y$ to mean $X = \OO(Y)$ and $X \asymp Y$ to mean $X \lesssim Y$ and $Y \lesssim X$.

The estimate for the extremal eigenvalues of the random $d$-regular graphs, Theorem~\ref{t:eigloc},
follows as a corollary of Theorem~\ref{thm:edgerigidity} and \cite[Theorem 1.1]{MR3758727}, which states that with high probability, the second-largest eigenvalue and the smallest eigenvalue of a random $d$-regular graph is of order $\OO(1)$.
\begin{proposition}[{\cite[Theorem 1.1]{MR3758727}}]\label{t:priorb}
Fix $\fc > 0$. For $1\leq d\lesssim N^{2/3}$ and large enough $N$, there exists a constant $\fK$ depending on $\fc$ such that, with probability at least $1-N^{-1/\fc}$,
\begin{align}
\max\{\la_2, -\la_N\}\leq \fK.
\end{align}
\end{proposition}

\begin{proof}[Proof of Theorem~\ref{t:eigloc}]
We recall the following elementary estimates on $\md(z)$ which hold for bounded $|z|$.
Let $z=2+\kappa+\ri \eta$ or $z=-2-\kappa+\ri\eta$ for $0\leq \eta\leq \fK, 0 \leq 2 + \kappa\leq \fK$.
Then
\begin{align} \label{eqn:imtilminf}
\Im[\md(z)]\asymp
\begin{cases}
\sqrt{|\kappa|+\eta} & \txt{if $\kappa\leq 0$,}\\
\frac{\eta}{\sqrt{|\kappa|+\eta}} & \txt{if $\kappa\geq 0$.}
\end{cases}
\end{align}
This estimate follows from \eqref{e:md} and the analogous estimate for $\msc$; see, for example, \cite[Lemma~6.2]{MR3699468}.

We choose $\fa = 6$ in Theorem \ref{thm:edgerigidity}, and take $z=2+\kappa+\ri \eta\in \cDe$, where 
\begin{align*}
  \kappa=N^\fc\left(\frac{1}{d^3}+\frac{1}{N^{2/3}}+\frac{d^2}{N^{4/3}}\right),
  \qquad
  \eta = \frac{N^{\fc/2}}{N\sqrt{\kappa}} .
\end{align*}
In particular, $\eta \ll \kappa$.
Then under our assumption $1\ll d\ll N^{2/3}$, \eqref{e:stateoutS} implies
\begin{align}
|m(z)-\md(z)|\ll \frac{1}{N\eta}.
\end{align}
Since $\Im[\md(z)]\asymp \eta/\sqrt{\kappa}\ll 1/N\eta$, we get $\Im[m(z)]\ll 1/N\eta$.
Since any eigenvalue in $[2+\kappa-\eta, 2+\kappa+\eta]$ would yield a positive contribution of size at least $1/N\eta$ to $\Im[m(z)]$,
this implies that there cannot be any eigenvalue on the interval $[2+\kappa-\eta, 2+\kappa+\eta]$. Since we can take any $\kappa$ in the interval $[N^\fc(d^{-3}+N^{-2/3}+d^2 N^{-4/3}), \fK - 2]$, combining with Proposition~\ref{t:priorb}, we conclude that with probability $1 - {2}N^{-1/\fc}$ we have $\la_2\leq 2+ N^\fc(d^{-3}+N^{-2/3}+d^2 N^{-4/3})$. Since $\fc>0$ is arbitrary this implies the same conclusion with probability $1-N^{-1/\fc}$ as in the statement of the theorem. The lower bound for $\la_N$ follows from the same argument. This finishes the proof of Theorem~\ref{t:eigloc}.
\end{proof}

\subsection{Stability of self-consistent equation}

Recall from Corollary~\ref{c:mdproperty} that for $z\in \cDe$,
\begin{align} \label{e:mdproperty-bis}
|P_\fa'(z,\md(z))|\asymp \sqrt{|\kappa|+\eta}+\OO( d^{-\fa/2}), \quad
  P_\fa''(z,\md(z))=2+\OO( d^{-1/2}),\quad
  P_\fa'''(z,m_d(z))=\OO(1).
\end{align}
The following proposition on the stability of the self-consistent equation, relying on the above square root behaviour at the edge,
is essentially \cite[Lemma 4.5]{MR3183577}.

\begin{proposition}\label{p:stable}
Fix an integer $\fa \geq 1$.
There exists a constant $\varepsilon>0$ such that the following holds. Suppose that $\delta:\cDe\rightarrow \bC$ satisfies $N^{-2}\leq \delta(z)\leq \varepsilon$ for $z\in \cDe$, and that $\delta$ is Lipschitz continuous with Lipschitz constant $N$. Suppose moreover that for each fixed $\kappa$, the function $\eta\mapsto \delta(2+\kappa+\ri\eta)$ is nonincreasing for $\eta>0$. Suppose that for all $z\in \cDe$ we have
\begin{align}
|P_\fa(z, \md(z))|+|P_\fa(z, m(z))|\leq \delta(z).
\end{align}
Then we have for $z=2+\kappa+\ri \eta\in \cDe$,
\begin{align}
|m(z)-\md(z)|=\OO\left(\frac{\delta(z)}{\sqrt{|\kappa|+\eta+\delta(z)}}\right),
\end{align}
where the implicit constant is independent of $z$ and $N$.
\end{proposition}

\begin{proof}
Let $z=2+\kappa+\ri\eta\in \cDe$. 
By Taylor expansion, \eqref{e:mdproperty-bis}, and Proposition~\ref{thm:rigidity}, it follows that
\begin{align}\begin{split}\label{e:taylor}
P_\fa(z, m(z))
&=P_\fa(z, \md(z))+P_\fa'(z, \md(z))(m(z)-\md(z))\\
&+(P_\fa''(z, \md(z))+o(1))(m(z)-\md(z))^2/2.
\end{split}\end{align}
We abbreviate $\cR(z)\deq P_\fa(z,m(z))-P_\fa(z,\md(z))$. There exists $a(z)\asymp\sqrt{|\kappa|+\eta}$ and $b(z)\asymp 1$, such that
\begin{align}\label{e:cReq}
\cR(z)=a(z)(m(z)-\md(z))+b(z)(m(z)-\md(z))^2.
\end{align}
With \eqref{e:cReq}, Proposition \ref{p:stable} follows by a continuity argument that is essentially the same as \cite[Lemma 4.5]{MR3183577}.
\end{proof}

\subsection{Estimates on individual Green's function entries}

In the following we first prove some estimates on the individual entries of Green's function,
which slightly improve the estimates on the diagonal Green's function entries from \cite[Theorem 1.1]{MR3688032}
using the results from the previous sections.

\begin{proposition}\label{p:newboundG}
Fix $1\ll d\ll N^{2/3}$. Uniformly for $z=2+\kappa+\ri \eta\in \cDe$, we have \eqref{Lambda_assumptions}
with $\Lambdao$ and $\Lambdad$ given by \eqref{Lambda_strong}.
The analogous statement holds for the left edge.
\end{proposition}

\begin{proof}
From Proposition~\ref{thm:rigidity}, we have \eqref{Lambda_assumptions}
for all $z \in \cDe$, with $\Lambdao$ and $\Lambdad$ given by \eqref{Lambda_weak}.
(In all all bounds below, $\Lambdao$ and $\Lambdad$ will continue to be given by \eqref{Lambda_weak}.)
By the rough bounds $\Im[m]\prec  1$ and $|P_\fa'|\prec 1$, it follows from Proposition \ref{p:Pe} that
\begin{align}\begin{split}\label{e:easybound}
\bE[|P_\fa|^{2r}]
  &\prec 
 \frac{\Lambdad^2}{d^{1/2}}\bE\left[\frac{|P_\fa|^{2r-2}}{N\eta}\right]+\left(\frac{1}{d^{\fa/2}}+\frac{d^{3/2}\Lambdao}{N}\right)\bE[|P_\fa|^{2r-1}]
    \\
    &\phantom{{}={}}+\max_{1 \leq s \leq 2r}\bE\left[\frac{|P_\fa|^{2r-s}}{(N\eta)^s}\right]+\left(\frac{\Lambdad}{d}+\frac{\Lambdao^2}{d^{1/2}}+\frac{d\Lambdao}{N}\right) \max_{1 \leq s \leq 2r - 1}\bE\left[ \frac{|P_\fa|^{2r-s-1}}{(N\eta)^s}\right].
\end{split}\end{align}
By Jensen's inequality, we get from \eqref{e:easybound} that for any $r\geq 1$,
\begin{align}\label{e:momentbound}
\bE[|P_\fa(z, m( z))|^{2r}]
\prec\frac{1}{d^{r\fa} }+\frac{\Lambdad^{2r}}{d^{r}}+\frac{1}{(N\eta)^{2r}}
+\frac{d^{3r}\Lambdao^{2r}}{N^{2r}}.
\end{align}
Therefore
\begin{align}
|P_\fa(z, m(z))|\prec 
\frac{1}{d^{\fa/2}}+\frac{\Lambdad^2}{d^{1/2}}+\frac{1}{N\eta}+\frac{d^{3/2}\Lambdao}{N},
\end{align}
uniformly for $z\in \cDe$. 
Taking
\begin{align}
\delta(z)=N^\fc \left(\frac{1}{d^{\fa/2}}+\frac{\Lambdad^2}{d^{1/2}}+\frac{1}{N\eta}+\frac{d^{3/2}\Lambdao}{N}\right),
\end{align}
in Proposition \ref{p:stable}, where $\fc > 0$, we obtain
\begin{align}\label{e:newmNbound}
  |m(z)-\md(z)|\prec \frac{1}{d^{\fa/4}}+\frac{\Lambdad}{d^{1/4}}+\frac{1}{(N\eta)^{1/2}}+\frac{d^{3/4}\Lambdao^{1/2}}{N^{1/2}}
 \prec \frac{1}{d^{1/2}}+\frac{d^{3/2}}{N}+\frac{1}{\sqrt{N\eta}}.
\end{align}
where the last bound follows from the definitions of $\Lambdao$ and $\Lambdad$ in \eqref{Lambda_weak}
  and taking $\fa\geq 2$.
The proof is completed by applying \cite[Lemma~5.4]{MR3688032}. Indeed,
thanks to \cite[Lemma~5.4]{MR3688032}, we have
\begin{align}
1+(m(z)+z)G_{ii}(z)=\OO_\prec \left(\frac{1}{d^{1/2}}+\frac{d^{3/2}}{N}+\frac{1}{\sqrt{N\eta}}\right).
\end{align}
With \eqref{e:newmNbound} as input and noticing that $|\msc(z)-m_d(z)|=\OO(1/d)$, we get 
\begin{align}
1+(\msc(z)+z)G_{ii}(z)=\OO_\prec \left(\frac{1}{d^{1/2}}+\frac{d^{3/2}}{N}+\frac{1}{\sqrt{N\eta}}\right).
\end{align}
Using \eqref{e:sc_sce} we obtain
\begin{align}
G_{ii}(z)-\msc(z)=\OO_\prec \left(\frac{1}{d^{1/2}}+\frac{d^{3/2}}{N}+\frac{1}{\sqrt{N\eta}}\right),
\end{align}
from which we deduce that
\begin{align}
\max_{i}|G_{ii}(z)-\md(z)|\prec  \frac{1}{d^{1/2}}+\frac{d^{3/2}}{N}+\frac{1}{\sqrt{N\eta}},
\end{align}
as claimed.
\end{proof}

\subsection{Proof of Theorem~\ref{thm:edgerigidity}}

Throughout this section, we abbreviate $\Lambda(z)=|m(z)-\md(z)|$, and $\Phi(z)=\Im[\md(z)]$. Then $\Im[m]=\Phi(z)+\OO(\Lambda(z))$, and $P_\fa'(z,m(z))=P_\fa'(z,\md(z))+\OO(\Lambda(z))$ from Corollary \ref{c:mdproperty}. Moreover, we take $\Lambdao$ and $\Lambdad$ to be given by \eqref{Lambda_strong}. By Proposition \ref{p:newboundG}, we know that \eqref{Lambda_assumptions} holds for any $z \in \cDe$. 
Notice that Proposition \ref{p:Pe} implies that, for any $z\in \cDe$,
\begin{align}\begin{split}\label{e:Pe2}
  &\phantom{{}={}}\bE[|P_\fa|^{2r}]
  \prec \frac{\Lambdao^2}{d^{1/2}}\max_{1 \leq s \leq 2r - 1} \bE\left[ \left(\frac{\Im[m]|P_\fa'|}{N\eta}\right)^s|P_\fa|^{2r-s-1}\right]
  +\bE\left[\frac{\Im[m]|P_\fa'|}{(N\eta)^2}|P_\fa|^{2r-2}\right]\\
  &+\frac{\Lambdao}{d^{1/2}}\bE\left[\frac{\Im[m]|P_\fa'|^2}{(N\eta)^3}|P_\fa|^{2r-3}\right]
  +\left(\frac{1}{d^{\fa/2}}+\frac{d^{3/2}\Lambdao}{N}\right)\bE[|P_\fa|^{2r-1}]
    +\max_{1 \leq s \leq 2r}\bE\left[\left(\frac{\Im[m]}{N\eta}\right)^s|P_\fa|^{2r-s}\right].
\end{split}\end{align}
We notice that for $z=2+\kappa+\ri \eta\in \cDe$, $|P_\fa'(z,m(z))|\lesssim \sqrt{|\kappa|+\eta}+\Lambda(z)$.
By Jensen's inequality, it follows from \eqref{e:Pe2} that
\begin{align}\begin{split}\label{e:Pbound}
&\phantom{{}={}}\bE\left[|P(z, m( z))|^{2r}\right]\prec\bE\left[ \left(\frac{1}{N\eta}+\frac{\Lambdao^2}{d^{1/2}}\right)^{r}\left(\frac{(\Phi+\Lambda)(\sqrt{|\kappa|+\eta}+\Lambda)}{N\eta}\right)^{r}\right]\\
  &+\bE\left[\left(\frac{\Phi+\Lambda}{N\eta}\right)^{2r}\right]
  +\bE\left[\left(\frac{\Lambdao}{d^{1/2}}\right)^{2r/3}\frac{(\Phi+\Lambda)^{2r/3}(\sqrt{|\kappa|+\eta}+\Lambda)^{4r/3}}{(N\eta)^{2r}}\right]
    +\bE\left[\left(\frac{1}{d^{\fa/2}}+\frac{d^{3/2}\Lambdao}{N}\right)^{2r}\right].
\end{split}\end{align}

Before proving Theorem~\ref{thm:edgerigidity}, we prove the following weaker estimate which will be used as an input in the Proof of Theorem \ref{thm:edgerigidity}.

\begin{proposition}
Uniformly for any $z=2+\kappa+\ri \eta\in \cDe$, we have
\begin{align}
\Lambda(z)=|m(z)-\md(z)| \prec \sqrt{|\kappa|+\eta}.
\end{align}
\end{proposition} 

\begin{proof}
Thanks to \eqref{eqn:imtilminf}, we have $\Phi(z)\prec \sqrt{|\kappa|+\eta}$. From our choice of the spectral domain $\cDe$ in \eqref{e:defcD}, we have
\begin{align}
\frac{1}{N\eta}\leq \sqrt{|\kappa|+\eta},\quad  \frac{1}{d^{\fa/2}}, \frac{d^{3/2}\Lambdao}{N}, \frac{\Lambdao^2}{d^{1/2}}\frac{1}{N\eta}\leq |\kappa|+\eta.
\end{align}
Thus \eqref{e:Pbound} implies
\begin{align}\label{e:momentbound1}
\bE\left[|P_\fa(z, m( z))|^{2r}\right]\prec(|\kappa|+\eta)^{2r}+(|\kappa|+\eta)^{r}\bE\left[\Lambda(z)^{2r}\right].
\end{align}
We have the Taylor expansion \eqref{e:taylor}
\begin{align}
 P_\fa(z, m(z))
&=P_\fa(z, \md(z))
+P_\fa'(z, \md(z))(m( z)-\md(z))+(1+\oo(1))(m(z)-\md(z))^2,
\label{e:taylorexpand}\end{align}
where we used that $P_\fa''=2+\OO(1/d^{1/2}+d/N)$, $P_{\fa}'''=\OO(1)$, and $\Lambda(z)\ll1$.
Rearranging the last equation \eqref{e:taylorexpand} and using the definition of $\Lambda$, we have arrived at 
\begin{align}
\Lambda(z)^2\lesssim\Lambda(z)\sqrt{|\kappa|+\eta}+ |P_\fa(z,m(z))|+|P_\fa(z,\md(z))|, 
\end{align}
and thus 
\begin{align}\label{e:Lambda4r}
\bE[\Lambda(z)^{4r}]\lesssim(|\kappa|+\eta)^r\bE[\Lambda(z)^{2r}]+ \bE[|P_\fa(z,m(z))|^{2r}]+ |P_\fa(z,\md(z))|^{2r},
\end{align}
for any fixed integer $r \geq 1$.
Now we replace $\bE[|P_\fa(z,m(z))|^{2r}]$ in \eqref{e:Lambda4r} by the right-hand side of \eqref{e:momentbound1}.
Moreover, on the domain $\cDe$, from Proposition \ref{p:idfP}, we have $|P_\fa(z,\md(z))|\leq |\kappa|+\eta$.
With $\bE[\Lambda(z)^{2r}] \leq \bE[\Lambda(z)^{4r}]^{1/2}$ by the Cauchy--Schwarz inequality, we thus obtain
\begin{align}
\bE[\Lambda(z)^{4r}]\prec (|\kappa|+\eta)^{2r},
\end{align}
for any $r\geq 1$. The claim $\Lambda(z)\prec \sqrt{|\kappa|+\eta}$ follows from Markov's inequality.
\end{proof}

\begin{proof}[Proof of Theorem \ref{thm:edgerigidity}]
We assume that there exists some deterministic control parameter $\Theta(z)$,  for any $z=2+\kappa+\ri\eta\in \cDe$, we have the a priori estimate
\begin{align}
|m(z)-\md(z)|\prec \Theta(z)\lesssim \sqrt{|\kappa|+\eta}.
\end{align}
From \eqref{e:Pbound} we get
\begin{align}\begin{split}\label{e:Pbound2}
|P_\fa(z, m( z))|&\prec \left(\frac{1}{N\eta}+\frac{\Lambdao^2}{d^{1/2}}\right)^{1/2}\left(\frac{(\Phi+\Theta)\sqrt{|\kappa|+\eta}}{N\eta}\right)^{1/2}
  +\frac{\Phi+\Theta}{N\eta}\\
  &+\left(\frac{\Lambdao}{d^{1/2}}\right)^{1/3}\frac{(\Phi+\Theta)^{1/3}(|\kappa|+\eta)^{1/3}}{N\eta}
    +\frac{1}{d^{\fa/2}}+\frac{d^{3/2}\Lambdao}{N}.
\end{split}\end{align}
Since $\Phi + \Theta \lesssim\sqrt{|\kappa|+\eta}$, \eqref{e:Pbound2} simplifies to
\begin{align}\begin{split}\label{e:estimateP}
|P_\fa(z, m(z))|
&\prec \left(\frac{1}{N\eta}+\frac{d^{5/2}}{N^2}+\frac{1}{d^{3/2}}\right)^{1/2}\left(\frac{(\Phi+\Theta)\sqrt{|\kappa|+\eta}}{N\eta}\right)^{1/2}
\\
&+{ \left(\frac{\Lambdao}{d^{1/2}}\right)^{1/3}\frac{(\Phi+\Theta)^{1/3}(|\kappa|+\eta)^{1/3}}{N\eta}}
+\left(\frac{1}{d^{\fa/2}}+\frac{d}{N}+\frac{d^3}{N^2}+\frac{d^{3/2}}{N\sqrt{N\eta}}\right).
\end{split}\end{align}

To conclude the proof, we consider the cases $\kappa \leq 0$ and $\kappa \geq 0$ separately.
If $\kappa\leq 0$, then we have $\Phi(z)\asymp \sqrt{|\kappa|+\eta}$, and \eqref{e:estimateP} simplifies to
\begin{align}\begin{split}\label{e:estimateP3}
|P_\fa(z, m(z))|
&\prec \left(\frac{1}{N\eta}+\frac{d^{5/2}}{N^2}+\frac{1}{d^{3/2}}\right)^{1/2}\left(\frac{|\kappa|+\eta}{N\eta}\right)^{1/2}+\left(\frac{1}{d^{\fa/2}}+\frac{d}{N}+\frac{d^3}{N^2}+\frac{d^{3/2}}{N\sqrt{N\eta}}\right).
\end{split}\end{align}
Thanks to Proposition \ref{p:stable}, by taking $\delta(z)$ the right-hand side of \eqref{e:estimateP3} times $N^\fc$, we get
\begin{align}
|m( z)-\md( z)|\prec
\frac{1}{(N\eta)^{1/2}} \left(\frac{1}{N\eta}+\frac{d^{5/2}}{N^2}+\frac{1}{d^{3/2}}\right)^{1/2}+\frac{1}{\sqrt{|\kappa|+\eta}}\left(\frac{1}{d^{\fa/2}}+\frac{d}{N}+\frac{d^3}{N^2}+\frac{d^{3/2}}{N\sqrt{N\eta}}\right).
\end{align}
This finishes the proof of \eqref{e:stateinS}.

If $\kappa\geq 0$, then $\Phi(z)\asymp\eta/ \sqrt{|\kappa|+\eta}$, and \eqref{e:estimateP} simplifies to
\begin{align}\begin{split}\label{e:simplebound}
|P_\fa(z, m(z))|
&\prec 
\frac{1}{d^{\fa/2}}+\frac{d}{N}+\frac{d^3}{N^2}+\frac{1}{N^{1/2}d^{3/4}}+\frac{d^{3/2}}{N^{3/2}\eta^{1/2}}
+\frac{1}{N\eta^{1/2}}
\\
&+\left(\frac{\Lambdao}{d^{1/2}}\right)^{1/3}\left(\frac{(|\kappa|+\eta)^{1/3}}{N\eta}\Theta^{1/3}+\frac{(|\kappa|+\eta)^{1/6}}{N\eta^{2/3}}\right)\\
&+\left(\frac{1}{N\eta}+\frac{d^{5/2}}{N^2}+\frac{1}{d^{3/2}}\right)^{1/2}
\left(\frac{\sqrt{|\kappa|+\eta}}{N\eta}\right)^{1/2}\Theta^{1/2}.
\end{split}\end{align}
Thanks to Proposition \ref{p:stable}, by taking $\delta(z)$ to be the right-hand side of \eqref{e:simplebound} times $N^\fc$, we get
\begin{align}\begin{split}\label{e:outS0}
|m(z)-\md(z)|&\prec 
\frac{1}{\sqrt{|\kappa|+\eta}}\left(\frac{1}{d^{\fa/2}}+\frac{d}{N}+\frac{d^3}{N^2}+\frac{1}{N^{1/2}d^{3/4}}+\frac{d^{3/2}}{N^{3/2}\eta^{1/2}}+\frac{1}{N\eta^{1/2}}\right)\\
&+\left(\frac{\Lambdao}{d^{1/2}}\right)^{1/3}\left(\frac{1}{N\eta(|\kappa|+\eta)^{1/6}}\Theta^{1/3}+\frac{1}{N\eta^{2/3}(|\kappa|+\eta)^{1/3}}\right)\\
&
+\left(\frac{1}{N\eta}+\frac{d^{5/2}}{N^2}+\frac{1}{d^{3/2}}\right)^{1/2}
\left(\frac{1}{N\eta\sqrt{|\kappa|+\eta}}\right)^{1/2}\Theta^{1/2}.
\end{split}\end{align}
Since the exponent of $\Theta$ on the right-hand side of \eqref{e:outS0} is less than $1$, by iterating \eqref{e:outS0} a bounded number of times and recalling the definition of $\prec$, we get
\begin{align}\begin{split}\label{e:outS}
&\phantom{{}={}}|m(z)-\md(z)|\prec 
\left(\frac{1}{d}+\frac{d}{N}+\frac{1}{(dN\eta)^{1/2}}\right)^{1/3}\frac{1}{N\eta^{2/3}(|\kappa|+\eta)^{1/3}}\\
&+\frac{1}{\sqrt{|\kappa|+\eta}}\left(\frac{1}{d^{\fa/2}}+\frac{d}{N}+\frac{d^3}{N^2}+\frac{1}{N^{1/2}d^{3/2}}+\frac{d^{3/2}}{N^{3/2}\eta^{1/2}}+\frac{1}{N\eta^{1/2}}+\frac{1}{d^{3/2}N\eta}+\frac{1}{(N\eta)^2}\right).
\end{split}\end{align}
This finishes the proof of \eqref{e:stateoutS}.
\end{proof}

\section{ Edge universality: proof of Theorem \ref{thm:univ}} \label{sec:universality}

In this section we prove the edge universality of random $d$-regular graphs in the regime $N^{2/9}\ll d\ll N^{1/3}$ where Theorem \ref{t:eigloc} provides optimal bounds on the extremal eigenvalues. Thus, throughout this section we assume that there exists a universal constant $\fe$ such that 
\begin{align}\label{e:dbound}
N^{2/9+\fe}\leq d\leq N^{1/3-\fe}.
\end{align}

Our strategy is based on the usual three-step approach of random matrix theory \cite{MR3699468}. Starting from the (rescaled) adjacency matrix $H = H_0$ from \eqref{def_H}, we run a matrix-valued Brownian motion $H(t)$. Using the rigidity estimates from Theorems \ref{t:eigloc} and \ref{thm:edgerigidity} combined with \cite{Ben-Yau}, we deduce that for $t \gg N^{-1/3}$ the matrix $H(t)$ has GOE edge statistics; see Proposition \ref{t:universalityHt} below. The main work in this section is a comparison argument to show that $H_0$ and $H(t)$, $t \gg N^{-1/3}$, have the same edge statistics; see Proposition \ref{thm:comp} below. This comparison argument has the same spirit as the one from \cite{MR3729611}. Its underlying principle is that for large $d$, a Markovian switching dynamics of the graph that leaves the random regular graph invariant is well approximated by Brownian motion, when considering observables that characterize the local spectral statistics. The main difference to \cite{MR3729611} is that, since we are working at the edge, we need to incorporate precise rigidity estimates on the locations of the eigenvalues near the spectral edge, and the necessary cancellations are more delicate.

We recall the constrained GOE $W$ as introduced in \cite[Section 2.1]{MR3729611}: $W$  is the centred Gaussian process on the space $\cal M \deq \h{H \in \R^{N \times N} \col H = H^*, H \bld 1 = 0}$ with covariance $\E \scalar{W}{X} \scalar{W}{Y} = \scalar{X}{Y}$, where $\scalar{X}{Y} \deq \frac{N}{2} \Tr (XY)$ for $X,Y \in \cal M$. Explicitly, its covariance is given by
\begin{equation} \label{W_ibp}
\E [W_{ij} W_{kl}] = \frac{1}{N} \pbb{\delta_{ik} - \frac{1}{N}} \pbb{\delta_{jl} - \frac{1}{N}}
+ \frac{1}{N} \pbb{\delta_{il} - \frac{1}{N}} \pbb{\delta_{jk} - \frac{1}{N}}.
\end{equation}
It may be viewed as the usual Gaussian Orthogonal Ensemble restricted to matrices with vanishing row and column sums. The following result is a straightforward consequence of \eqref{W_ibp}.

\begin{lemma}\label{p:intbypart}
For the constrained GOE $W$ we have the integration by parts formula
\begin{align}\label{e:intbypart}
\bE[W_{ij}F(W)]=\frac{1}{N^3}\sum_{kl}\bE[\del_{ij}^{kl} F(W)].
\end{align}
\end{lemma}

Next, we define the matrix-valued process
\begin{align}\label{e:Ht}
H(t) \deq \ee^{-t/2}H+\sqrt{1-\ee^{-t}} \, W,
\end{align}
where $H$ was defined in \eqref{def_H} in terms of the adjacency matrix of the $d$-regular graph. Thus, $H(0)=H$ and $H(\infty) = W$.

We recall the projection matrix  $P_{\perp}=I-{\bf 1}{\bf 1^*}/N$ from Section \ref{sec:notation}. Then the matrix $H(t)$ and $P_\perp$ commute, i.e.\ $H(t)P_\perp=P_\perp H(t)$.
For $z \in \bC_+ =\{z\in \bC\col \Im[z]>0\}$, we define the \emph{time-dependent Green's function} by
\begin{equation}
  G(t; z) \deq P_\perp (H(t)-z)^{-1}P_{\perp},
\end{equation}
so that $G(t;z)$ and $(H(t)-z)^{-1}$ agree on the image of $P_\perp$, i.e.\ the subspace of $\bR^N$ perpendicular to $\bm 1$ which carries the nontrivial spectrum of $H(t)$. The matrix $H(t)$ has a trivial eigenvalue $\la_1(t)=\ee^{-t/2}d/\sqrt{d-1}$ with eigenvector $\bf1$. We denote the remaining eigenvalues of $H(t)$ by $\la_2(t)\geq \la_{3}(t)\geq \cdots \la_{N-1}(t)\geq \la_N(t)$, and the Stieltjes transform of the empirical eigenvalue distribution of $H(t) P_\perp$ by $m(t;z)$,
\begin{align} \label{def_mtz}
m(t;z)
=\frac{1}{N}\Tr G(t;z)=\frac{1}{N}\sum_{i=2}^N\frac{1}{\la_i(t)-z}.
\end{align}

\begin{definition}
An event holds with \emph{very high probability} if for any $\fc > 0$ it has probability at least $1 - \OO_\fc(N^{-1/\fc})$.
\end{definition}

Under the assumption \eqref{e:dbound}, we have the following corollary of Theorems \ref{t:eigloc} and \ref{thm:edgerigidity}.
\begin{corollary}\label{c:rigidity3}
Fix a constant $\fe>0$ and suppose that $N^{2/9+\fe}\leq d\leq N^{1/3-\fe}$. Then, for any fixed $\fc>0$ we have with very high probability
\begin{align}
|\la_2(0)-2|, |\la_N(0)+2|\leq N^{-2/3+\fc},
\end{align}
and uniformly for any $z=E+\ri \eta$, with $-4\leq E\leq 4$ and $\eta\geq N^{-2/3}$,
\begin{align}
|m(0;z)-\md(z)|\leq N^\fc\left( \frac{1}{N\eta}+\frac{d}{N\sqrt{\eta}}\right).
\end{align}
\end{corollary}

\begin{remark}
As an easy consequence of Corollary \ref{c:rigidity3}, we have that for any $z=E+\ri\eta$, with $-4\leq E\leq 4$ and $\eta\geq N^{-2/3+\fc}$, $\Im[m(0;z)]\asymp \Im[\md(z)]$.
\end{remark}

\subsection{Free convolution}\label{s:fc}
The asymptotic eigenvalue density of the matrix $H(t)$ is governed by the free additive convolution of the rescaled Kesten--McKay measure with the semicircle law at time $s = 1 - \ee^{-t}$. We recall some properties of measures obtained by the free convolution with a semicircle distribution from \cite{MR1488333}. The semicircle density $\rhosc(x)$ is given by \eqref{def_sc}, and the semicircle density of variance $s$ is $s^{-1/2}\rhosc(s^{-1/2}x)$. Given a probability measure $\mu$ on $\R$, we denote its free convolution with a semicircle distribution of variance $s$ by $\mu_s$. The Stieltjes transforms of $\mu$ and $\mu_s$ are given by $G_\mu(z)=\int \frac{\rd \mu(x)}{x-z}$ and $G_{\mu_s}(z)=\int \frac{\rd \mu_{s}(x)}{x-z}$ respectively. Then the following holds \cite{MR1488333}.
\begin{enumerate}[(i)]
\item We denote the set $U_s = \{z\in\bC_+: \int \frac{\rd \mu(x)}{|z-x|^2}<\frac{1}{s}\}$. Then $z \mapsto z-s G_\mu(z)$ is a homeomorphism from $\bar{U}_s$ to $\bC_+\cup \bR$ and conformal from $U_s$ to $\bC_+$. We denote its inverse by $F_{\mu_s}: \bC_+\cup \bR\mapsto \bar U_s$.
\item The Stieltjes transform of $\mu_s$ is characterized by $G_\mu(z)=G_{\mu_s}(z-sG_\mu(z))$, for any $z\in U_s$.
\end{enumerate} 

By the inversion formula for the Stieltjes transform, we deduce from (ii) that the density of $\mu_s$ is given by $\rd \mu_s(x) / \rd x= \Im [G_\mu(F_{\mu_s}(x))]/\pi$.

The asymptotic eigenvalue density of $W$ is the semicircle density $\rhosc(x)$ and the asymptotic eigenvalue density of $\sqrt{1 - \ee^{-t}} \, W$ is the semicircle density at time $s = 1 - \ee^{-t}$.
The asymptotic eigenvalue density of $H(t)$ is the free convolution of rescaled Kesten--McKay law $\mu(\dd x) = \ee^{t/2} \rho_d(\ee^{t/2} x) \dd x$ at  time $\ee^{-t}$ and the semicircle density at time $1 - \ee^{-t}$. We denote its density by $\rho_d(t;x)$ and its Stieltjes transform by $\md(t;z) = \int \frac{\rho_d(t;x)}{x - z} \dd x$.
 Since $G_\mu(\ee^{-t/2} z) = \ee^{t/2} m_d(z)$,  we deduce from (i) and (ii) above that
\begin{align}\label{e:selffc}
\md(t;\xid(t;z))=\ee^{t/2}\md(z),
\end{align}
where
\begin{align} 
\xid(t;z)\deq \ee^{-t/2}z-\ee^{t/2}(1-\ee^{-t})\md(z)
\end{align}
is a homeomorphism from the set $\{z\in \bC_+: \int \frac{\rho_d(x)}{|x-z|^2}\rd x\leq\frac{1}{e^t-1}\}$ to $\bC_+\cup \bR$. We denote the functional inverse of $\xid(t;z)$ by $F_d(t;z)$ which is a homeomorphism from $\bC_+\cup \bR$ to the set $\{z\in \bC_+: \int \frac{\rho_d(x)}{|x-z|^2}\rd x\leq\frac{1}{e^t-1}\}$. Thus, $\rho_d(t;x)=\Im[m_d(F_d(t;x))]/\pi$. To find the support of the measure $\rho_d(t;x)$, we notice that there exists $z_d^{\pm}(t)\in \bR$ such that $\{z\in \bC_+: \int \frac{\rho_d(x)}{|x-z|^2}\rd x=\frac{1}{e^t-1}\}$ consists of the intervals $(-\infty, z_d^-(t)]\cup [z_d^+(t), \infty)$ and an arc $\cC$ from $z_d^-(t)$ to $z_d^+(t)$. Those two endpoints $z_{d}^\pm (t)\in\bR$ are the largest and smallest real solutions to
\begin{align}\label{e:edgeeqn1}
\int \frac{\rho_d(x)}{|x-z|^2}\rd x=\frac{1}{e^t-1},
\end{align}
and the function $F_d(t;x)$ maps the support of $\rho_d(t;x)$ to the arc $\cC$. As a consequence,  the right and left edges of the measure $\rho_d(t;x)$ are given by
\begin{align}\label{e:edgeeqn2}
E_{d}^\pm(t)=\xi_d(t;z_d^\pm(t))=\ee^{-t/2}z_{d}^\pm(t)-\ee^{t/2}(1-\ee^{-t})\md(z_{d}^\pm(t)).
\end{align}

\begin{lemma}\label{p:edgeloc}
Let $L_t=2+t/d$. Then the right and left edges of the measure $\rho_d(t;x)$ are given by
\begin{align}\label{e:edgeloc}
E_{d}^\pm(t)=\pm\left(L_t+\OO\left(t^3+\frac{1}{d^3}\right)\right).
\end{align}
\end{lemma}
\begin{proof}
From \eqref{e:sc_sce} and \eqref{e:md} we recall that for any $z\in \bC$,
\begin{align}\label{e:zmdmsc}
z=-\msc(z)-\frac{1}{\msc(z)},\quad \md(z)=\frac{\msc(z)}{1-\msc^2(z)/(d-1)}.
\end{align}
By taking the derivative we obtain
\begin{align}\label{e:md'}
\md'(z)=\frac{\msc^2(z)(1+\msc^2(z)/(d-1))}{(1-\msc^2(z))(1-\msc^2(z)/(d-1))^2}.
\end{align}
We can solve for $z_d^\pm(t)$ using \eqref{e:edgeeqn1} and \eqref{e:md'},
\begin{align}\label{e:et}
\frac{1}{e^t-1}=\int \frac{\rho_d(x)}{|x-z_d^\pm(t)|^2}\rd x
=m_d'(z_d^\pm(t))
=\frac{\msc^2(z_d^\pm(t))(1+\msc^2(z_d^\pm(t))/(d-1))}{(1-\msc^2(z_d^\pm(t)))(1-\msc^2(z_d^\pm(t))/(d-1))^2}.
\end{align}
In the regime $t\ll 1 $ and $d\gg1$, we can solve for $\msc(z_d^\pm(t))$ from \eqref{e:et} and get
\begin{align}\label{e:msceqn}
\msc(z_{d}^{\pm}(t))=\pm\left(1-\frac{t}{2}-\frac{3t}{2d}-\frac{t^2}{8}+\OO\left(t^3+\frac{1}{d^3}\right)\right).
\end{align}
Using \eqref{e:zmdmsc}, \eqref{e:msceqn} implies that
\begin{align}\label{e:zdteqn}\begin{split}
&z_d^\pm(t)=-\msc(z_d^\pm(t))-\frac{1}{\msc(z_d^\pm(t))}
=\mp\left(2+\frac{t^2}{4}+\OO\left(t^3+\frac{1}{d^3}\right)\right),\\
&m_d(z_d^\pm(t))=\frac{\msc(z_d^\pm(t))}{1-\msc^2(z_d^\pm(t))/(d-1)}=\pm\left(1-\frac{t}{2}+\frac{1}{d}+\OO\left(t^2+\frac{1}{d^2}\right)\right).
\end{split}\end{align}
Lemma \ref{p:edgeloc} follows by plugging \eqref{e:zdteqn} into \eqref{e:edgeeqn2} and expanding the exponents to third order.
\end{proof}

\subsection{Rigidity and edge universality of $H(t)$}

In this section we collect some estimates on the Green's function $G(t;z)$ and the Stieltjes transform $m(t;z)$ of $H(t)$,
and state the edge universality result for $H(t)$ when $t \gg N^{-1/3}$.
All statements and estimates in this section follow directly from \cite{MR3690289,Ben-Yau, AH}.

First, using the rigidity estimates of Corollary~\ref{c:rigidity3} as input,
the rigidity estimates on the Stieltjes transform $m(t;z)$ of $H(t)$ follow from \cite{AH, Ben-Yau}.

\begin{proposition}\label{p:rigidity2}
Fix a constant $\fe>0$ and $N^{2/9+\fe}\leq d\leq N^{1/3-\fe}$. For an arbitrarily small constant $\fc>0$  and any time $0\leq t\ll 1$, with very high probability we have
\begin{align}
|\la_2(t)-L_t|, |\la_N(t)+L_t|\leq N^{-2/3+\fc},
\end{align}
and uniformly for any $z=E+\ri \eta$, with $-4\leq E\leq 4$ and $\eta\geq N^{-2/3}$,
\begin{align}\label{e:mtmdt}
|m(t;z)-\md(t;z)|\leq N^\fc\left( \frac{1}{N\eta}+\frac{d}{N\sqrt{\eta}}\right),
\end{align}
where $L_t=2+t/d$ and $\md(t;z)$ is defined by \eqref{e:selffc} and $m(t;z)$ by \eqref{def_mtz}.
\end{proposition}

Using the rigidity estimates from Corollary~\ref{c:rigidity3} and the estimates on the Green's function entries of $H(0)$
from Proposition~\ref{p:newboundG} as input,
the entrywise estimates on Green's function of $H(t)$ with $t>0$ follow from an argument similar to the proof of \cite[Theorem 2.1]{MR3690289}.
We remark that in the statement of \cite[Theorem 2.1]{MR3690289}, it assumed that $\Im[m_0]$ is bounded from below and that $t\gg \eta_*$.
However, in the proof of \cite[Theorem 2.3]{MR3690289}, these assumptions are only used to show that $|m_t(z)-m_{{\rm fc}, t}(z)|$ is small.
With the required estimate of $|m_t(z)-m_{{\rm fc}, t}(z)|$ already established by \eqref{e:mtmdt},
the remaining part of \cite[Theorem 2.1]{MR3690289} does not use that $\Im[m_0]$ is bounded from below or that $t\gg \eta_*$.
Therefore, with \eqref{e:mtmdt} given,
the remaining proof of \cite[Theorem 2.1]{MR3690289} applies and gives the following result on the entrywise estimates of Green's function of $H(t)$.

\begin{proposition}\label{p:entry-wiselaw}
Fix constant $\fe>0$ and suppose that $N^{2/9+\fe}\leq d\leq N^{1/3-\fe}$. For an arbitrarily small constant $\fc>0$  and any time $0\leq t\ll 1$,
with very high probability we have
\begin{align}
|G_{ij}(t;z)-\delta_{ij}\md(t;z)|\leq N^\fc\left( \frac{1}{d^{1/2}}+\frac{1}{\sqrt{N\eta}}\right),
\end{align}
uniformly for any $z=E+\ri \eta$, with $-4\leq E\leq 4$ and $\eta\geq N^{-2/3}$.
\end{proposition}
The edge universality of $H(t)$ for $t \gg N^{-1/3}$ follows from the following result due to \cite{Ben-Yau}.

\begin{proposition}\label{t:universalityHt}
Fix a constant $\fe>0$ and suppose that $N^{2/9+\fe}\leq d\leq N^{1/3-\fe}$.
Let $\fd>0$ be a sufficiently small constant and set $t=N^{-1/3+\fd}$.
Let $H(t)$ be as in \eqref{e:Ht}, which has an eigenvalue $\la_1(t)=\ee^{-t/2}d/\sqrt{d-1}$, and we denote its remaining eigenvalues by $\la_2(t)\geq \la_{3}(t)\geq \cdots \la_{N-1}(t)\geq \la_N(t)$.  Fix $k\geq 1$ and $s_1,s_2,\cdots, s_k \in \R$. Then
\begin{align}\begin{split}\label{e:compGOE}
&\phantom{{}={}} \bP_{H(t)}\left( N^{2/3} ( \lambda_{i+1}(t) - L_t )\geq s_i,1\leq i\leq k \right)\\
&= \bP_{\mathrm{GOE}}\left( N^{2/3} ( \mu_i - 2  )\geq s_i,1\leq i\leq k \right) +\oo\left(1\right),
\end{split}\end{align}
where $L_t=2+t/d$ is as defined in Lemma \ref{p:edgeloc}, and $\mu_1\geq \mu_2\geq \cdots \geq\mu_N$ are the eigenvalues of a the Gaussian Orthogonal Ensemble.
The analogous statement holds for the smallest eigenvalues.
\end{proposition}

\begin{remark} \label{rem:independence}
By an appropriate modification of the analysis of Dyson Brownian motion from~\cite{Ben-Yau}, Proposition~\ref{t:universalityHt} also holds for the joint distribution of the $k$ largest and smallest eigenvalues. In particular, this implies that, under the same assumption as in the previous proposition, the asymptotic joint distribution of $N^{2/3}(\lambda_2(t) - L_t, -\lambda_N(t) - L_t)$ is a pair of independent Tracy--Widom$_1$ distributions. Moreover, Proposition~\ref{thm:comp} below can be extended, by merely cosmetic changes, to also cover the joint distribution of $k$ largest and smallest eigenvalues. Thus, we get the following extension of Theorem~\ref{thm:univ}.
\end{remark}

\begin{theorem} \label{thm:univ_gen}
Suppose that $N^{2/9} \ll d \ll N^{1/3}$.
Fix $k\geq 1$ and $s_1^\pm,s_2^\pm,\cdots, s_k^\pm \in \R$. Then
\begin{multline*}
\P_H \pB{N^{2/3}(\lambda_{i+1} - 2) \geq s_i^+, N^{2/3}(-\lambda_{N+1 - i} - 2) \geq s_i^- , 1 \leq i \leq k}
\\
= \P_{\mathrm{GOE}} \pB{N^{2/3}(\mu_{i} - 2) \geq s_i^+, N^{2/3}(-\mu_{N+1 - i} - 2) \geq s_i^- , 1 \leq i \leq k} + \oo(1).
\end{multline*}
\end{theorem}

\subsection{Green's function comparison}

In this section we prove the following short-time comparison result for the edge eigenvalue statistics of $H(t)$. 
\begin{proposition}\label{thm:comp}
Fix a constant $\fe>0$ and suppose that $N^{2/9+\fe}\leq d\leq N^{1/3-\fe}$.
Let $\fd>0$ be a sufficiently small constant and set $t=N^{-1/3+\fd}$. Let $H(t)$ be as in \eqref{e:Ht}, which has an eigenvalue $\la_1(t)=\ee^{-t/2}d/\sqrt{d-1}$, and we denote its remaining eigenvalues by $\la_2(t)\geq \la_{3}(t)\geq \cdots \la_{N-1}(t)\geq \la_N(t)$.    Fix $k\geq 1$ and $s_1,s_2,\cdots, s_k \in \R$. Then
\begin{align}\begin{split}\label{e:comp1}
&\phantom{{}={}} \bP_{H}\left( N^{2/3} ( \lambda_{i+1}(0) - 2 )\geq s_i,1\leq i\leq k \right)\\
&= \bP_{H(t)}\left( N^{2/3} ( \lambda_{i+1}(t) - L_t  )\geq s_i,1\leq i\leq k \right) +\oo(1),
\end{split}\end{align}
where $L_t=2+t/d$ is as defined in Lemma \ref{p:edgeloc}. The analogous statement holds for the smallest eigenvalues.
\end{proposition}

Before proving Proposition \ref{thm:comp} we use it to conclude the proof of edge universality of random $d$-regular graphs, Theorem \ref{thm:univ}.
\begin{proof}[Proof of Theorem \ref{thm:univ}]
Combine Propositions \ref{t:universalityHt} and \ref{thm:comp}.
\end{proof}

The rest of this section is devoted to the proof of Proposition \ref{thm:comp}.
For any $E$, we define 
\begin{align}
\cN_t(E)\deq |\{i\geq 2\col\la_i(t)\geq L_t+E\}|,
\end{align}
and we write $\cN_0(E)$ as $\cN(E)$. We take $\ell=N^{-2/3-\fd/9}$ and $\eta=N^{-2/3-\fd}$. Then with very high probability, from Propositions \ref{p:rigidity2} and \ref{p:entry-wiselaw},
 with \eqref{eqn:imtilminf} and \eqref{e:selffc},
and with the fact that $\eta\Im[m(t;L_t+\kappa+\ri \eta)]$ is a monotone decreasing function of $\eta$ (which is immediate
  from the spectral representation), for any $|\kappa|\leq N^{-2/3+\fd}$, we have 
\begin{align}\label{e:mNtbound}
\Im[m(t; L_t+\kappa+\ri\eta)]\leq N^{-1/3+C\fd},
\end{align}
and  similarly, since $\max_{ij} \eta |G_{ij}(t; L_t+\kappa+\i\eta)|$ is decreasing in $\eta$  (see \cite[Lemma 2.1]{MR3688032}),
\begin{align}\label{e:Gtbound}
\max_{ij} |G_{ij}(t; L_t+\kappa+\ri\eta)|\leq  N^{C\fd}, \quad \sum_{ij}|G_{ij}(t; L_t+\kappa+\ri\eta)|^2\leq  N^{4/3+C\fd}.
\end{align}
Moreover,  we have $|\del_z m(t;z)|\leq \Im[m(t;z)]/\Im[z]\leq N^{1/3+C\fd}$ for any $z$ with $\left|\Re[z]-L_t\right|\leq N^{-2/3+\fd}$ and $N^{-2/3-\fd}\leq \Im[z]\leq N^{-2/3}$. Therefore, by integrating from $z=2+\kappa+\ri \eta$ to $z=2+\kappa+\ri N^{-2/3}$, we get with very high probability
\begin{align}\begin{split}\label{e:mtedge}
&\phantom{{}={}}m(t; L_t+\kappa+\ri\eta)
=m(t; L_t+\kappa+\ri N^{-2/3})+\OO(N^{-1/3+C\fd})\\
&=\md(t; L_t+\kappa+\ri N^{-2/3})+\OO(N^{-1/3+C\fd})=\md(t; E^+_d(t))+\OO(N^{-1/3+C\fd})
=-1+\OO(1/d)
\end{split}\end{align}
where we used  Lemma \ref{p:edgeloc}, Proposition \ref{p:rigidity2} and the square root behavior of the density $\rho_d(t;x)$:
\begin{align}
|\md(t; L_t+\kappa+\ri N^{-2/3})-\md(t; E^+_d(t))|\asymp \sqrt{|L_t+\kappa+\ri N^{-2/3}-E^+_d(t)|}\lesssim N^{-1/3+C\fd},
\end{align}
which follows from \cite[Proposition A.1]{AH}.

Next, we define
\begin{align}
\chi_E(x)={\bf 1}_{[E, N^{-2/3+\fd}]}(x-L_t),\quad \theta_\eta(x)\deq \frac{\eta}{\pi(x^2+\eta^2)}=\frac{1}{\pi}\Im\frac{1}{x+\ri \eta}.
\end{align}
We have that $\la_1(t)=\ee^{-t/2}d\gg L_t$ and with very high probability it holds $\la_2(t)\leq L_t+N^{-2/3+\fd}$ by Proposition \ref{p:rigidity2}. By the same argument as in \cite[Lemma 2.7]{MR3034787}, we get that 
\begin{align}\label{e:resolventexp}
\Tr(\chi_{E+\ell}*\theta_\eta)(H(t))-N^{-\fd/9}\leq \cN_t(E)\leq \Tr(\chi_{E-\ell}*\theta_\eta)(H(t))+N^{-\fd/9}
\end{align}
with very high probability. Let $K_i\col \bR\mapsto[0,1]$ be a monotonic smooth function satisfying,
\begin{align}
K_i(x)=\left\{\begin{array}{cc}
0 & \text{if } x\leq i-2/3,\\
1& \text{if } x\geq i-1/3.
\end{array}\right.
\end{align}
Since ${\bf 1}_{\cN_t(E)\geq i}=K_i(\cN_t(E))$ and $K_i$ is monotonically increasing, we have 
\begin{align}
K_i\left(\Tr(\chi_{E+\ell}*\theta_\eta)(H(t))\right)+\OO(N^{-\fd/9})\leq {\bf 1}_{\cN_t(E)\geq i}\leq K_i\left(\Tr(\chi_{E-\ell}*\theta_\eta)(H(t))\right)+\OO(N^{-\fd/9}).
\end{align}
In this way we can express the locations of eigenvalues in terms of integrals of the Stieltjes transform of the empirical eigenvalue densities:
\begin{align}\begin{split}\label{e:sandwich}
&\phantom{{}={}}\bE_{H(t)}\left[\prod_{i=1}^k K_i\left(\Im\left[\frac{N}{\pi}\int_{s_iN^{-2/3}+\ell}^{N^{-2/3+\fd}} m(t;L_t+y+\ri \eta) \,  \rd y  \right]\right)\right]+\OO\left(N^{-\fd/9}\right)\\
&\leq \bP_{H(t)}\left( N^{2/3} ( \lambda_{i+1}(t) - L_t)\geq s_i,1\leq i\leq k \right) =\bE\left[\prod_{i=1}^k{\bf 1}_{\cN_t(s_iN^{-2/3})\geq i}\right] \\
&\leq \bE_{H(t)}\left[\prod_{i=1}^k K_i\left(\Im\left[\frac{N}{\pi}\int_{s_iN^{-2/3}-\ell}^{N^{-2/3+\fd}} m(t; L_t+y+\ri \eta) \,  \rd y \right]\right)\right]+\OO\left(N^{-\fd/9}\right).\end{split}
\end{align}

For the product of the functions of Stieltjes transform, we have the following comparison theorem.

\begin{proposition}\label{p:comp}Fix constant $\fe>0$ and $N^{2/9+\fe}\leq d\leq N^{1/3-\fe}$.
Let $\fd>0$ be sufficiently small and set $t=N^{-1/3+\fd}$ and $\eta=N^{-2/3-\fd}$. Let $F:\bR^k\mapsto \bR$ be a fixed smooth test function. Then for $E_1,E_2,\cdots, E_k=\OO( N^{-2/3})$ we have
\begin{align}\begin{split}\label{e:comparison} 
&\phantom{{}={}}\bE_{H}\left[F\left(\left\{\Im\left[N\int_{E_i}^{N^{-2/3+\fd}} m(0; L_0+y+\ri \eta) \,  \rd y \right]\right\}_{i=1}^k\right)\right]\\
&=\bE_{H(t)}\left[\left\{F\left(\Im\left[N\int_{E_i}^{N^{-2/3+\fd}} m(t;L_t+y+\ri \eta) \, \rd y \right]\right\}_{i=1}^k\right)\right]+\OO\left(\frac{N^{2/3+C\fd} t}{d^{3/2}}+\frac{N^{1/3+C\fd}t}{d^{1/2}}\right).
\end{split}\end{align}
\end{proposition}

\begin{proof}[Proof of Proposition \ref{thm:comp}]
Since $d\geq N^{2/9+\fe}$ and $t=N^{-1/3+\fd}$, for $\fd$ small, the error terms in \eqref{e:comparison}  are of order $\OO(N^{-\fd})$.
By combining \eqref{e:sandwich} and \eqref{e:comparison}, we thus get
\begin{align}\begin{split}
& \phantom{\leq} \; \; \bP_{H(t)}\left( N^{2/3} ( \lambda_{i+1}(t) - L_t )\geq s_i+2N^{2/3}\ell,1\leq i\leq k \right)\\
&\leq \bE_{H(t)}\left[\prod_{i=1}^k K_i\left(\Im\left[\frac{N}{\pi}\int_{s_iN^{-2/3}+\ell}^{N^{-2/3+\fd}} m(t; L_t+y+\ri \eta) \,  \rd y \right]\right)\right]+\OO\left(N^{-\fd/9}\right)\\
&\leq \bP_{H}\left( N^{2/3} ( \lambda_{i+1}(0) - L_0 )\geq s_i,1\leq i\leq k \right) +\OO\left(N^{-\fd/9}\right)\\
&\leq \bE_{H(t)}\left[\prod_{i=1}^k K_i\left(\Im\left[\frac{N}{\pi}\int_{s_iN^{-2/3}-\ell}^{N^{-2/3+\fd}} m(t;L_t+y+\ri \eta) \,  \rd y \right]\right)\right]+\OO\left(N^{-\fd/9}\right)\\
&\leq \bP_{H(t)}\left( N^{2/3} ( \lambda_{i+1}(t) - L_t )\geq s_i-2N^{2/3}\ell,1\leq i\leq k \right)+\OO(N^{-\fd/9}).
\end{split}
\end{align}
Since $N^{2/3}\ell=N^{-\fd/9}\ll1$, \eqref{e:comp1} follows. 
\end{proof}

\begin{proof}[Proof of Proposition \ref{p:comp}]
For simplicity of notation, we only prove the case $k=1$; the general case can be proved in the same way. Let
\begin{align}
X_t=\Im\left[N\int_{E}^{N^{-2/3+\fd}}m(t; L_t +y+\ri \eta) \, \rd y\right].
\end{align}
We shall prove that 
\begin{align}\label{e:onecomp}
|\bE[F(X_t)]-\bE[F(X_0)]|\leq N^{C\fd}\left(\frac{N^{2/3} t}{d^{3/2}}+\frac{N^{1/3}t}{d^{1/2}}\right).
\end{align}
The derivative of $\bE[F(X_t)]$ with respect to the time $t$ is
\begin{align}\begin{split}\label{e:derF}
\frac{\rd}{\rd t}\bE[F(X_t)]
&=\bE\left[F'(X_t)\frac{\rd X_t}{\rd t}\right]
=\bE\left[F'(X_t)\Im\int_{E}^{N^{-2/3+\fd}}\left(N \sum_{ij}\dot{H}_{ij}(t)\frac{\del m}{\del H_{ij}}+\dot{L}_t\sum_{ij}G_{ij}^2\right)\rd y\right]\\
&=\bE\left[F'(X_t)\Im\int_{E}^{N^{-2/3+\fd}}\left(-\sum_{ija}\dot{H}_{ij}(t)G_{ai}G_{aj}+\dot{L}_t\sum_{ij}G_{ij}^2\right)\rd y\right]
\end{split}\end{align}
where we abbreviate $G(t; L_t+ y +\ri \eta)$ and $m(t; L_t+y +\ri \eta)$ by $G$ and $m$ respectively.
In the following, we estimate the right-hand side. By definition,
\begin{align}\label{h:derive}
\dot{H}_{ij}(t)=-\frac{1}{2}\ee^{-t/2}H_{ij}+\frac{\ee^{-t}}{2\sqrt{1-\ee^{-t}}}W_{ij}.
\end{align}
In the following, we use the notation $\partial_{ij}^{kl}$ applied to a function of $H(t)$, such as $G$ or $m$, to denote the directional derivative \eqref{e:D-defn} with respect to $H(t)$.  For any directional derivative $\partial$ we therefore have
\begin{equation*}
\frac{\partial}{\partial H} F(H(t)) = \ee^{-t/2} \, \partial F(H(t)), \qquad
\frac{\partial}{\partial W} F(H(t)) = \sqrt{1 - \ee^{-t}} \, \partial F(H(t)).
\end{equation*}
Plugging \eqref{h:derive} into \eqref{e:derF}, and using the Gaussian integration by parts \eqref{e:intbypart}, we therefore obtain
\begin{align}\begin{split}\label{e:derexp}
&-\sum_{ija}\bE\left[\dot H_{ij}(t)F'(X_t) G_{ai}G_{ja}\right]
=\frac{1}{2}\sum_{ija}\bE\left[\left(\ee^{-t/2}H_{ij}-\frac{\ee^{-t}}{\sqrt{1-\ee^{-t}}}W_{ij}\right)F'(X_t) G_{ai}G_{ja}\right]\\
&=\frac{\ee^{- t / 2}}{2 \sqrt{d-1}}\sum_{ ija}\bE\left[A_{ij}(F'(X_t)G_{ai}G_{ja})\right] -\frac{\ee^{-t}}{2 N^3}\sum_{ijkla}\bE[\del_{ij}^{kl}(F'(X_t)G_{ai}G_{ja})].
\end{split}\end{align}
To estimate the first term, we apply Corollary~\ref{c:intbp1} with the random variable $F_{ij}=F'(X_t) G_{ai} G_{ja}$
in \eqref{e:intbp1}.
Since
$\cal C_{ij}(F_{ij},A) \prec \abs{G_{ai} G_{ja}}$,
using the Ward identity \eqref{e:Ward} and $\sum_{ij} A_{ij} = Nd$,
the resulting error term is bounded by
\begin{equation}
  \frac{1}{\sqrt{d-1}}\sum_{ ija}
  \frac{d\bE[A_{ij}\cal C_{ij}(F_{ij},A)]}{N} \prec
  \frac{d^{1/2}}{N}  \sum_{ija} \E[A_{ij} \abs{G_{ai} G_{ja}}]
  \prec d^{3/2} \frac{\E[\Im m]}{\eta} \leq C d^{3/2} N^{1/3 + C\fd}\,;
\end{equation}
in the last inequality, we used \eqref{e:mNtbound} and that
$\eta = N^{-2/3 - \fd}$ and $\abs{\kappa} \leq N^{-2/3 + \fd}$.
In summary,
Taylor expanding the discrete derivative in \eqref{e:intbp1}
and noting that the first term on the right-hand side of \eqref{e:intbp1} vanishes by $\sum_i G_{ai} = 0$,
we find that \eqref{e:derexp} is bounded by
\begin{align}\begin{split} \label{e:derexp-bis}
&\frac{\ee^{-t}}{2 Nd (d-1)}\sum_{ijkla}\bE\left[A_{ik}A_{jl}  \del_{ij}^{kl} (F'(X_t)G_{ai}G_{ja})\right] -\frac{\ee^{-t}}{2 N^3}\sum_{ijkla}\bE[\del_{ij}^{kl}(F'(X_t)G_{ai}G_{ja})]\\
&+\sum_{n = 2}^{ \fb }\frac{\ee^{-(n+1)t/2}}{2 Nd (d-1)^{(n+1)/2} n!}\sum_{ijkla}\bE\left[A_{ik}A_{jl}(\del_{ij}^{kl})^n(F'(X_t)G_{ai}G_{ja})\right] +\OO\left(d^{3/2}N^{1/3+C\fd}\right)
\end{split}\end{align}
for some constant $\fb$.
The terms in \eqref{e:derexp-bis} are estimated in the following two claims;
we postpone their proofs to the end of this section.

\begin{claim}\label{c:J1}
For the first two terms in \eqref{e:derexp-bis} we have,
\begin{align}\begin{split}\label{e:J1}
&\phantom{{}={}}\frac{\ee^{-t}}{2Nd (d-1)}\sum_{ijkla}\bE\left[A_{ik}A_{jl}\del_{ij}^{kl}(F'(X_t)G_{ai}G_{ja})\right] -\frac{\ee^{-t}}{2N^3}\sum_{ijkla}\bE[\del_{ij}^{kl}(F'(X_t)G_{ai}G_{ja})]\\
&=\frac{3}{d}\sum_{ij}\bE[F'(X_t)G_{ij}^2]+\OO\left(\frac{N^{1+C\fd}}{d^{1/2}}\right).
\end{split}\end{align}
\end{claim}

\begin{claim}\label{c:Jn}
For any $n\geq 2$, let,
\begin{align}
J_n\deq 
\frac{\ee^{-(n+1)t/2}}{2Nd (d-1)^{(n+1)/2} n!}\sum_{ijkla}\bE\left[A_{ik}A_{jl}(\del_{ij}^{kl})^n(F'(X_t)G_{ai}G_{ja})\right].
\end{align}
Then, we have the estimates
\begin{align}\label{e:J2}
J_2=-\frac{12}{d}\sum_{ij}\bE[F'(X_t)G_{ij}^2]+\OO\left(\frac{N^{4/3+C\fd}}{d^{3/2}}+\frac{N^{1+C\fd}}{d^{1/2}}\right),
\end{align}
\begin{align}\label{e:J3}
J_3=\frac{8}{d}\sum_{ij}\bE[F'(X_t)G_{ij}^2]+\OO\left(\frac{N^{4/3+C\fd}}{d^{3/2}}+\frac{N^{1+C\fd}}{d}\right),
\end{align}
and for any $n\geq 4$, $J_n=\OO(N^{4/3+C\fd}/d^{(n-1)/2})$.
\end{claim}

It follows from Claims \ref{c:J1} and \ref{c:Jn} and \eqref{e:derexp-bis} that \eqref{e:derexp} can be estimated by
\begin{align}\begin{split}\label{e:derexp2}
-\sum_{ija}\bE\left[\dot H_{ij}(t)F'(X_t) G_{ai}G_{ja}\right]
&=\frac{3}{d}\sum_{ij}\bE[F'(X_t) G_{ij}^2]-\frac{12}{d}\sum_{ij}\bE[F'(X_t) G_{ij}^2 ]\\
&+\frac{8}{d}\sum_{ij}\bE[F'(X_t) G_{ij}^2]+\OO\left(\frac{N^{4/3+C\fd}}{d^{3/2}}+\frac{N^{1+C\fd}}{d^{1/2}}\right)\\
&=-\frac{1}{d}\sum_{ij}\bE[F'(X_t) G_{ij}^2]+\OO\left(\frac{N^{4/3+C\fd}}{d^{3/2}}+\frac{N^{1+C\fd}}{d^{1/2}}\right).
\end{split}\end{align}
From  Lemma~\ref{p:edgeloc}, we recall that $L_t=2+t/d$, i.e., $\dot{L}_t=1/d$.
Therefore plugging \eqref{e:derexp2} into \eqref{e:derF}, the two terms in \eqref{e:derF} cancel up to an error, and we get
\begin{align}
\frac{\rd}{\rd t}\bE[F(X_t)]=\bE\left[\int_E^{N^{-2/3+\fd}}\OO\left(\frac{N^{4/3+C\fd}}{d^{3/2}}+\frac{N^{1+C\fd}}{d^{1/2}}\right) \, \rd y\right]
=\OO\left(\frac{N^{2/3+C\fd}}{d^{3/2}}+\frac{N^{1/3+C\fd}}{d^{1/2}}\right),
\end{align}
and \eqref{e:onecomp} follows. This finishes the proof of Proposition \ref{p:comp}.
\end{proof}

\begin{proof}[Proof of Claim \ref{c:J1}]
We have
\begin{align}\label{e:derFGG}
\del_{ij}^{kl}(F'(X_t)G_{ai}G_{ja})=\del_{ij}^{kl}(F'(X_t))G_{ai}G_{ja}+F'(X_t)\del_{ij}^{kl}(G_{ai} G_{ja}).
\end{align}
For the first term in \eqref{e:derFGG}, with very high probability we have
\begin{align}\label{e:dF0}
|\del_{ij}^{kl}F'(X_t)|
=\left|F''(X_t)\Im\left[N\int_{E}^{N^{-2/3+\fd}}\del_{ij}^{kl}m(t; L_t +y+\ri \eta) \, \rd y\right]\right|\leq 
N^{-1/3+C\fd},
\end{align}
where we used \eqref{e:d2m} and \eqref{e:mNtbound}. Therefore, the sum arising from the first term in \eqref{e:derFGG} can be estimated as
\begin{align}\begin{split}\label{e:dF}
&\phantom{{}={}}\frac{\ee^{-t}}{2Nd (d-1)}\sum_{ijkla}\bE\left[A_{ik}A_{jl}\del_{ij}^{kl}(F'(X_t))G_{ai}G_{ja}\right]\\
&=\frac{\ee^{-t}}{2Nd (d-1)}\sum_{ijkla}\frac{d^2}{N^2}\bE\left[\del_{ij}^{kl}(F'(X_t))G_{ai}G_{ja}\right]+\OO\left(\frac{N^{1+C\fd}}{d^{1/2}}\right)\\
&=\frac{\ee^{-t}}{2N^3}\sum_{ijkla}\bE\left[\del_{ij}^{kl}(F'(X_t))G_{ai}G_{ja}\right]
+\OO\left(\frac{N^{1+C\fd}}{d^{1/2}}\right).
\end{split}\end{align}
where we used Corollary \ref{c:intbp}, \eqref{e:dF0} and that with very high probability $\sum_{ij} |(G^2)_{ij}|=N\Im[m(t;L_t+ y+\ri\eta)]/\eta\leq N^{4/3+C\fd}$ from \eqref{e:mNtbound}.
For the sum corresponding to the second term in \eqref{e:derFGG}, we use the notation $A(t)=\sqrt{d-1}H(t)$ and write
\begin{align}\begin{split}\label{e:stterm}
&\phantom{{}={}}\frac{\ee^{-t}}{2Nd (d-1)}\sum_{ijkla}\bE\left[A_{ik}A_{jl}F'(X_t)\del_{ij}^{kl}(G_{ai}G_{ja})\right]\\
&=\frac{1}{2Nd (d-1)}\sum_{ijkla}\bE\left[A_{ik}(t)A_{jl}(t)F'(X_t)\del_{ij}^{kl}(G_{ai}G_{ja})\right]\\
&-\frac{\ee^{-t/2}\sqrt{1-\ee^{-t}}}{2Nd(d-1)^{1/2} }\sum_{ijkla}\bE\left[(A_{ik} W_{jl}+W_{ik}A_{jl})F'(X_t)\del_{ij}^{kl}(G_{ai}G_{ja})\right]\\
&-\frac{(1-\ee^{-t})}{2Nd }\sum_{ijkla}\bE\left[W_{ik}W_{jl}F'(X_t)\del_{ij}^{kl}(G_{ai}G_{ja})\right].
\end{split}\end{align}
We can estimate the second and third term on the right-hand side of \eqref{e:stterm} using Lemma \ref{p:intbypart}, e.g.
\begin{align}\begin{split}
&\phantom{{}={}}\frac{ \ee^{-t/2} \sqrt{1-\ee^{-t}}}{2Nd(d-1)^{1/2} }\sum_{ijkla}\bE\left[W_{ik}A_{jl} F'(X_t)\del_{ij}^{kl}(G_{ai}G_{ja})\right]
\\
&=\frac{ \ee^{-t/2} (1-\ee^{-t})}{2N^4d(d-1)^{1/2} }\sum_{ijkli'k'a}\bE\left[A_{jl} \del_{ik}^{i'k'}(F'(X_t)\del_{ij}^{kl}(G_{ai}G_{ja}))\right]\\
&= \OO \pBB{ \frac{(1-\ee^{-t})}{2N^4d(d-1)^{1/2} }\sum_{ijkli'k'}\bE\left[A_{jl}\Im[m]/\eta\right]}
=\OO\left( \frac{N^{1+C\fd}}{d^{1/2}}\right),
\end{split}\end{align}
where in the third line we used \eqref{e:GGbd} and \eqref{e:mNtbound}.
In the same way, we in fact have that the second and third term on the right-hand side of \eqref{e:stterm} are all bounded by $\OO( N^{1+C\fd}/d^{1/2})$. 

In the following we estimate the first term on the right-hand side of \eqref{e:stterm}. These are terms in the form $\sum_{ijkla}A_{ik}(t)A_{jl}(t)F'(X_t)\times\{\txt{monomial of Green's function entries}\}$, where, we recall, the Green's function is that of $H(t) = (d - 1)^{-1/2} A(t)$. For them we can 
\begin{enumerate}
\item sum over indices which appear only once and use the relations $\sum_i A_{ij}(t)=\sum_j A_{ij}(t)=\ee^{-t/2}d$ and $\sum_iG_{ij}=\sum_jG_{ij}=0$ to get expressions involving $\Tr G^2$, $\Tr G^3$, or $A(t)G$;
\item simplify the expressions using the identity $A(t)G=GA(t)=(d-1)^{1/2}(H(t)-z+z)G=(d-1)^{1/2}(zG+P_\perp)$;
\item estimate the final expressions using \eqref{e:Gkbd}, \eqref{e:GGbd}, \eqref{e:mtedge}, $z=2+\OO(N^{-1/3+\fd})$, $\max_{i\neq j}|G_{ij}|\prec 1/\sqrt{d}$ and $\max_{i}|G_{ii}+1|\prec 1/\sqrt{d}$ from Proposition \ref{p:entry-wiselaw}. 
\end{enumerate}

Using the above procedure, we get that 
\begin{align}\begin{split}\label{e:firstder1}
&\phantom{{}={}}\sum_{ijkla}A_{ik}(t)A_{jl}(t)\del_{ij}^{kl}(G_{ai}G_{ja})\\
&=-2\ee^{-t}d^2(\Tr G^3+Nm\Tr G^2)+4(d-1)N \Tr G^2+\OO(N^{7/3+C\fd}),
\end{split}\end{align}
and
\begin{align}\label{e:thirdder1}
\sum_{ijkla}\del_{ij}^{kl}(G_{ai}G_{ja})
=-2N^2(\Tr G^3+Nm\Tr G^2).
\end{align}
It follows by combining  \eqref{e:dF}, \eqref{e:firstder1} and \eqref{e:thirdder1} that
\begin{align*}
&\phantom{{}={}}\frac{\ee^{-t}}{2Nd (d-1)}\sum_{ijkla}\bE\left[A_{ik}A_{jl}\del_{ij}^{kl}(F'(X_t)G_{ai}G_{ja})\right] -\frac{\ee^{-t}}{2N^3}\sum_{ijkla}\bE[\del_{ij}^{kl}(F'(X_t)G_{ai}G_{ja})]\\
&=\frac{1}{2Nd (d-1)}\sum_{ijkla}\bE\left[A_{ik}(t)A_{jl}(t)F'(X_t)\del_{ij}^{kl}(G_{ai}G_{ja})\right]
\\
&\phantom{{}={}} -\frac{\ee^{-t}}{2N^3}\sum_{ijkla}\bE[F'(X_t)\del_{ij}^{kl}(G_{ai}G_{ja})]+\OO\left(\frac{N^{1+C\fd}}{d^{1/2}}\right)\\
&=-\frac{\ee^{-t}}{(d-1)N}\bE\left[F'(X_t)(\Tr G^3+Nm\Tr G^2)\right]+\frac{2}{d}\bE[F'(X_t)\Tr G^2]+\OO\left(\frac{N^{1+C\fd}}{d^{1/2}}\right)\\
&=\frac{3}{d}\bE[F'(X_t)\Tr G^2]
+\OO\left(\frac{N^{1+C\fd}}{d^{1/2}}\right),
\end{align*}
where in the last line we used \eqref{e:Gkbd} and \eqref{e:mtedge}.
\end{proof}

\begin{proof}[Proof of Claim \ref{c:Jn}]
For \eqref{e:J2}, similarly to \eqref{e:dF0}, we have $|(\del_{ij}^{kl})^2F'(X_t)|\leq N^{-1/3+C\fd}$ with very high probability and
\begin{align}\label{e:J2est-1}
J_2= 
\frac{\ee^{-3t/2}}{4Nd (d-1)^{3/2}}\sum_{ijkla}\bE\left[A_{ik}A_{jl}F'(X_t)(\del_{ij}^{kl})^2(G_{ai}G_{ja})\right]+\OO\left(\frac{N^{1+C\fd}}{d^{1/2}}\right).
\end{align}
Thanks to Proposition \ref{p:entry-wiselaw}, we have $\max_{i\neq j} |G_{ij}|\prec 1/\sqrt{d}$. Those terms from $(\del_{ij}^{kl})^2(G_{ai}G_{ja})$ that contain four off-diagonal terms yield a contribution of the form
\begin{align}\label{e:J2est0}
\frac{\ee^{-3t/2}}{4Nd (d-1)^{3/2}}\sum_{ijkla}\bE\left[A_{ik}A_{jl}F'(X_t)\{\txt{terms with $4$ off-diagonal terms}\}\right]=\OO\left(\frac{N^{4/3+C\fd}}{d^{3/2}}\right).
\end{align}
The leading contributions are from those terms from $(\del_{ij}^{kl})^2(G_{ai}G_{ja})$ which contain two or three off-diagonal Green's function entries. By the same estimate as in \eqref{e:stterm}, we have
\begin{align}\begin{split}\label{e:3offterm}
&\phantom{{}={}}\frac{\ee^{-3t/2}}{4Nd (d-1)^{3/2}}\sum_{ijkla}\bE\left[A_{ik}A_{jl}F'(X_t)\{\txt{terms with $\leq 3$ off-diagonal terms}\}\right]\\
&=\frac{\ee^{-t/2}}{4Nd (d-1)^{3/2}}\sum_{ijkla}\bE\left[A_{ik}(t)A_{jl}(t)F'(X_t)\{\txt{terms with $\leq 3$ off-diagonal terms}\}\right]
+\OO\left(\frac{N^{1+C\fd}}{d}\right).
\end{split}\end{align}
Those terms in \eqref{e:3offterm} can be treated by the same procedure as described in the proof of Claim~\ref{c:J1}, and we get
\begin{align}\label{e:J2est}
\eqref{e:3offterm}=-\frac{12}{d}\sum_{ij}\bE[F'(X_t)G_{ij}^2]+\OO\left(\frac{N^{4/3+C\fd}}{d^{3/2}}+\frac{N^{1+C\fd}}{d^{1/2}}\right).
\end{align}
The claim \eqref{e:J2} follows from combining \eqref{e:J2est-1}, \eqref{e:J2est0}, \eqref{e:3offterm} and \eqref{e:J2est}.

In the following we prove \eqref{e:J3}. Similarly to \eqref{e:dF0}, we have $|(\del_{ij}^{kl})^3F'(X_t)|\leq N^{-1/3+C\fd}$ with very high probability and
\begin{align}\label{e:J3est1}
J_3=
\frac{\ee^{-2t}}{12Nd (d-1)^{2}}\sum_{ijkla}\bE\left[A_{ik}A_{jl}F'(X_t)(\del_{ij}^{kl})^3(G_{ai}G_{ja})\right]
+\OO\left(\frac{N^{1+C\fd}}{d}\right).
\end{align}
Thanks to Proposition \ref{p:entry-wiselaw}, we have $\max_{i\neq j} |G_{ij}|\prec 1/\sqrt{d}$. 
Those terms from $(\del_{ij}^{kl})^3(G_{ai}G_{ja})$ which contain at least three off-diagonal terms yield a contribution of the form
\begin{align}\label{e:J3est2}
\frac{\ee^{-2t}}{12Nd (d-1)^{2}}\sum_{ijkla}\bE\left[A_{ik}A_{jl}F'(X_t)\{\txt{terms with $\geq 3$ off-diagonal terms}\}\right]
=\OO\left(\frac{N^{4/3+C\fd}}{d^{3/2}}\right).
\end{align}
The leading contribution is from those terms that contain exactly two off-diagonal terms.
By the same estimate as in \eqref{e:stterm}, we have
\begin{align}\begin{split}\label{e:J3est3}
&\phantom{{}={}}\frac{\ee^{-2t}}{12Nd (d-1)^{2}}\sum_{ijkla}\bE\left[A_{ik}A_{jl}F'(X_t)\{\txt{terms with $ 2$ off-diagonal terms}\}\right]\\
&=\frac{\ee^{-t}}{12Nd (d-1)^{2}}\sum_{ijkla}\bE\left[A_{ik}(t)A_{jl}(t)F'(X_t)\{\txt{terms with $2$ off-diagonal terms}\}\right]
+\OO\left(\frac{N^{1+C\fd}}{d^{3/2}}\right).
\end{split}\end{align}
We can estimate above terms using the procedure as described in the proof of Claim \ref{c:J1}, and get
\begin{align}\begin{split}\label{e:J3est4}
&\phantom{{}={}}\frac{\ee^{-t}}{12Nd (d-1)^{2}}\sum_{ijkla}\bE\left[A_{ik}(t)A_{jl}(t)F'(X_t)\{\txt{terms with $2$ off-diagonal terms}\}\right]\\
&=\frac{8}{d}\bE\left[F'(X_t)\Tr G^2\right]+\OO\left(\frac{N^{4/3+C\fd}}{d^{3/2}}\right).
\end{split}\end{align}
The claim \eqref{e:J3} follows from combining \eqref{e:J3est1}, \eqref{e:J3est2}, \eqref{e:J3est3} and \eqref{e:J3est4}.

For fixed $n\geq 4$, we have the trivial bound
\begin{align*}\begin{split}
|J_n|
&\lesssim \frac{1}{Nd^{(n+3)/2}}\sum_{ijkla}\bE[A_{ik}A_{jl}(\del_{ij}^{kl})^n(F'(X_t)G_{ai}G_{ja})]\\
&\lesssim \frac{1}{Nd^{(n+3)/2}}\sum_{ijkl}\bE[A_{ik}A_{jl}\Im[m]/\eta]
=\OO\left(\frac{N^{4/3+C\fd}}{d^{(n-1)/2}}\right). \qedhere
\end{split}\end{align*}
\end{proof}

\bibliography{all}
\bibliographystyle{plain}

\medskip
\subsection*{Acknowledgements}
The work of J.H.\ is supported by the Institute for Advanced Study.
A.K.\ gratefully acknowledges funding from the European Research Council (ERC) under the European Union's Horizon 2020 research and innovation programme (grant agreement No.\ 715539\_RandMat) and from the Swiss National Science Foundation through the SwissMAP grant. 
The work of H.-T.Y.\ is partially supported by NSF Grants DMS-1606305 and DMS-1855509, and a Simons Investigator award.

\end{document}